\documentclass[11pt, twoside, leqno]{article}
\usepackage{enumerate}
\usepackage{graphics}
\usepackage{graphicx}
\usepackage{amssymb}
\usepackage{amsmath}
\usepackage{amsthm}
\usepackage{color}
\usepackage{mathrsfs}
\usepackage{indentfirst}

\usepackage{txfonts}

\allowdisplaybreaks

\def\red{\color{red}}

\def\rr{{\mathbb R}}
\def\rn{{\mathbb{R}^n}}

\def\zz{{\mathbb Z}}

\def\nn{{\mathbb N}}

\def\cm{{\mathcal M}}

\def\fz{\infty }
\def\az{\alpha}

\def\lz{\lambda}

\def\lf{\left}
\def\r{\right}

\def\ls{\lesssim}

\def\noz{\nonumber}

\def\com{\complement}

\def\loc{{\mathop\mathrm{\,loc\,}}}
\def\supp{\mathop\mathrm{\,supp\,}}

\def\XXint#1#2#3{{\setbox0=\hbox{$#1{#2#3}{\int}$ }
		\vcenter{\hbox{$#2#3$ }}\kern-.6\wd0}}

\DeclareMathOperator{\esssup}{ess\,sup}
\DeclareMathOperator{\essinf}{ess\,inf}

\def\ta{\theta}

\def\B{{B}_{p,q}^\az}

\def\sub{\substack}

\def\f{\frac}
\def\vi{\varphi}

\def\({\left(}
\def \){ \right)}
\def\da{\delta}

\def\lz{{\lambda}}

\def\va{\varepsilon}

\def\CD{{\mathcal D}}

\def\BB{{\mathbb B}}

\newcommand{\wt}{\widetilde}

\newcommand{\p}{\partial}

\def\al{\alpha}

\def\RR{\mathbb{R}}
 \def\supp{\operatorname{supp}}

\def\BB{\mathbb{B}}

\def\ld{\lambda}

\def\og{\omega}

\newcommand{\bp}{ \begin{proof} }

\newcommand{\ep}{ \end{proof} }

\def\XXint#1#2#3{{\setbox0=\hbox{$#1{#2#3}{\int}$ }
		\vcenter{\hbox{$#2#3$ }}\kern-.6\wd0}}

\def\cA{\mathcal{A}}

\newtheorem{theorem}{Theorem}[section]
\newtheorem{lemma}[theorem]{Lemma}
\newtheorem{corollary}[theorem]{Corollary}

\newtheorem{proposition}[theorem]{Proposition}

\theoremstyle{definition}
\newtheorem{remark}[theorem]{Remark}
\newtheorem{definition}[theorem]{Definition}
\renewcommand{\appendix}{\par
	\setcounter{section}{0}%
	\setcounter{subsection}{0}%
	\setcounter{subsubsection}{0}%
	\gdef\thesection{\@Alph\c@section}%
	\gdef\thesubsection{\@Alph\c@section.\@arabic\c@subsection}%
	\gdef\theHsection{\@Alph\c@section.}%
	\gdef\theHsubsection{\@Alph\c@section.\@arabic\c@subsection}%
	\csname appendixmore\endcsname
}

\numberwithin{equation}{section}

\begin{document}
	
\title{\bf\Large Brezis--Van Schaftingen--Yung Formulae in Ball Banach Function
Spaces with Applications to Fractional
Sobolev and Gagliardo--Nirenberg Inequalities
	\footnotetext{\hspace{-0.35cm} 2020 {\it
			Mathematics Subject Classification}. Primary 46E35;
		Secondary 26D10, 42B25, 26A33, 35A23.
		\endgraf {\it Key words and phrases.} Sobolev semi-norm,
		Gagliardo semi-norm, ball Banach function space,
		fractional Sobolev inequality, Gagliardo--Nirenberg
		inequality.
		\endgraf The first author is supported by
		NSERC of Canada Discovery grant RGPIN-2020-03909 and
        this project is also supported by the National
		Natural Science Foundation of China (Grant Nos.
		11971058, 12071197 and 12122102)
		and the National Key Research
		and Development Program of China
		(Grant No. 2020YFA0712900).}}
\author{Feng Dai, Xiaosheng Lin, Dachun Yang\footnote{Corresponding
author, E-mail: \texttt{dcyang@bnu.edu.cn}/{\red August 10, 2022}/Final version.},
	\ Wen Yuan and Yangyang Zhang}
\date{}
\maketitle
\vspace{-0.7cm}
\begin{center}
\begin{minipage}{13cm}
{\small {\bf Abstract}\quad
Let $X$ be a ball Banach function space on ${\mathbb R}^n$.
In this article, under some mild assumptions
about both $X$ and the boundedness of the Hardy--Littlewood maximal operator
on the associate space  of the convexification of $X$,
the authors prove that, for
any locally integrable function $f$ with $\|\,|\nabla f|\,\|_{X}<\infty$,
$$\sup_{\lambda\in(0,\infty)}\lambda\left
\|\left|\left\{y\in{\mathbb R}^n:\ |f(\cdot)-f(y)|
>\lambda|\cdot-y|^{\frac{n}{q}+1}\right\}\right|^{\frac{1}{q}}
\right\|_X\sim \|\,|\nabla f|\,\|_X$$
with the positive equivalence constants independent of $f$,
where the index $q\in(0,\infty)$ is related to $X$
and $|\{y\in{\mathbb R}^n:\ |f(\cdot)-f(y)|
>\lambda|\cdot-y|^{\frac{n}{q}+1}\}|$ is the Lebesgue measure of
the set under consideration.
In particular, when $X:=L^p({\mathbb R}^n)$ with $p\in [1,\infty)$, the above formulae
hold true for any given $q\in (0,\infty)$ with $n(\frac{1}{p}-\frac{1}{q})<1$,
which when $q=p$
are exactly the recent surprising formulae
of H. Brezis, J. Van Schaftingen, and P.-L. Yung, and which in other cases
are new. This generalization has a wide
range of applications and, particularly, enables the
authors to establish new fractional Sobolev and new
Gagliardo--Nirenberg inequalities in various function spaces,
including Morrey spaces, mixed-norm Lebesgue spaces,
variable Lebesgue spaces, weighted Lebesgue spaces,
Orlicz spaces, and Orlicz-slice (generalized amalgam) spaces,
and, even in all these special cases,
the obtained results are new. The proofs of these results strongly depend on
the Poincar\'e inequality, the extrapolation,
the exact operator norm on $X'$ of the Hardy--Littlewood maximal operator, and
the exquisite geometry of $\mathbb{R}^n.$
}
\end{minipage}
\end{center}
	
\vspace{0.2cm}
	
\section{Introduction}

It is well known that, for any given $s\in(0,1)$ and $p\in[1,\infty)$,
the \emph{homogeneous fractional
Sobolev space} $\dot{W}^{s,p}(\rn)$ is defined to be the set of all the measurable
functions $f$ on $\rn$ having the following finite Gagliardo semi-norm
\begin{align}\label{wsp}
\|f\|_{\dot{W}^{s,p}(\rn)}:=&\,\left[\int_{\rn}\int_{\rn}\frac{|f(x)-f(y)|^p}
{|x-y|^{n+sp}}\,dx\,dy\right]^{\frac{1}{p}}\\
=&:\left\| \f {f(x)-f(y)} {|x-y|^{\f np+s}}\right\|_{L^p(\rn\times\rn)}.\noz
\end{align}
These spaces play a key role in harmonic analysis
and partial differential equations (see, for instance, \cite{bbm02,crs10,cv11,H08,HT08,m11,mm11,npv}).

A well-known \emph{drawback} of the Gagliardo semi-norm
in \eqref{wsp} is that one can not
recover the homogeneous Sobolev semi-norm $\|\,|\nabla f|\,\|_{L^p(\rn)}$ when $s=1$,
in which case the integral in \eqref{wsp} is infinite unless $f$
is a constant (see \cite{bbm,ref7}). Here and thereafter, for any
differentiable function $f$ on $\rn$, $\nabla f$ denotes the
gradient of $f$, namely,
$$\nabla f:=\left(\frac{\partial f}{\partial x_1},
\cdots,\frac{\partial f}{\partial x_n}\right),$$
and, for any given $p\in[1,\infty)$, the \emph{homogeneous Sobolev space} $\dot{W}^{1,p}(\rn)$
is defined to be the set of all the
locally integrable functions $f$ on $\rn$ having the following finite semi-norm
$$\|f\|_{\dot{W}^{1,p}(\rn)}:=
\|\,|\nabla f|\,\|_{L^p(\rn)}.$$
An important approach to recover $\|\,|\nabla f|\, \|_{L^p(\rn)}$ out of
Gagliardo semi-norms is due to Bourgain et al.
\cite{bbm01} who in particular proved that,
for any given $p\in[1,\infty)$ and for any $f\in \dot{W}^{1,p}(\rn)$,
\begin{align*}
\lim_{s\in(0,1), s\to1}(1-s)\|f\|_{\dot{W}^{s,p}(\rn)}^p=C_{(p,n)}
 \|\,|\nabla f|\,\|_{L^p(\rn)}^p,
\end{align*}
where $C_{(p,n)}$ is a positive constant depending only on both $p$ and $n.$
Very recently, Brezis et al.  \cite{ref8}
discovered an alternative way to repair this \emph{defect}
by replacing the $L^p$ norm in \eqref{wsp} with the weak
$L^p$ quasi-norm, namely, $\|\cdot\|_{L^{p,\infty}(\rn\times\rn)}$.
For any given $p\in[1,\infty),$ Brezis et al. in
\cite{ref8,BSSY21} proved that there exist positive constants
$C_1$ and $C_2$ such that, for any
 $f\in \dot{W}^{1,p}(\rn)$,
\begin{align}\label{bsy}
C_1 \|\,|\nabla f|\,\|_{L^p(\rn)}\leq \left\|\frac{f(x)-f(y)}{|x-y|
^{\frac{n}{p}+1}}\right\|
_{L^{p,\infty}(\rn\times\rn)}\le C_2 \|\,|\nabla f|\,\|_{L^p(\rn)},
\end{align}
where
\begin{align}\label{weak-norm}
&\left\|\frac{f(x)-f(y)}{|x-y|^{\frac{n}{p}+1}}\right\|_{L^{p,\infty}
(\rn\times\rn)}\\
&\quad:=\sup_{\ld\in(0,\infty)}\lambda
\left|\left\{(x,y)\in\rn\times\rn:\
\frac{|f(x)-f(y)|}{|x-y|^{\frac{n}{p}+1}}>
\lambda\right\}\right|^{\frac{1}{p}},\notag
\end{align}
here and thereafter, for any given $m\in\nn$ and any Lebesgue
measurable set $E\subset \RR^m$, the \emph{symbol}
$|E|$ denotes its Lebesgue measure.
The equivalence \eqref{bsy} in
particular allows Brezis et al. in \cite{ref8} to derive
some surprising alternative estimates of fractional Sobolev and
Gagliardo--Nirenberg
inequalities in some exceptional
cases involving $\dot{W}^{1,1}(\rn)$, where
the \emph{anticipated} fractional Sobolev and
Gagliardo--Nirenberg inequalities fail; see also \cite{bm18,bm19}
for more studies on the Gagliardo--Nirenberg
inequality. For later discussions, we
use the Fubini theorem to write the weak $L^p$-norm
in \eqref{weak-norm} and the corresponding Gagliardo
semi-norm in \eqref{wsp}, respectively, as
\begin{align*}
&\left\|\frac{f(x)-f(y)}{|x-y|^{\frac{n}{p}+1}}\right\|
_{L^{p,\infty}(\rn\times\rn)}\\
&\quad=\sup_{\ld\in(0,\infty)}\lambda\left[\int_{\rn}
\left|\left\{y\in\rn:\ |f(x)-f(y)|>\lambda|x-y|
^{\frac{n}{p}+1}\r\}\right|dx\right]^{1/p}
\end{align*}
and
\begin{align*}
\|f\|_{\dot{W}^{s,p}(\rn)}=\left\|\left[\int_{\rn}
\frac{|f(\cdot)-f(y)|^p}{|\cdot-y|^{n+sp}}\,dy \right]
^{\frac 1p}\right\|_{L^p(\rn)}.
\end{align*}
Consequently,
the estimate \eqref{bsy} takes
the following version: for any
 $f\in \dot{W}^{1,p}(\rn)$,
\begin{align}\label{bsy-2}
&\sup_{\ld\in(0,\infty)}\lambda\left[\int_{\rn}
\left|\left\{y\in\rn:\ |f(x)-f(y)|>\lambda|x-y|
^{\frac{n}{p}+1}\r\}\right|dx\right]^{1/p}\\
&\quad\sim \|\,|\nabla f|\,\|_{L^p(\rn)}\notag
\end{align}
with the positive equivalence constants independent of $f$.
More related works can be found in
\cite{bsvy22,bvy21,dlyyz,DMariv,DMariv20,DT19,gy21,bn16}.

Let us also give a few comments on the proof of
\eqref{bsy} in \cite{ref8}. The proof of the
lower bound is relatively simpler. Indeed, a
substantially sharper lower bound was obtained
in \cite{ref8}, using a method of rotation and
the Taylor remainder theorem. On the other hand,
as was pointed out in \cite{ref8}, the stated
upper bound for any given $p\in(1,\infty)$ can be easily
deduced from the following Lusin--Lipschitz inequality in \cite{b91}:
for any differentiable function $f$ and any $x,y\in\rn$,
\begin{equation}\label{Lusin}
|f(x)-f(y)|\lesssim |x-y| \left[ \mathcal{M}
(|\nabla f|)(x)+\mathcal{M} (|\nabla f|)(y)\right],
\end{equation}
where the implicit positive  constant is independent of
$x,y$, and $f$.
Here and thereafter, the \emph{Hardy--Littlewood
maximal operator} $\cm$
is defined by setting, for any $f\in L_{\loc}^1(\rn)$ (the set of all locally
 integrable functions on $\rn$)
and $x\in\rn$,
\begin{equation}\label{2-4-c}
\cm(f)(x):=\sup_{B\ni x}\frac1{|B|}\int_B|f(y)|\,dy,
\end{equation}
where the supremum is taken over all balls $B\subset\rn$
containing $x$. Thus, \emph{the hard core of the proof of
the upper bound in \eqref{bsy} is the case $p=1$}.
The proof in \cite{ref8}, which actually works for
the full range  $p\in[1,\infty)$,
uses the Vitali covering lemma in one variable  and then a
method of rotation. Thus, the rotation invariance
of the space $L^p(\rn)$ seems to play a vital
role in the proof of
\eqref{bsy} in \cite{ref8}.

The main purpose of this article is to give an essential
extension of the main results
[particularly, the equivalence \eqref{bsy-2}] in \cite{ref8,BSSY21}. Such
extensions are fairly nontrivial because our
setting typically involves function spaces that
are neither rotation invariance nor translation invariance.
Somewhat surprisingly, even when returning to the
standard Lebesgue space $L^p(\rn)$, we  have the
following new estimate (see Theorem \ref{thm-6-12} below):
for any given $p\in[1,\infty)$ and $q\in(0,\infty)$ with $n(\frac{1}{p}-\frac{1}{q})<1$, and for any
$f\in \dot{W}^{1,p}(\rn)$,
\begin{align}\label{1-6-1}
&\sup_{\ld\in(0,\infty)}\lambda\left[\int_{\rn}
\lf|\lf\{y\in\rn:\ |f(x)-f(y)|>\lambda|x-y|
^{\frac{n}{q}+1}\r\}\r|^{\frac{p}{q}} dx \right]
^{\frac 1p}\\\noz
&\quad\sim\|\,|\nabla f|\,\|_{L^p(\rn)},
\end{align}
where the positive equivalence
constants are independent of $f$.
In the case of $p=q$, \eqref{1-6-1} is exactly the
surprising estimate \eqref{bsy}
in \cite{ref8,BSSY21}.

Our main result extends the results
\eqref{bsy}
in \cite{ref8,BSSY21} to
a wide class of function spaces on $\rn$, including
Morrey spaces, mixed-norm Lebesgue spaces, variable
Lebesgue spaces, weighted Lebesgue spaces,
Orlicz spaces, and Orlicz-slice spaces (see,
respectively, Subsections \ref{s6.1} through \ref{s6.6}
below for their histories and definitions). We
treat these spaces in a uniform manner in the
setting of ball quasi-Banach function spaces
recently introduced by Sawano et al. \cite{ref3}.
Ball quasi-Banach function spaces are quasi-Banach
spaces of measurable functions on $\rn$ in which
the quasi-norm is related to the Lebesgue measure
on $\rn$ in an appropriate way (see Definition
\ref{Debqfs} below). These function spaces play an
important role in many branches of  analysis.
They are less restrictive than the classical
Banach function spaces introduced in the
book \cite[Chapter 1] {ref4}. For more studies
on ball quasi-Banach function spaces,
we refer the reader to \cite{ref3,s,ref3,yyy20,zhyy21,cwyz20,yhyy22-1,yhyy22-2}
for the Hardy space associated with ball quasi-Banach function spaces,
to \cite{ref5,h21,wyy20} for the boundedness of operators on ball quasi-Banach function spaces,
and to \cite{ins,WYYZ,hcy21,tyyz21} for the applications of ball quasi-Banach function spaces.

To be precise, in this article, our aim is to establish the following
analogue of \eqref{bsy-2} for the
quasi-norm $\|\cdot\|_X$ of a given ball quasi-Banach function
space $X$ under some mild assumptions about both $X$ and
the boundedness of the Hardy--Littlewood maximal operator
on the associate space  of the convexification of $X$
(see Theorems \ref{theorem3.9} and \ref{theorem3} below): for any
$f\in \dot{W}^{1,X}(\rn)$,
\begin{align}\label{1-13}
\sup_{\ld\in(0,\infty)}\lambda\left
\|\left|\lf\{y\in\rn:\ |f(\cdot)-f(y)|>\lambda|\cdot-y|
^{\frac{n}{q}+1}\r\}\right|^{\frac{1}{q}}
\right\|_X\sim \|\,|\nabla f|\,\|_X,
\end{align}
where the index $q\in(0,\infty)$ is related to $X$ under
consideration and the positive equivalence
constants are independent of $f.$
In particular, when returning
to the special case of $X:=L^p(\rn)$
with $p\in[1,\infty)$, we obtain the
formula \eqref{1-6-1}, which  when $p=q$ is just \eqref{bsy} obtained in
\cite{ref8,BSSY21}, and which when $q\in(0,\infty)$ satisfying 
$n(\frac{1}{p}-\frac{1}{q})<1$ seems new. Moreover, we prove that, when $q\in[1,\infty)$,
the condition  $n(\frac{1}{p}-\frac{1}{q})<1$
is \emph{sharp} in some sense [see Remark \ref{trie}(iii) 
below for the details].
Similarly to the case of $X:=L^p(\rn)$ in \cite{ref8,BSSY21},
\eqref{1-13} also allows us to extend
the fractional Sobolev and  the Gagliardo--Nirenberg
inequalities to the setting of ball quasi-Banach function spaces
(see Corollaries \ref{corollary3.111} and \ref{corollary1001} below).

The formula \eqref{1-13}
gives an equivalence
between the Sobolev semi-norm and the quantity involving the
difference of the function under consideration. It is quite
remarkable that such an equivalence holds true
for a ball Banach function
space $X$.
Indeed, finding an appropriate way to
characterize smoothness of functions
via their finite differences is a
notoriously difficult problem in approximation
theory, even for some simple weighted Lebesgue space
in one dimension (see \cite{K15, MT} and the
references therein). A major difficulty
comes from the fact that difference
operators $\Delta_h f:=f(\cdot +h)-f(\cdot)$
for any $h\in\rn$ are no longer bounded
on general weighted $L^p$ spaces.
It turns out that, via using the extrapolation in \cite{ref6} and
the exact operator norm on the associate space of $X$ of the Hardy--Littlewood
maximal operator, the estimate \eqref{1-13} in $X$
follows from the following  estimates in
weighted Lebesgue spaces with Muckenhoupt weights.
\begin{theorem}\label{thm-1-3}
Let $p\in[1,\infty)$ and $\og\in A_{1}(\rn)$.
Then, for any $f\in \dot{W}^{1,p}_\omega(\rn)$,
\begin{equation}\label{1-14a}
  \sup_{\ld\in(0,\infty)} 	\ld^p \int_{\RR^n}
	 \int_{\rn} \mathbf{1}_{_{E_f(\ld,p)}}(x, y)\, dy\,
	 \og(x)\, dx \sim\int_{\RR^n} |
	 \nabla f(x)|^p \og(x) \, dx,
\end{equation}
where, for any $\ld\in(0,\infty),$
	$$E_{f}(\ld,p):=\lf\{ (x, y)\in \RR^n\times \RR^n:\
	|f(x)-f(y)|>\ld |x-y|^{\f np+1} \r\}$$
	and the positive equivalence
	constants are independent of $f.$
\end{theorem}

Indeed, we prove an improved version of Theorem
\ref{thm-1-3} (see Theorem \ref{thm-4-1} below). The
$A_p(\rn)$-condition on the weights $\og$ in Theorem
\ref{thm-1-3} is necessary in some sense
in the case of $n=1$ (see Theorem \ref{thm-4-2} below).
Note that, unlike the integral $\int_{\rn} \int_{\rn}
\cdots \,dx\, dy$ in the estimate \eqref{bsy},
the integral $\int_{\rn}
\int_{\rn} \cdots \,dy \og(x) \,dx$ in \eqref{1-14a} is
not symmetric with respect to both $x$ and $y$, which
causes additional technical difficulties in the proof
of the upper bound in \eqref{1-14a}. Indeed,
the proof of the upper estimate in \eqref{1-14a} is
fairly nontrivial. On one hand, the proof of \eqref{bsy}
in the unweighted case in the article \cite{ref8,BSSY21} is
based on a method of rotation,
and  seems to be inapplicable in the weighted
case here. On the other hand, \emph{using the
Lusin-Lipschitz inequality \eqref{Lusin} would give
the stated upper bound in \eqref{1-14a} for any given $p\in(1,\infty)$,
which is not enough for our purpose because it excludes the endpoint case $p=1$}.
Recall that \emph{the hard core of both the main results in \cite{ref8,BSSY21}
and also Theorem \ref{thm-1-3} is the endpoint case
$p=1$} [see Remark \ref{trie}(ii) for more details]. Instead of the Vitali covering lemma, we use several
adjacent systems of dyadic cubes in $\rn$
and hence the exquisite geometry of $\rn$ (see, for instance, \cite[Section 2.2]{LSU12}) to
overcome these obstacles.

The remainder of this article is organized as follows.

In Section \ref{sec-density}, we first
introduce the homogeneous ball Banach
Sobolev space $\dot{W}^{1,X}(\rn)$ which extend the concept of the homogeneous
Sobolev space $\dot{W}^{1,p}(\rn)$ to the ball
Banach function space. Motivated by
\cite{HK95},
we show that
the set of  all infinitely  differentiable  functions whose
gradients have compact supports is dense in  $\dot{W}^{1,X}(\rn)$
(see Theorem \ref{density} below),
which plays a key role in the proofs of both Theorems \ref{thm-1-3}
and \ref{theorem3.9}.  To prove
Theorem \ref{density},
we first establish the Poincar\'e inequality
on the homogeneous ball Banach
Sobolev space $\dot{W}^{1,X}(\rn)$
(see Proposition \ref{zidong} below).

Section \ref{sec:4} is devoted to the proof of  Theorem
\ref{thm-1-3} which characterizes the Sobolev semi-norm
in weighted Lebesgue spaces. To show Theorem \ref{thm-1-3},
we prove a more general result (see Theorem \ref{thm-4-1}
below), which plays an essential role in the proof of
Theorem \ref{theorem3.9}. First, we establish
the lower
estimate of Theorem \ref{thm-4-1} in ball quasi-Banach
function spaces
(see Theorem \ref{theorem4.88} below),
which is a part of Theorem \ref{theorem3.9}.
In the proof of the upper estimate of Theorem \ref{thm-4-1},
as was aforementioned,
since the  integral in the estimate
\eqref{1-1e} of Theorem \ref{thm-4-1} is not symmetric,
we use several adjacent systems of dyadic cubes in $\rn$
(see Lemma \ref{lem-HK} below) and hence the exquisite
geometry of $\rn$  to overcome this obstacle.
Finally, we prove Theorem \ref{thm-4-2} which
shows that the $A_p(\rn)$-condition on
the weight $\omega$ in Theorem \ref{thm-1-3} is \emph{necessary} in
some sense in the case of $n=1.$

In Section \ref{sec:5}, we generalize \eqref{bsy}
to ball Banach function spaces under some mild assumptions.
However, the calculations in \cite{ref8} need to use the following three crucial
properties of $L^p(\rn)$, which are not available for ball Banach function spaces:
the rotation invariance, the translation invariance, and the explicit expression of the norm.
Borrowing some
ideas from the extrapolation theorem in \cite{ref6}, using Theorem \ref{thm-4-1}
and the exact operator norm on the associate space of $X$ of the Hardy--Littlewood
maximal operator, we  establish the characterization
of the Sobolev semi-norm in ball
Banach function spaces (see both Theorems \ref{theorem3.9} and \ref{theorem3}
below). As applications,
we also establish alternative fractional Sobolev
and alternative Gagliardo--Nirenberg
inequalities in ball Banach function spaces
(see Corollaries \ref{corollary3.111} and \ref{corollary1001} below).

In Section \ref{sec:6}, we apply all these results
obtained in Section \ref{sec:5}, respectively, to
$X:=M_r^\alpha({\mathbb R}^n)$ (the Morrey space),
$X:=L^{p(\cdot)}({\mathbb R}^n)$ (the variable Lebesgue space),
$X:=L^{\vec{p}}(\rn)$ (the mixed-norm Lebesgue space),
$X:=L^p_{\omega}(\rn)$ (the weighted Lebesgue space),
$X:=L^\Phi({\mathbb R}^n)$ (the Orlicz space),
or $X:=(E_\Phi^r)_t({\mathbb R}^n)$
(the Orlicz-slice space or the generalized amalgam space),
all these results are totally new.

Finally, we make some conventions on notation.
Let $\nn:=\{1,2,\ldots\}$ and $\zz_+:=\nn\cup\{0\}$.
 We always denote by $C$ a \emph{positive constant}
which is independent of the main parameters, but it
may vary from line to line. We also use
$C_{(\alpha,\beta,\ldots)}$ to denote a positive
constant depending on the indicated parameters $\alpha,
\beta,\ldots.$ The \emph{symbol} $f\lesssim g$ means
that $f\le Cg$. If $f\lesssim g$ and $g\lesssim f$,
we then write $f\sim g$. If $f\le Cg$ and $g=h$ or
$g\le h$, we then write $f\ls g\sim h$ or $f\ls g\ls h$,
\emph{rather than} $f\ls g=h$ or $f\ls g\le h$.
We use $\mathbf{0}$ to denote the \emph{origin} of $\rn$.
If $E$ is a subset of $\rn$, we denote by $\mathbf{1}_E$
its characteristic function and, for any
measurable set $E\subset\rn$ with $|E|<\infty$
and for any $f\in L^1_{\loc}(\rn)$, let
\begin{align*}
	\fint_Ef(x)\,dx:=
	\f1{|E|}\int_{E}f(x)\,dx=:f_E.
\end{align*}
For any $x\in\rn$ and $r\in(0,\infty),$ let
$B(x,r):=\{y\in\rn:\
|x-y|<r\}$ and
\begin{align}\label{4}
	\mathbb{B}:=\{B(x,r):\ x\in\rn\ \mathrm{and}
	\ r\in(0,\infty)\}.
\end{align}
For any $\alpha\in(0,\infty)$ and any ball
$B:=B(x_B,r_B)$ in $\rn$, with $x_B\in\rn$ and
$r_B\in(0,\infty)$, let $\alpha B:=B(x_B,\alpha r_B)$.
Also, for any
$q\in[1,\infty]$, we denote by $q'$ its
\emph{conjugate exponent}, namely, $1/q+1/q'=1$.
Finally, when we prove a lemma,
proposition, theorem, or corollary, we always
use the same symbols in the wanted proved lemma,
proposition, theorem, or corollary.

\section{Density in $\dot{W}^{1,X}(\rn)$}\label{sec-density}

In this section,
We first extend the concept of homogeneous Sobolev spaces to the
ball Banach function space $\dot{W}^{1,X}(\rn)$. Moreover,
we establish the
Poincar\'e inequality  on
the homogeneous ball Banach Sobolev space $\dot{W}^{1,X}(\rn)$
and
then show that the set of  all infinitely  differentiable
functions whose gradients have compact supports is dense in $\dot{W}^{1,X}(\rn)$
(Theorems \ref{zidong} and \ref{density} below).

First, we recall some preliminaries on ball
quasi-Banach function spaces introduced in \cite{ref3}.
Denote by the \emph{symbol} $\mathscr M(\rn)$ the
set of all measurable functions on $\rn$.
\begin{definition}\label{Debqfs}
	A quasi-Banach space $X\subset{
		\mathscr M}(\rn)$, equipped with
	a quasi-norm
	$\|\cdot\|_X$ which makes sense for
	all functions in ${\mathscr M}(\rn)$,
	is called a \emph{ball quasi-Banach
		function space} if it satisfies that
	\begin{itemize}
		\item[(i)] for any $f\in {\mathscr M}(\rn)$,
		$\|f\|_X=0$ implies that $f=0$ almost everywhere;
		\item[(ii)] for any $f,g\in {\mathscr M}(\rn)$,
		$|g|\le |f|$ almost everywhere implies
		that $\|g\|_X\le\|f\|_X$;
		\item[(iii)] for any $\{f_m\}_{m\in\nn}
		\subset {\mathscr M}(\rn)$
		and $f\in {\mathscr M}(\rn)$, $0\le f_m\uparrow f$
		almost everywhere as $m\to\fz$ implies that
		$\|f_m\|_X\uparrow\|f\|_X$ as $m\to\fz$;
		\item[(iv)] $B\in\BB$ implies that $\mathbf
		{1}_B\in X$, where $\BB$ is the same as in \eqref{4}.
	\end{itemize}
	Moreover, a ball quasi-Banach function space
	$X$ is called a \emph{ball Banach function space}
	if the norm of $X$ satisfies the triangle
	inequality: for any $f,g\in X$,
	\begin{align*}
		\|f+g\|_X\le \|f\|_X+\|g\|_X,
	\end{align*}
	and that, for any $B\in \BB$, there exists a
	positive constant $C_{(B)}$, depending on $B$, such that,
	for any $f\in X$,
	\begin{equation*}\label{eq2.3}
		\int_B|f(x)|\,dx\le C_{(B)}\|f\|_X.
	\end{equation*}
\end{definition}
\begin{remark}
\begin{itemize}
\item[{\rm(i)}] Let $X$ be a ball
quasi-Banach function space on $\rn$.
By \cite[Remark 2.5(i)]{yhyy22-1} (see also \cite[Remark 2.6(i)]{yhyy22-2}),
we conclude that, for any $f\in {\mathscr M}
(\rn)$, $\|f\|_X=0$ if and only if $f=0$
almost everywhere.

\item[{\rm(ii)}] As was mentioned in
\cite[Remark 2.5(ii)]{yhyy22-1} (see also \cite[Remark 2.6(ii)]{yhyy22-2}),
we obtain an equivalent formulation of
Definition \ref{Debqfs}
via replacing any ball $B$ by any bounded
measurable set $E$ therein.

\item[{\rm(iii)}] In Definition \ref{Debqfs},
if we replace any ball $B$ by any measurable
set E with $|E|<\infty$,
we obtain the definition of (quasi-)Banach
function spaces originally introduced in
\cite[Definitions 1.1 and 1.3]{ref4}.
Thus, a (quasi-)Banach function space
is always a ball (quasi-)Banach function space.

\item[{\rm(iv)}] By \cite[Theorem 2]{dfmn21},
we conclude that both (ii) and (iii)
of Definition \ref{Debqfs} imply that
any ball quasi-Banach function space is complete.
\end{itemize}
\end{remark}

We recall the   definition of ball Banach function
spaces with absolutely continuous norm;
see
\cite[Definition 3.1]{ref4} and \cite[Definition 3.2]{wyy20}.

\begin{definition}
	A ball Banach function space $X$ is said to have an
	\emph{absolutely continuous norm} if, for any $f\in X$
	and any sequence of measurable sets  $\{E_j\}_{j\in\nn}\subset \rn$
	satisfying that $\mathbf{1}_{E_j}\to 0$ almost everywhere as $j\to\fz$,
	$\|f\mathbf{1}_{E_j}\|_X\to 0$ as $j\to\fz$.
\end{definition}
We extend the concept of homogeneous Sobolev spaces to ball Banach function spaces.
\begin{definition}
	Let $X$ be a ball Banach function space.
	The \emph{homogeneous ball Banach Sobolev space $\dot{W}^{1,X}(\rn)$}
	is defined to be the set of all the
	distributions $f$ on $\rn$ such that
	$|\nabla f|\in X$ equipped with the quasi-norm
	\begin{equation*}
		\|f\|_{\dot{W}^{1,X}(\rn)}:=\left\|\,|\nabla f|\,\right\|_X,
	\end{equation*}
	where $\nabla f:=(\partial_1f,\cdots,\partial_nf)$
	denotes the distributional gradient of $f$.
\end{definition}
In what follows, for any given $r\in(0,\infty),$ we use $L^r_{\loc}(\rn)$ to
denote the set of all locally $r$-order integrable functions on $\rn$.
For any given $r\in(0,\infty)$, the \emph{centered ball average operator} $B_r$
is defined by setting, for any $f\in L_{\loc}^1(\rn)$ and $x\in\rn$,
\begin{align}\label{pingjun}
	B_r(f)(x):=\frac1{|B(x,r)|}\int_{B(x,r)}|f(y)|\,dy.
\end{align}
In what follows, we denote by the symbol $C_{\mathrm{c}}(\rn)$ [resp., $C^\fz(\rn)$]
the set of all continuous functions with compact support [resp., all
infinitely differentiable functions] on $\rn$.
The following  conclusion is the main result of this section.
\begin{theorem}\label{density}
	Let $X$ be a ball Banach function space. Assume that
	the centered ball average operators $\{B_r\}_{r\in(0,\infty)}$ are uniformly bounded on $X$ and
	$X$ has an absolutely continuous norm.
	Then, for any $f\in \dot{W}^{1,X}(\rn)$, there exists
	a sequence  $\{f_{k}\}_{k\in\nn}\subset C^\infty(\rn)$ with
   $|\nabla f_k|\in C_{\rm{c}}(\rn)$ for any $k\in\nn$ such that
	\begin{align*}
		\lim_{k\to\infty}\|f-f_k\|_{\dot{W}^{1,X}(\rn)}=0\ \ \text{and}\ \
		\lim_{k\to\infty}\|(f-f_k)\mathbf{1}_{B(\mathbf{0},R)}\|_{X}=0
	\end{align*}
	for any  $R\in(0,\infty)$.
\end{theorem}
\begin{remark}
Let $X$ be the same as in Theorem \ref{density}.
Here and thereafter,
the \emph{symbol} $C_{\mathrm{c}}^\fz(\rn)$ denotes
the set of all infinitely differentiable functions on $\rn$
with compact support.
We  point out that,  when $n\in[2,\infty)\cap\nn$,  a
slight modification of the proof
that used in  Theorem \ref{density} shows that
$C^\infty_{\mathrm{c}}(\rn)$ is density in $\dot{W}^{1,X}(\rn)$
which when $X:=L^p(\rn)$ with $p\in[1,\infty)$ is a part of \cite[Theorem 4]{HK95}.
\end{remark}
Recall that, for any given bounded open set $U\subset\rn$,
the \emph{Sobolev space} $W^{1,1}(U)$
is defined to be the set of all the integrable functions $f$ on $U$
such that
$$\|f\|_{W^{1,1}(U)}:=
\|f\|_{L^1(U)}+\|\,|\nabla f|\,\|_{L^1(U)}<\infty,$$
where $\nabla f:=(\p_1f,\cdots,\p_n f)$
is the distributional gradient of $f$.
Moreover, we denote by $W^{1,1}_{\mathrm{loc}}(\rn)$
the set of
all the locally integrable functions $f$ on $\rn$ satisfying that
$f\in W^{1,1}(U)$ for any bounded open set $U\subset\rn$.
To prove Theorem \ref{density},
we need the following conclusion.
\begin{lemma}\label{jubu}
	Let
	$X$ be a ball Banach function space.
	Then
	\begin{align*}
		\dot{W}^{1,X}(\rn)\subset
		W^{1,1}_{\rm{loc}}(\rn).
	\end{align*}
\end{lemma}
\begin{proof}
Let $f\in \dot{W}^{1,X}(\rn)$. Then $|\nabla f|\in X.$
From this and the definition of  $X$, it follows that $|\nabla f|\in L^1_{\loc}(\rn)$. This, together with
\cite[Section 1.1.2]{m11}, implies that $f\in L^1_{\rm{loc}}(\rn)$. This finishes the proof of Lemma \ref{jubu}.
\end{proof}
We now establish the following
Poincar\'e inequality
on the homogeneous ball Banach
Sobolev space $\dot{W}^{1,X}(\rn)$,
which plays a key role in the proof of
Theorem \ref{density}.
\begin{proposition}\label{zidong}
	Let  $x_0\in\rn$, $R\in(0,\infty)$,
	and
	$\Omega:=B(x_0,R)$ when $n\in[1,\infty)\cap\nn$ or $$\Omega:=\{x\in\rn:\ R<|x|<2R\}$$ when $n\in[2,\infty)\cap\nn.$
	Assume that $X$ is a ball Banach function space
	and
	the centered ball average operators $\{B_r\}_{r\in(0,\infty)}$ are uniformly bounded on $X$.
	Then there exists a positive constant $C$, independent of $R$, such that, for any $f\in \dot{W}^{1,X}(\rn)$,
	\begin{align}\label{zhencang}
		\|(f-f_\Omega)\mathbf{1}_{\Omega}\|_{X}\leq C
		R\|\,|\nabla f|\mathbf{1}_{\Omega}\|_{X}.
	\end{align}
\end{proposition}
\begin{remark}
	Let $X$ be the same as in Proposition \ref{zidong} and $f\in\dot{W}^{1,X}(\rn)$. Then, for
	any  ball $B\subset\rn$, $f\mathbf{1}_{B}\in X.$ Indeed, by Lemma \ref{jubu}, we find that
	$f\in L^1_{\rm{loc}}(\rn)$, which, together with \eqref{zhencang}, further implies that, for any ball $B\subset\rn,$
	\begin{align*}
		\|f\mathbf{1}_B\|_{X}
		&\leq
		\|(f-f_B)\mathbf{1}_{B}\|_{X}+\|f_B\mathbf{1}_{B}\|_{X}\\\noz
		&\lesssim
		r_B\|f\|_{\dot{W}^{1,X}(\rn)}
		+\fint_{B}|f(y)|\,dy\|\mathbf{1}_{B}\|_{X}<\infty.
	\end{align*}
\end{remark}
To prove Proposition \ref{zidong}, we need several lemmas.
Let $n\in[2,\infty)\cap\nn.$
Recall that,
for any given measurable function $g$ in $\rn$, the \emph{Riesz potential} $\mathrm{I}_1(g)$ is defined by setting, for any $x\in\rn,$
\begin{align*}
\mathrm{I}_1(g)(x):=\int_{\rn}\frac{g(y)}{|x-y|^{n-1}}\,dy.
\end{align*}
We have the following conclusion.
\begin{lemma}\label{jinbule}
Assume that $X$ is a ball Banach function space and
the centered ball average operators $\{B_r\}_{r\in(0,\infty)}$ are uniformly bounded on $X$.
Let $\Omega$ be a bounded domain in $\rn$ and $n\in[2,\infty)\cap\nn.$
Then there exists a positive constant $C$, independent of  $\Omega$, such that, for any $g\in X$,
\begin{align*}
\|\mathrm{I}_1(|g|\mathbf{1}_{\Omega})\mathbf{1}_\Omega\|_{X}
\leq C
\mathrm{diam}\,(\Omega)
\|g\mathbf{1}_\Omega\|_{X},
\end{align*}
where $\,\mathrm{diam}\,(\Omega):=\sup_{x,y\in\Omega}|x-y|.$
\end{lemma}

\begin{proof}
Let  $D:=2\,\mathrm{diam}\,(\Omega)$. Then, for any $x\in\Omega,$
\begin{align*}
\mathrm{I}_1(|g|\mathbf{1}_{\Omega})(x)&=
\int_{\Omega}\frac{|g(y)|}{|x-y|^{n-1}}\,dy
=
\int_{B(x,D)}\frac{|g(y)\mathbf{1}_{\Omega}(y)|}{|x-y|^{n-1}}\,dy\\\noz
&\leq
\sum_{k=1}^\infty
(2^{-k}D)^{1-n}\int_{2^{-k}D\leq |x-y|<2^{-k+1}D}|g(y)\mathbf{1}_{\Omega}(y)|\,dy\\\noz
&\lesssim
D\sum_{k=1}^\infty
2^{-k}\fint_{B(x,2^{-k+1}D)}|g(y)\mathbf{1}_{\Omega}(y)|\,dy\\
&\sim
D\sum_{k=1}^\infty2^{-k}B_{2^{-k+1}D}(g\mathbf{1}_{\Omega})(x).
\end{align*}
From this and the assumptions that both $X$ is a ball Banach function space and
the operators $\{B_r\}_{r\in(0,\infty)}$ are uniformly bounded on $X$,
it follows that
\begin{align*}
\|\mathrm{I}_1(|g|\mathbf{1}_{\Omega})\mathbf{1}_\Omega\|_{X}
&\lesssim
\mathrm{diam}\,(\Omega)\left\|\sum_{k=1}^\infty2^{-k}B_{2^{-k+1}D}(g\mathbf{1}_{\Omega})\right\|_{X}\\\noz
&\lesssim\mathrm{diam}\,(\Omega)\sum_{k=1}^\infty2^{-k}\left\|B_{2^{-k+1}D}(g\mathbf{1}_{\Omega})\right\|_{X}
\lesssim
\mathrm{diam}\,(\Omega)\|g\mathbf{1}_{\Omega}\|_{X}.
\end{align*}
This finishes the proof of Lemma \ref{jinbule}.
\end{proof}
The following concept of the
associate space of  ball Banach function spaces
can be found in \cite{ref3}; see also, for instance,
\cite[Chapter 1, Definitions 2.1 and 2.3]{ref4} for the corresponding one of Banach function spaces.
\begin{definition}\label{def-X'}
	For any ball Banach function space $X$,
	the \emph{associate space} (also called the
	\emph{K\"othe dual}) $X'$ is defined by setting
	\begin{equation}\label{asso}
		X':=\lf\{f\in\mathscr M(\rn):\ \|f\|_{X'}
		:=\sup_{\{g\in X:\ \|g\|_X=1\}}\|fg\|
		_{L^1(\rn)}<\infty\r\},
	\end{equation}
	where $\|\cdot\|_{X'}$ is called the
	\emph{associate norm} of $\|\cdot\|_X$.
\end{definition}
\begin{remark}\label{bbf}
	By \cite[Proposition 2.3]{ref3},
	we find that, if $X$ is a ball Banach
	function space, then its associate space $X'$ is
	also a ball Banach function space.
\end{remark}
The following H\"older inequality is a direct
corollary of both Definition \ref{Debqfs}(i)
and \eqref{asso} (see \cite[Theorem 2.4]{ref4}).
\begin{lemma}\label{lemma3.5}
	Let $X$ be a ball Banach function space and $X'$ its associate space.
	If $f\in X$ and $g\in X'$, then $fg$ is integrable and
	\begin{equation*}\label{12}
		\int_\rn|f(x)g(x)|\,dx\le \|f\|_X\|g\|_{X'}.
	\end{equation*}
\end{lemma}
\begin{lemma}\label{fushe}
Assume that $X$ is a ball Banach function space and
the centered ball average operators $\{B_r\}_{r\in(0,\infty)}$ are
uniformly bounded on $X$. Then, for any ball $B\subset\rn$,
\begin{align*}
\|\mathbf{1}_{B}\|_{X}\|\mathbf{1}_{B}\|_{X'}\sim|B|
\end{align*}
with the positive equivalence constants depending only on $X$.
\end{lemma}
\begin{proof}
Let $B$ be a ball in $\rn$.
Notice that, for any $g\in L^1_{\rm{loc}}(\rn)$ and $x\in\rn,$
\begin{align*}
\frac{1}{|B|}\int_{B}|g(y)|\,dy\mathbf{1}_{B}(x)
&\leq
\frac{2^n}{|B(x,2r_B)|}\int_{B(x,2r_B)}|g(y)|\,dy\mathbf{1}_{B}(x)\\
&=2^nB_{2r_B}(g)(x)\mathbf{1}_{B}(x),
\end{align*}
where $r_B$ denotes the radius of $B$.
From this, the definition of $X'$,
and the assumption that  the centered ball average operators
$\{B_r\}_{r\in(0,\infty)}$ are uniformly bounded on $X$, it follows that
\begin{align*}
\frac{\|\mathbf{1}_{B}\|_{X}\|\mathbf{1}_{B}\|_{X'}}{|B|}
\leq
\sup_{\|g\|_{X}=1}\frac{1}{|B|}\int_{B}|g(y)|\,dy\|\mathbf{1}_{B}\|_{X}
\leq
\sup_{\|g\|_{X}=1}2^n\|B_{2r_B}(g)\mathbf{1}_{B}\|_{X}\lesssim 1.
\end{align*}
On the other hand, by Lemma \ref{lemma3.5}, we obtain
\begin{align*}
|B|=\int_{\rn}\mathbf{1}_{B}\,dx
\leq
\|\mathbf{1}_{B}\|_{X}\|\mathbf{1}_{B}\|_{X'}.
\end{align*}
This
finishes the proof of Lemma \ref{fushe}.
\end{proof}
Now, we show Proposition \ref{zidong}.
\begin{proof}[Proof of Proposition \ref{zidong}]
Let $f\in\dot{W}^{1,X}(\rn)$.
By Lemma \ref{jubu}, we find that
$f\in W^{1,1}_{\rm{loc}}(\rn)$.	
Next, we show Proposition \ref{zidong} by considering the following two cases on $n.$

Case 1) $n=1.$
In this case, $\Omega=B(x_0,R)$ for some $x_0\in\rr$ and $R\in(0,\infty)$.
We first show that, for almost every $x\in\Omega$,
\begin{align}\label{2022717b}
|f(x)-f_\Omega|\leq \int_{\Omega}|f'(y)|\,dy.
\end{align}
Let $\eta\in C^\infty_{\mathrm{c}}(\rr)$ satisfy
$\supp (\eta)\subset(-1,1)$ and $\int_{\rr}\eta(x)\,dx=1.$
For any $k\in\nn$, let $\eta_k(\cdot):=k\eta(k\cdot)$ and
$f_k(\cdot):=(f\ast \eta_k)(\cdot)$.
By both  (i) and (vi) of
\cite[Theorem 4.1]{eg15}, we find that $f_k\in C^\infty(\rr)$ for any $k\in\nn$,
\begin{align}\label{2022717}
\lim_{k\to\infty}\|f-f_k\|_{L^1(\Omega)}=0,
\ \ \text{and}\ \
\lim_{k\to\infty}\|f'-f_k'\|_{L^1(\Omega)}=0.
\end{align}
Without loss of generality, we may assume that $\lim_{k\to\infty}f_k(x)=f(x)$ for almost every $x\in\Omega.$ Since
$f_k\in C^\infty(\rr)$ for any $k\in\nn$, it follows that, for any $k\in\nn$ and $x,y\in\Omega,$
\begin{align*}
|f_k(x)-f_k(y)|
=
\left|\int_{y}^xf'_k(t)\,dt\right|
\leq
\int_{\Omega}|f'_k(t)|\,dt
\end{align*}
and hence
\begin{align*}
\left|f_k(x)-\frac{1}{|\Omega|}\int_{\Omega}f_k(y)\,dy\right|
\leq
\frac{1}{|\Omega|}\int_{\Omega}|f_k(x)-f_k(y)|\,dy
\leq
\int_{\Omega}|f'_k(t)|\,dt.
\end{align*}
Using this and \eqref{2022717}, and
letting $k\to\infty$,  we conclude that
\eqref{2022717b} holds true.

By \eqref{2022717b} and Lemmas \ref{lemma3.5} and \ref{fushe}, we obtain
\begin{align*}
\|(f-f_\Omega)\mathbf{1}_{\Omega}\|_{X}
&\leq
\|\mathbf{1}_{\Omega}\|_{X}\int_\Omega|f'(y)|\,dy
\leq
\|\mathbf{1}_{\Omega}\|_{X}
\|\mathbf{1}_{\Omega}\|_{X'}\|f'\mathbf{1}_{\Omega}\|_{X}\\\noz
&\lesssim
R\|f'\mathbf{1}_{\Omega}\|_{X}.
\end{align*}
This proves \eqref{zhencang} in this case.

Case 2) $n\in[2,\infty)\cap\nn.$
In this case, from \cite[Lemma 8.2.1(b)]{DHHR}, we infer that there
exists a ball $B\subset \Omega$ and positive constants $c_1,c_2$ such that,
for almost every $x\in\Omega$,
\begin{align}\label{qingxing}
|f(x)-f_{B}|\leq c_1 \mathrm{I}_1(|\nabla f|\mathbf{1}_{\Omega})(x)
\end{align}
and
\begin{align}\label{qingxing2}
|B|\leq |\Omega|\leq c_2 |B|,
\end{align}
where the positive constants $c_1$ and $c_2$ depend only on $n$.
Using \eqref{qingxing} and Lemma \ref{jinbule}, we further obtain
\begin{align}\label{bkb0}
\|(f-f_B)\mathbf{1}_{\Omega}\|_{X}
\lesssim
\|\mathrm{I}_1(|\nabla f|\mathbf{1}_{\Omega})\mathbf{1}_{\Omega}\|_{X}
\lesssim
R\|\,|\nabla f|\mathbf{1}_{\Omega}\|_{X}.
\end{align}

On the other hand,
by \cite[Theorem 8.2.4(b)]{DHHR},
we have the following classical Poincar\'e inequality
\begin{align*}
\|f-f_{\Omega}\|_{L^1(\Omega)}
\lesssim
R\|\,|\nabla f|\,\|_{L^1(\Omega)},
\end{align*}
where the implicit positive constant
depends only on $n$.
From this,
Lemmas \ref{lemma3.5} and
\ref{fushe}, and \eqref{qingxing2},
it follows that
\begin{align*}
\|(f_{B}-f_{\Omega})\mathbf{1}_{\Omega}\|_{X}
&\leq
\frac{\|\mathbf{1}_{\Omega}\|_{X}}{|B|}\int_{B}|f(x)-f_\Omega|\,dx
\lesssim
\frac{\|\mathbf{1}_{\Omega}\|_{X}}{|\Omega|}\int_{\Omega}|f(x)-f_\Omega|\,dx\\
&\lesssim
\frac{R\|\mathbf{1}_{\Omega}\|_{X}}{|\Omega|}\int_{\Omega}|\nabla f(x)|\,dx
\lesssim
R\frac{\|\mathbf{1}_{\Omega}\|_{X}\|\mathbf{1}_{\Omega}\|_{X'}}{|\Omega|}\|\,|\nabla f|\mathbf{1}_{\Omega}\|_{X}\\
&\lesssim
R\frac{\|\mathbf{1}_{B(\mathbf{0},2R)}\|_{X}\|\mathbf{1}_{B(\mathbf{0},2R)}\|_{X'}}{|B(\mathbf{0},2R)|}\|\,|\nabla f|\mathbf{1}_{\Omega}\|_{X}
\lesssim
R \|\,|\nabla f|\mathbf{1}_{\Omega}\|_{X}.
\end{align*}
This, together with \eqref{bkb0}
and the assumption that $X$ is a ball
Banach function space, further implies that
\eqref{zhencang} holds true. This finishes
 the proof of Proposition \ref{zidong}.
\end{proof}

To prove Theorem \ref{density}, we need the following conclusion.
\begin{proposition}\label{nuanuan}
Let $X$ be a ball Banach function space. Assume that
the centered ball average operators $\{B_r\}_{r\in(0,\infty)}$ are uniformly bounded on $X$ and
$X$ has an absolutely continuous norm.
Then,
for any $f\in \dot{W}^{1,X}(\rn)$,
there exists a sequence $\{f_k\}_{k\in\nn}\subset C^\infty(\rn)\cap \dot{W}^{1,X}(\rn)$
such that
\begin{align*}
\lim_{k\to\infty}\|f-f_k\|_{\dot{W}^{1,X}(\rn)}=0
\ \ \text{and}\ \
\lim_{k\to\infty}\|(f-f_k)\mathbf{1}_{B(\mathbf{0},R)}\|_{X}=0
\end{align*}
for  any $R\in(0,\infty)$.
\end{proposition}
\begin{proof}
Let $f\in \dot{W}^{1,X}(\rn)$.
It is well known that there exists a function $\eta \in C^\infty_{\mathrm{c}}(\rn)$  such that
$\supp\,(\eta) \subset B(\mathbf{0},1)$
and
$\int_\rn \eta(x)\,dx=1$. For any $k\in \nn$, let  $\eta_k(\cdot):=k^n\eta(k\cdot)$ and
$f_k(\cdot):=(f\ast\eta_k)(\cdot)$.
Next, we show that
\begin{align}\label{jintou}
\lim_{k\to\infty}\|f-f_k\|_{\dot{W}^{1,X}(\rn)}=0.
\end{align}

By Lemma \ref{jubu}, we find that $f\in W^{1,1}_{\rm{loc}}(\rn)$.
Using this and both (i) and (v) of
\cite[Theorem 4.1]{eg15},
we find that, for any $k\in\nn$ and $i\in\{1,\ldots,n\}$,
$f_k\in C^\infty(\rn)$ and
\begin{align}\label{gu2}
\frac{\p f_k}{\p x_i}=\frac{\p f}{\p x_i}\ast\eta_k.
\end{align}
Let $i\in\{1,\ldots,n\}$.
Since $X$ has an absolutely continuous norm,
from \cite[Proposition 3.8]{tyyz21},
it follows that $C_{\mathrm{c}}(\rn)$ is dense in $X$. Thus,  for any $\epsilon\in(0,\infty)$, there exists a
function $g\in C_{\mathrm{c}}(\rn)$ such that
\begin{align}\label{nageren}
\left\|\frac{\p f}{\p x_i}-g\right\|_{X}\leq \epsilon.
\end{align}
By this, \eqref{gu2}, and the assumption that $X$ is a ball Banach function space, we have
\begin{align}\label{gu}
\left\|\frac{\p f}{\p x_i}-\frac{\p f_k}{\p x_i}\right\|_{X}
\leq
\left\|\frac{\p f}{\p x_i}-g\right\|_{X}
+
\|g-g\ast\eta_k\|_{X}
+
\left\|g\ast\eta_k-\frac{\p f}{\p x_i}\ast \eta_k\right\|_{X}.
\end{align}
Assume that $\supp\,(g)\subset B(\mathbf{0},N)$ with some $N\in(0,\infty)$. Then we have
\begin{align}\label{gu4}
\|g-g\ast \eta_k\|_{X}
&\leq
\left\|k^{n}\int_{\rn}|g(\cdot)-g(y)|\eta(k[\cdot-y])\,dy\right\|_{X}\\\noz
&\lesssim
\left\|\,
\fint_{B(\cdot,k^{-1})}|g(\cdot)-g(y)|\,dy\right\|_{X}\\\noz
&\lesssim
\sup_{|x-y|\leq k^{-1}}|g(x)-g(y)|\,\|\mathbf{1}_{B(\mathbf{0},N+1)}\|_{X}\to0
\end{align}
as $k\to\infty.$ On the other hand, using
\eqref{nageren}
and the assumption that  the operators $\{B_r\}_{r\in(0,\infty)}$ are
uniformly bounded on $X$, we conclude that, for any $k\in\nn,$
\begin{align*}
\left\|g\ast\eta_k-\frac{\p f}{\p x_i}\ast \eta_k\right\|_{X}
&\leq
\left\|k^n\int_{\rn}\left|g(\cdot-y)-\frac{\p f}{\p x_i}(\cdot-y)\right|\eta(ky)\,dy\right\|_{X}\\\noz
&=
\left\|\int_{\rn}\left|g(\cdot-y/k)-\frac{\p f}{\p x_i}(\cdot-y/k)\right|\eta(y)\,dy\right\|_{X}\\\noz
&\lesssim
\left\|\,\fint_{B(\cdot,k^{-1})}\left|g(y)-\frac{\p f}{\p x_i}(y)\right|\,dy
\right\|_{X}\\\noz
&\sim
\left\|B_{k^{-1}}\left(g-\frac{\p f}{\p x_i}\right)\right\|_{X}
\lesssim
\left\|g-\frac{\p f}{\p x_i}\right\|_{X}
\lesssim
\epsilon.
\end{align*}
From this, \eqref{nageren}, \eqref{gu4}, and \eqref{gu}, it follow that, for any $\epsilon\in(0,\infty)$,
\begin{align*}
\limsup_{k\to\infty}\left\|\frac{\p f}{\p x_i}-\frac{\p f_k}{\p x_i}\right\|_{X}
\lesssim \epsilon.
\end{align*}
Let $\epsilon\to0$. Then we have, for any $i\in\{1,\ldots,n\}$,
\begin{align*}
\lim_{k\to\infty}\left\|\frac{\p f}{\p x_i}-\frac{\p f_k}{\p x_i}\right\|_{X}=0.
\end{align*}
This implies that \eqref{jintou} holds true.

Finally, we show that
\begin{align}\label{gala}
\lim_{k\to\infty}\|(f-f_k)\mathbf{1}_{B(\mathbf{0},R)}\|_{X}=0
\end{align}
for any $R\in(0,\infty)$.
By \cite[Theorem 4.1(iii)]{eg15}, we find that, for any $R\in(0,\infty)$,
\begin{align*}
\lim_{k\to\infty}\int_{B(\mathbf{0},R)}
|f(y)-f_k(y)|\,dy=0.
\end{align*}
Using this, the assumption that $X$ is a ball Banach function space,
Proposition \ref{zidong}, and \eqref{jintou}, we obtain, for any $R\in(0,\infty)$,
\begin{align*}
&\|(f-f_k)\mathbf{1}_{B(\mathbf{0},R)}\|_{X}\\\noz
&\quad\leq
\left\|\left[f-f_k-\frac{1}{|B(\mathbf{0},R)|}\int_{B(\mathbf{0},R)}
|f(y)-f_k(y)|\,dy\right]\mathbf{1}_{B(\mathbf{0},R)}\right\|_{X}\\\noz
&\quad\quad+\frac{1}{|B(\mathbf{0},R)|}\int_{B(\mathbf{0},R)}
|f(y)-f_k(y)|\,dy\|\mathbf{1}_{B(\mathbf{0},R)}\|_{X}\\\noz
&\quad\lesssim
R\|\,|\nabla f-\nabla f_k|\,\|_{X}
+\frac{1}{|B(\mathbf{0},R)|}\int_{B(\mathbf{0},R)}|f(y)-f_k(y)|\,dy\|\mathbf{1}_{B(\mathbf{0},R)}\|_{X}\to0
\end{align*}
as $k\to\infty$. This proves \eqref{gala} and hence finishes the proof of Theorem \ref{nuanuan}.
\end{proof}
Now, we show Theorem \ref{density}.

\begin{proof}[Proof of Theorem \ref{density}]
Let  $f\in \dot{W}^{1,X}(\rn)$.
By Proposition \ref{nuanuan},
we can assume that $f\in C^\infty(\rn) \cap\dot{W}^{1,X}(\rn)$.
Let $\varphi\in C^\infty_{\mathrm{c}}(\rn)$ be such that $0\leq\varphi\leq1$,
$\supp\, (\varphi)\subset B(\mathbf{0},2)$, and
$\varphi\equiv1$ in $B(\mathbf{0},1)$.
For any given $k\in\nn$ and any $x\in\rn$, let $\eta_k(x):=\varphi(x/k)$.
Now, we prove Theorem \ref{density} by considering the following two cases on $n.$

Case 1) $n=1.$ In this case,
let $f_k(x):=\int_0^xf'(t)\eta_k(t)\,dt+f(0)$
for any $x\in\rr$ and  any given $k\in\nn.$
By the assumption that
$X$ has an absolutely continuous norm, we have
\begin{align*}
\lim_{k\to\infty}\left\|f_k'-f'\right\|_X=
\lim_{k\to\infty}
\left\|f'\eta_k-f'\right\|_X=0.
\end{align*}
Let $R\in(0,\infty)$. Notice that, for any $k\in\nn\cap[R,\infty)$ and $x\in B(\mathbf{0},R)$,
$$f_{k}(x)=\int_0^xf'(t)\,dt+f(0)=f(x)$$
and hence
\begin{align*}
\lim_{k\to\infty}\|(f-f_k)\mathbf{1}_{B(\mathbf{0},R)}\|_{X}=0.
\end{align*}
This finishes the proof of Theorem \ref{density} in this case.

Case 2)  $ n\in[2,\infty)\cap\nn$.
In this case, for any $k\in\nn$, let $\Omega_k:=\{x\in\rn:\ k<|x|<2k\}.$
By Proposition \ref{zidong}, we obtain, for any $k\in\nn,$
\begin{align}\label{14533}
\|(f-f_{\Omega_k})\mathbf{1}_{\Omega_k}\|_{X}
\lesssim
k\|\,|\nabla f|\mathbf{1}_{\Omega_k}\|_{X}.
\end{align}
For any $k\in\nn$ and $x\in\rn$, let $f_k(x):=(f(x)-f_{\Omega_k})\eta_k(x)+f_{\Omega_k}$. Now, we show that
$f_k\to f$ in $\dot{W}^{1,X}(\rn).$
Notice that, for any $j\in\{1,\ldots n\},$
\begin{align}\label{tianheii}
\frac{\p(f_k-f)}{\p x_j}=\frac{\p f}{\p x_j}(\eta_k-1)+(f-f_{\Omega_k})\frac{\p \eta_k}{\p x_j}.
\end{align}
Since
$X$ has an absolutely continuous norm, we deduce that
\begin{align}\label{wuwo}
\lim_{k\to\infty}\left\|\frac{\p f}{\p x_j}(\eta_k-1)\right\|_{X}=0.
\end{align}
Observe that, for any $k\in\nn$ and $x\in\rn,$
\begin{align*}
\left|\frac{\p \eta_k}{\p x_j}(x)\right|
=
\frac{1}{k}\left|\frac{\p \varphi}{\p x_j}\left(\frac{x}{k}\right)\right|
\leq \frac{1}{k}\|\,|\nabla \varphi|\,\|_{L^\infty(\rn)}.
\end{align*}
From this, \eqref{14533}, and the assumption that
$X$ has an absolutely continuous norm,   it follows that
\begin{align*}
\left\|(f-f_{\Omega_k}) \frac{\p \eta_k}{\p x_j}\right\|_{X}
&=
\left\|(f-f_{\Omega_k}) \frac{\p \eta_k}{\p x_j}\mathbf{1}_{\Omega_k}\right\|_{X}\\
&\lesssim
\|\,|\nabla f|\mathbf{1}_{\Omega_k}\|_{X}
\lesssim
\|\,|\nabla f|\mathbf{1}_{[B(\mathbf{0},k)]^\com}\|_{X}\to0
\end{align*}
as $k\to\infty.$  This, together with both \eqref{wuwo} and \eqref{tianheii},
further implies that, for any $j\in\{1,\ldots,n\},$
\begin{align*}
\lim_{k\to\infty}\left\|\frac{\p(f_k-f)}{\p x_j}\right\|_{X}=0
\end{align*}
and hence
\begin{align*}
\lim_{k\to\infty}\|f-f_k\|_{\dot{W}^{1,X}(\rn)}=0.
\end{align*}
Let $R\in(0,\infty)$. Observe that,
for any $k\in\nn\cap [R,\infty)$ and
$x\in B(\mathbf{0},R)$,
\begin{align*}
f_k(x)=f(x)-f_{\Omega_k}+f_{\Omega_k}=f(x)
\end{align*}
and hence
\begin{align*}
	\lim_{k\to\infty}\|(f-f_k)\mathbf{1}_{B(\mathbf{0},R)}\|_{X}=0.
\end{align*}
This then finishes the proof of Theorem \ref{density}.
\end{proof}
The the following definition of the convexification of a ball Banach function
space can be found in \cite[Definition 2.6]{ref3}.
\begin{definition}\label{Debf}
	Assume that $X$ is a ball quasi-Banach function space
	and $p\in(0,\infty)$. The $p$-\emph{convexification}
	$X^p$ of $X$ is defined by setting $X^p:=\{f
	\in\mathscr M(\rn):\ |f|^p\in X\}$
	equipped with the quasi-norm $\|f\|_{X^p}:=\|\,|f|^p\|_X^{\frac{1}{p}}$ for any $f\in X^p$.
\end{definition}
The following lemma gives a sufficient condition for
the uniform boundedness of centered ball average operators
$\{B_r\}_{r\in(0,\infty)}$ on $X$,
which is just \cite[Lemma 3.11]{dgpyyz}.
\begin{lemma}\label{yongwu}
Let $X$ be a ball Banach function space and $p\in[1,\infty)$.
Assume that $X^{\f 1 p}$ is a ball Banach function space
and the Hardy--Littlewood maximal
operator $\cm$ is bounded on $(X^{1/p})'$.
Then the centered ball average operators $\{B_r\}_{r\in(0,\infty)}$
are uniformly bounded on $X$.
\end{lemma}
Using Lemma \ref{yongwu} and
Theorem \ref{density}, we immediately obtain the following conclusion.
\begin{corollary}\label{density2}
	Let $X$ be a ball Banach function space and $p\in[1,\infty)$.
	Assume that $X^{\f 1 p}$ is a ball Banach function space,
	 the Hardy--Littlewood maximal
	operator $\cm$ is bounded on $(X^{1/p})'$, and
	$X$ has an absolutely continuous norm.
	Then, for any $f\in \dot{W}^{1,X}(\rn)$, there exists
	a sequence  $\{f_{k}\}_{k\in\nn}\subset C^\infty(\rn)$ with
$|\nabla f_k|\in C_{\rm{c}}(\rn)$ for any $k\in\nn$ such that
	\begin{align*}
		\lim_{k\to\infty}\|f-f_k\|_{\dot{W}^{1,X}(\rn)}=0\ \ \text{and}\ \
		\lim_{k\to\infty}\|(f-f_k)\mathbf{1}_{B(\mathbf{0},R)}\|_{X}=0
	\end{align*}
	for any  $R\in(0,\infty)$.
\end{corollary}
\section{Estimates in Weighted Lebesgue Spaces}\label{sec:4}
In this section, we establish the characterization
of the Sobolev semi-norm in the weighted Lebesgue space
(see Theorem \ref{thm-4-1} below), which is just
Theorem \ref{thm-1-3} when $p=q$. We should point out
that
Theorem \ref{thm-4-1} plays a vital role in the
proof of Theorem \ref{theorem3.9} below.
Moreover, we show that the $A_p(\rn)$ condition
in Theorem \ref{thm-1-3} is sharp in some sense
(see Theorem \ref{thm-4-2} below).

We first recall the concept of
Muckenhoupt weights $A_p(\rn)$ (see, for
instance, \cite{ref1}).
\begin{definition}\label{weight}
	An \emph{$A_p(\rn)$-weight} $\omega$, with
	$p\in[1,\infty)$, is a nonnegative locally
	integrable function on $\rn$ satisfying
	that, when $p\in(1,\infty),$
	$$[\omega]_{A_p(\rn)}:=\sup_{Q \subset\rn}
	\left[\frac{1}{|Q|}\int_{Q}\omega(x)\,dx
	\right]\left\{\frac{1}{|Q|}\int_Q[\omega(x)]^
	{\frac{1}{1-p}}\,dx\right\}^{p-1}<\infty,$$
	and
	$$[\omega]_{A_1(\rn)}:=\sup_{Q\subset\rn}
	\frac{1}{|Q|}\int_Q\omega(x)\,dx\left
	[\|\omega^{-1}\|_{L^\infty(Q)}\right]<\infty,$$
	where the suprema are taken over all
	cubes $Q\subset\rn$.
	
	Moreover, let
	$A_{\infty}(\rn):=\bigcup_{p\in[1,\infty)}A_p(\rn).$
\end{definition}
\begin{definition}\label{twl}
	Let $p\in[0,\infty)$ and $\omega\in A_{\infty}
	(\rn).$ The \emph{weighted Lebesgue space}
	$L^p_{\omega}(\rn)$ is defined to be the
	set of all the measurable functions $f$ on $\rn$
	such that
	$$\|f\|_{L^p_{\omega}(\mathbb{R}^n)}:=\left
	[\int_{\mathbb{R}^n}|f(x)|^p\omega(x)\,dx
	\right]^{\frac{1}{p}}<\infty.$$
\end{definition}
The following lemma is a part of
\cite[Proposition 7.1.5]{ref1}.
\begin{lemma}\label{Lemma2.1}
	Let $p\in[1,\infty)$ and $\omega\in A_p(\rn).$ Then the following statements hold true.
	\begin{itemize}
		\item[{\rm(i)}] For any $\lambda
		\in(1,\infty)$ and any cube $Q\subset\rn$, one has
		$\omega(\lambda Q)\leq [\omega]_{A_p
			(\rn)}\lambda^{np}\omega(Q);$
		\item[{\rm(ii)}]
		$$[\omega]_{A_p(\rn)}=\sup_{Q \subset\rn}
		\sup_{\substack{f\mathbf{1}_Q\in L^p_\omega(\rn)\\
				\int_Q|f(t)|^p\omega(t)\,dt\in(0,\infty)}}\frac
		{[\frac{1}{|Q|}\int_Q|f(t)|\,dt]^p}
		{\frac{1}{\omega(Q)}\int_Q|f(t)|^p\omega(t)\,dt},$$
		where the first supremum is taken over all cubes $Q\subset\rn$.
	\end{itemize}
\end{lemma}

\begin{definition}\label{202276}
Let $p\in[1,\infty]$ and $\omega\in A_\infty(\rn)$. The
\emph{homogeneous weighted Sobolev space} $\dot{W}^{1,p}_\omega(\rn)$
is defined to be the set of all the
distributions $f$ on $\rn$ whose distributional gradients
$\nabla f:=(\p_1 f,\cdots,\p_n f)$ satisfy $|\nabla f|\in L^p_\omega(\rn)$
and, moreover, for any $f\in \dot{W}^{1,p}_\omega(\rn)$, let
\begin{align*}
\|f\|_{\dot{W}^{1,p}_\omega(\rn)}:=\|\,|\nabla f|\,\|_{L^p_\omega(\rn)}.
\end{align*}
\end{definition}
The following conclusion is the main result of this section.
\begin{theorem}\label{thm-4-1}
Let $p\in[1,\infty)$
		and $q\in(0,\infty)$ satisfy $n(\f 1p-\f1q) <1$.
		Assume that $\og\in A_1(\rn)$. Then there exist
		positive constants $C_1$ and $C_{([\omega]
		_{A_1(\rn)})}$ such that, for any $f\in \dot{W}^{1,p}_\omega(\rn)$,	
	\begin{align}\label{1-1e}
		&\left[\frac
		{K(q,n)}{n}\right]^{\frac{p}{q}}\int_{\RR^n} | \nabla f(x)|^p \og(x) \, dx\\\noz
		&\quad\leq\sup_{\ld\in(0,\infty)} \ld^p \int_
		{\RR^n}\left[\int_{ \RR^n} \mathbf{1}_
		{E_f(\ld,q)}(x, y)\, dy \right]^{\f pq}
		\og(x) \,dx \\\noz
		&\quad\leq C_1C_{([\omega]_{A_1(\rn)})}	
		\int_{\RR^n} | \nabla f(x)|^p \og(x) \, dx,
	\end{align}
	where, for any $\ld\in(0,\infty),$ $q\in(0,\infty),$ and any measurable function $f,$
	\begin{align}\label{E}
		E_f(\ld,q) :=\lf\{ (x, y)\in \RR^n\times \RR^n:
		\ |f(x)-f(y)|> \ld |x-y|^{\f nq+1} \r\},
	\end{align}
the positive constant $C_1$  depends only on $p,q$
and $n$, the positive constant
$C_{([\omega]_{A_1(\rn)})}$ increases as $[\omega]_
{A_1(\rn)}$ increases, and $C_{(\cdot)}$ is
continuous on $(0,\infty)$. Moreover, for any $f\in \dot{W}^{1,p}_\omega(\rn)$,
\begin{align*}
&\lim_{\ld\to\infty}\ld^p \int_{\RR^n}\left
[\int_{ \RR^n} \mathbf{1}_{E_f(\ld,q)}(x, y)\,
dy \right]^{\f pq} \og(x)\, dx\\
&\quad=\left[\frac
{K(q,n)}{n}\right]^{\frac{p}{q}}\int_{\RR^n}
| \nabla f(x)|^p \og(x) \, dx,
\end{align*}
where
\begin{align}\label{kqn}
K(q,n):=\int_{\mathbb{S}^{n-1}}|\xi\cdot
e|^q\,d\sigma(\xi)
\end{align}
with $e$ being some unit vector in $\rn.$
\end{theorem}

\begin{remark}\label{trie}
\begin{itemize}
\item[{\rm(i)}] As a consequence of Theorem
\ref{thm-4-1}, we obtain Theorem \ref{thm-1-3}.

\item[{\rm(ii)}]
For any given $p\in(1,\infty)$,
the upper estimate of \eqref{1-1e}
can be easily deduced from
both the Lusin--Lipschitz inequality \eqref{Lusin}
and the boundedness of the Hardy--Littlewood maximal
operator on $L^p(\rn)$. Thus, since the Hardy--Littlewood maximal
operator is not bounded on $L^1(\rn)$,
the \emph{hard core} of the proof
of \eqref{1-1e} is to include the case $p=1.$

\item[{\rm(iii)}]
Let $p,q\in[1,\infty)$ satisfy   $n\max\{0,\f 1p-\f1q\}<s<1$.
By \cite[Theorem 1.3]{Hov21},
we find that, if
$f\in
L_{\loc}^{\min\{p,q\}}(\rn)$,
then $f\in F_{p,q}^s(\rn)$
if and only if
\begin{align*}
\mathrm{I}:=\|f\|_{L^p(\rn)}+\left\|\left[\int_{\rn}\frac{|f(\cdot)-
f(y)|^q}{|\cdot-y|^{n+sq}}\,dy\right]
^{\frac{1}{q}}\right\|
_{L^p(\mathbb{R}^n)}
<\infty
\end{align*}
and, moreover, in this case,
\begin{align}\label{smp}
\mathrm{I}\sim\|f\|_{F_{p,q}^s(\rn)}
\end{align}
with the positive equivalence constants independent of $f$,
where $F_{p,q}^s(\rn)$ denotes
the classical Triebel--Lizorkin space
(see \cite[Section 2.3]{T83} for the precise definition).
Using \eqref{smp}, we find that,
for any given $p\in[1,\infty)$
and $s\in(0,1)$, $F^s_{p,p}(\rn)=W^{s,p}(\rn)$
with equivalent norms.
Also, by \cite[Theorem 1.5]{Hov21},  we conclude that,
under the assumption that $p,q\in[1,\infty)$,
\eqref{smp} is valid \emph{only if} $n\max\{0,\f 1p-\f1q\}\leq s<1$.
Moreover, using both (i) and (ii)
of Theorem \ref{ss1} below,
we conclude that, when $s=1$,
$p\in[1,\infty),$ and $q\in[1,p],$
$$
\left\|\left[\int_{\rn}\frac{|f(\cdot)-
	f(y)|^q}{|\cdot-y|^{n+q}}\,dy\right]
^{\frac{1}{q}}\right\|
_{L^p(\mathbb{R}^n)}
=\infty
$$
unless $f$
is a constant.
Thus, the Gagliardo quasi-semi-norm in \eqref{smp}
can not
recover the Triebel--Lizorkin quasi-semi-norm
$\|\cdot\|_{F_{p,q}^1(\rn)}$ when $s=1$ and,
in this sense, the assumption $n(\f1p-\f1q)<1$
in Theorem \ref{thm-4-1} seems to be \emph{sharp}.
Replacing the strong type quasi-norm in \eqref{smp} by the weak type
quasi-norm, and using Theorem \ref{thm-4-1}
with $\omega=1$,
we find that, for any
$f\in \dot{W}^{1,p}(\rn)$,
\begin{align*}
	\|f\|_{L^p(\rn)}+\sup_{\ld\in(0,\infty)} \ld \lf\{\int_
	{\RR^n}\left[\int_{ \RR^n} \mathbf{1}_
	{E_f(\ld,q)}(x, y)\, dy \right]^{\f pq}
	\, dx \r\}^{\f1p}
	\sim\|f\|_{F_{p,2}^1(\rn)}
\end{align*}
with the positive equivalence constants independent of $f$.
This indicates that  Theorem \ref{thm-4-1} is a \emph{perfect
replacement} of \eqref{smp} in
the critical case $s=1$.

\item[{\rm(iv)}]
Let $p\in[1,\infty)$
and $q\in(0,\infty)$ satisfy $n(\f 1p-\f1q) <1$.
Let $\mathbf{W\dot{F}}^1_{L^p_{\omega}(\rn),q}(\rn)$
be the
weak-type space $\mathbf{W\dot{F}}^1_{X,q}(\rn)$
in Definition \ref{np} below
with $X:=L^p_{\omega}(\rn)$.
Assume that $q_1,q_2\in(0,\infty)$
satisfy $n(\f 1p-\f1{q_1}) <1$
and $n(\f 1p-\f1{q_2}) <1$.
From Theorem \ref{thm-4-1},
it follows that
$$\mathbf{W\dot{F}}^1_{L^p_{\omega}(\rn),q_1}(\rn)\cap
\dot{W}^{1,p}_\omega(\rn)
=\mathbf{W\dot{F}}^1_{L^p_{\omega}(\rn),q_2}(\rn)
\cap
\dot{W}^{1,p}_\omega(\rn)$$
with equivalent quasi-norms. Thus, when $q\in(0,\infty)$
satisfies $n(\f 1p-\f1q) <1,$ the space  $\mathbf{W\dot{F}}^1_{L^p_{\omega}(\rn),q}(\rn)\cap
\dot{W}^{1,p}_\omega(\rn)$ is independent of $q.$	
\item[{\rm(v)}] For the purpose of our
		applications later, we only consider $A_1(\rn)$-weights
		here. However, our proof actually works equally well for
		more general $A_p(\rn)$-weights. For instance,
		a slight modification of the proofs in this
		section shows that \eqref{1-1e} holds
		true for any given
		$1\leq p=q<\infty$ and $\og\in
		A_{\min\{p, 1+\f pn\}}(\rn)$.		
	\end{itemize}
\end{remark}

\begin{theorem}\label{thm-4-2}
Let $p\in[1,\infty)$ and $\og$ be a non-negative
function on $\RR$. Assume that there exists a positive
constant $C_1$ such that, for any	$f\in C^1(\RR)$
satisfying that $f'$ has compact support, 		
	\begin{equation*}
	\sup_{\ld\in(0,\infty)} \ld^p \int_{\RR^2}
	\mathbf{1}_{E_f(\ld,p)} (x, y) \og(x)\, dx\,dy\leq C_1	
	\int_{\RR} | f'(x)|^p \og(x) \, dx,
	\end{equation*}
	where $E_f(\ld,p)$ is the same as  in \eqref{E} for
	any $\ld\in(0,\infty)$.
	Then $\og\in A_p(\RR)$.
\end{theorem}
\begin{remark}
In one dimension, Theorem \ref{thm-4-2} implies
that the $A_1(\mathbb{R})$
condition in Theorem \ref{thm-1-3} with $p=1$ is sharp.
\end{remark}
The proof of Theorem \ref{thm-4-1} is given in
Subsections \ref{subsection-2-1} and
\ref{sec-4-2}. In Subsection \ref{sec:4-3}, we
prove Theorem \ref{thm-4-2}.

\subsection{Proof of Theorem \ref{thm-4-1}: Upper Estimate }\label{sec-4-2}

We begin with recalling some conclusions about
Muckenhoupt weights $A_p(\rn)$.
The following lemma is a part of \cite[Theorem 7.1.9]{ref1}.
\begin{lemma}\label{jidahanshu}
	Let $p\in(1,\infty)$ and $\omega\in A_{p}(\rn)$. Then
	there exists a positive constant $C$, independent of $\omega,$ such that,
	for any $f\in L^p_\omega(\rn),$
	$$\|\mathcal{M}(f)\|_{L^p_\omega(\rn)}\leq C
	[\omega]_{A_p(\rn)}^{\frac{1}{p-1}}\|f\|
	_{L^p_\omega(\rn)},$$
	where $\mathcal{M}$ is the same as  in \eqref{2-4-c}.
\end{lemma}
For the proof of the upper bound in \eqref{1-1e},
we need to use several adjacent systems of dyadic
cubes, which can be found, for instance, in \cite[Section 2.2]{LSU12}.

\begin{lemma} \label{lem-HK}
	For any $\alpha\in\{0,\frac{1}{3},\frac{2}{3}\}^n,$
	let
	$$\mathcal{D}^\al:=\{2^j(k+[1,0)^n+(-1)^j\al):
	\ j\in\zz,\,k\in\zz^n\}.$$
	Then
	\begin{itemize}
		\item[{\rm(i)}] for any $Q,Q'\in\mathcal{D}^\al$ with the same
		$\alpha\in\{0,\frac{1}{3},\frac{2}{3}\}^n,$
		$Q\bigcap Q'\in\{\emptyset,Q,Q'\};$
		\item[{\rm(ii)}] for any ball $B\subset\rn,$ there
		exists an $\alpha\in\{0,\frac{1}{3},\frac{2}{3}
		\}^n$ and a $Q\in\mathcal{D}^\al$ such that $B\subset
		Q\subset CB,$ where the positive constant $C$
		depends only on $n.$
	\end{itemize}
\end{lemma}
Applying Lemma \ref{jubu} with $X:=L^p_{\omega}(\rn)$, we obtain the following corollary.
\begin{corollary}\label{jiaquan}
	Let $p\in[1,\infty)$ and $\omega\in A_p(\rn)$. Then
	\begin{align*}
		\dot{W}^{1,p}_{\omega}(\rn)\subset
		W^{1,1}_{\rm{loc}}(\rn).
	\end{align*}
\end{corollary}
We also need  the following
conclusion.
\begin{lemma}\label{lem-1-2a}
	Let $p\in[1,\infty),$ $\omega\in A_p(\rn)$, and $f\in \dot{W}^{1,p}_{\omega}(\rn)$.
Then there exists an $A\subset \rn$ with $|A|=0$ and
	a positive constant $C_{(n)}$, depending only on $n$, such that,
	for any $x\in \rn\setminus A$, $r\in(0,\infty)$, and any ball $B_1\subset B:=B(x,r)\subset 3B_1$,
	\begin{align}\label{guize}
		|f(x)-f_{B_1}|\leq C_{(n)} r\sum_{j=0}
		^\infty 2^{- j} \fint_{2^{-j} B} |\nabla f(y)|\, dy.
	\end{align}
\end{lemma}

\begin{proof}
By Corollary \ref{jiaquan}, we find that $f\in W^{1,1}_{\rm{loc}}(\rn).$
This, together with  the classical Poincar\'e inequality (see, for instance,
\cite[Theorem 8.2.4(b)]{DHHR}), implies that,
for any ball $B_0\subset\rn$,
\begin{align}\label{yiban}
\int_{B_0}|f(y)-f_{B_0}|\,dy
\leq
Cr_{B_0}\int_{B_0}|\nabla f(y)|\,dy,
\end{align}
where the positive constant $C$  depends only on $n.$
By the Lebesgue differentiation theorem,
we find that there exists an $A\subset\rn$ with $|A|=0$ such that,
for any $x\in\rn\setminus A,$
\begin{align*}
\lim_{j\to\infty}\frac{1}{|2^{-j}B|}\int_{2^{-j}B}|f(x)-f(y)|\,dy=0
\end{align*}
and hence
\begin{align*}
f(x)=\lim_{j\to\infty}\frac{1}{|2^{-j}B|}
\int_{2^{-j}B} f(y)\,dy=\lim_{j\to\infty} f_{2^{-j}B},
\end{align*}
where $B:=B(x,r)$ with $r\in(0,\infty).$ This, together with
\eqref{yiban}, implies that
\begin{align}\label{yueguang}
|f(x)-f_B|
&=
\lim_{j\to\infty}|f_{2^{-j}B}-f_B|
\leq
\sum_{j=0}^\infty|f_{2^{-j-1}B}-f_{2^{-j}B}|\\\noz
&\leq
\sum_{j=0}^\infty\frac{1}{|2^{-j-1}B|}
\int_{2^{-j-1}B}|f(y)-f_{2^{-j}B}|\,dy\\\noz
&\lesssim
r\sum_{j=0}
^\infty 2^{- j} \fint_{2^{-j} B} |\nabla f(y)|\, dy.
\end{align}
On the other hand, using \eqref{yiban} again, we have
\begin{align*}
|f_B-f_{B_1}|
&\leq
\frac{1}{|B_1|}\int_{B_1}|f(y)-f_B|\,dy
\lesssim
\frac{1}{|B|}\int_{B}|f(y)-f_B|\,dy\\\noz
&\lesssim
r\fint_{B}|\nabla f(y)|\,dy.
\end{align*}
From this and  \eqref{yueguang}, we deduce that \eqref{guize} holds true.
This finishes the proof of Lemma \ref{lem-1-2a}.
\end{proof}
We next state the upper estimate of Theorem \ref{thm-4-1} in the following
separate theorem.
\begin{theorem}\label{thm41}
Let $p\in[1,\infty)$
and $q\in(0,\infty)$ satisfy $n(\f 1p-\f1q) <1$.
Assume that $\og\in A_1(\rn)$. Then there exist
positive constants $C_1$ and $C_{([\omega]
_{A_1(\rn)})}$ such that, for any $f\in \dot{W}^{1,p}_\omega(\rn)$,		
	\begin{align}\label{zy}
		&\sup_{\ld\in(0,\infty)} \ld^p \int_
		{\RR^n}\left[\int_{ \RR^n} \mathbf{1}_
		{E_f(\ld,q)}(x, y)\, dy \right]^{\f pq}
		\og(x) \,dx\\\noz
		&\quad\leq C_1C_{([\omega]_{A_1(\rn)})}	
		\int_{\RR^n} | \nabla f(x)|^p \og(x) \, dx,
	\end{align}
	where $E_f(\ld,q)$ for any $\ld\in(0,\infty)$
	 is the same as  in
	\eqref{E},
	the positive constant  $C_1$ are independent of
	$\omega$,  the positive constant
	$C_{([\omega]_{A_1(\rn)})}$ increases as $[\omega]_
	{A_1(\rn)}$ increases, and
	$C_{(\cdot)}$ is
	continuous on $(0,\infty)$.
\end{theorem}

Let $p,$ $q,$ $\omega$, and $f$ be the same as in  Theorem \ref{thm41}.
For any $x,y\in\rn,$ let $B_{x,y} :=B(\f {x+y}2,
|x-y|)$,
\begin{align}\label{wuyu1}
&E_f^{(1)}(1,q)\\
&\quad:=\lf\{ (x, y)\in \RR^n\times
\RR^n:\ |f(x)-f_{B_{x,y}}|\ge 2^{-1}
|x-y|^{\f nq+1} \r\}\setminus (A\times\rn),\noz
\end{align}
and
\begin{align}\label{wuyu2}
&E_f^{(2)}(1,q)\\
&\quad:=\lf\{ (x, y)\in \RR^n\times \RR^n:\
|f(y)-f_{B_{x,y}}|\ge 2^{-1} |x-y|^{\f nq+1} \r\}\setminus (\rn\times A),\noz
\end{align}
where $A$ is the same as in Lemma \ref{lem-1-2a}.
We need the following several lemmas.
\begin{lemma}\label{lem-1-4}
	Let $p\in[1,\infty)$, $q\in(0,\infty)$, and $\omega
	\in A_1(\rn)$. Then there exists a positive constant
	$C$ such that, for any $f\in \dot{W}^{1,p}_\omega(\rn),$
	\begin{align}\label{1-3e}
		&\int_{\rn} \left[ \int_{\rn} \mathbf{1}_{E_f^{(1)}(1,q)}
		(x,y)\, dy\right]^{\f pq} \og(x)\, dx\\
        &\quad\leq C[
		\omega]_{A_1(\rn)} 	 \int_{\RR^n} | \nabla f(x)|
		^p \og(x) \, dx,\noz
	\end{align}
	where $E_f^{(1)}(1,q)$ is the same as  in \eqref{wuyu1} and the
	positive constant $C$ is independent of $\omega.$
\end{lemma}

\bp Indeed, to show this lemma, in the  proof below, we only need to
consider the case $\va=\f12$. However, some of the estimates
below with an arbitrarily given $\va\in (0,1)$
are needed in the proof of
Lemma \ref{lem-1-5} below and hence we give the
details here.

By Lemma \ref{lem-1-2a} with both  $B:=B(x,2|x-y|)$
and $B_1:=B(\frac{x+y}{2},|x-y|)=:B_{x,y}$, we find that there exists a
positive constant $c_1$, depending only on $n$, such
that, for any $x,y\in\rn$,
\begin{align*}
	|f(x)-f_{B_{x,y}} | \leq c_1 |x-y| \sum_{j=0}
	^\infty 2^{-j} \fint_{B(x, 2^{-j+1} |x-y|)}
	|\nabla f(z)|\, dz.
\end{align*}
This implies that, for any $(x, y)\in E_f^{(1)}(1,q)$,
there exists a $j_{x,y}\in\zz_+$,
depending only on both $x$ and $y$,
 such that
\begin{equation}\label{1-4e}
	\fint_{B(x, 2^{-j_{x,y}+1} |x-y|)} |\nabla f(z)|\,
	dz > c_2 2^{j_{x,y}(1-\va)} |x-y|^{n/q},
\end{equation}
where the positive constant $c_2$ depends only
on both $n$ and $\va$. For any
$(x, y)\in E_f^{(1)}(1,q)$,
let
$$B_{j_{x,y}}:= B(x, 2^{-j_{x,y}+1}|x-y|).$$
Applying Lemma \ref {lem-HK}
to the ball $B_{j_{x,y}}:= B(x, 2^{-j_{x,y}+1} |x-y|)$, we
find that there exists some $\alpha_{x,y}\in\{0,
\frac{1}{3},\frac{2}{3}\}^n$ and $Q_{j_{x,y}}\in\mathcal{D}^{\al_{x,y}}$
such that $B_{j_{x,y}}\subset Q_{j_{x,y}}\subset CB_{j_{x,y}},$ where
the positive constant $C$ depends only on $n.$
This implies that, for any $(x,y)\in
E^{(1)}_f(1,q)$,
there exists
a $j_{x,y}\in\zz_+$ and
a positive constant $c_3,$
depending only on $n,q,$ and $\va$, such that
\begin{equation}\label{1-5e}
	\fint_{Q_{j_{x,y}}} |\nabla f(z)|\, dz >
	c_3 2^{j_{x,y}(1-\va)} |2^{j_{x,y}} Q_{j_{x,y}}|^{\f 1q}.
\end{equation}
For any $\alpha\in\{0,\frac{1}{3},\frac{2}{3}\}^n$
and $j\in\zz_+$, we denote by the \emph{symbol}
$\cA_\al^j$ the collection of all the  dyadic cubes
$Q\in\CD^\al$  which satisfies \eqref{1-5e}
with $Q_{j_{x,y}}$ replaced by $Q$,
where $\mathcal{D}^\al$ is the same as  in Lemma \ref{lem-HK}.
From the definition of  $\cA_\al^j,$
Lemma \ref{Lemma2.1}(ii) with $\omega$ regarded as an $A_p(\rn)$ weight, and the fact
that $[\omega]_{A_p(\rn)}\leq[\omega]_{A_1(\rn)}$,
we deduce that, for any
$\alpha\in\{0,\frac{1}{3},
\frac{2}{3}\}^n$, $j\in\zz_+,$ and
$Q\in \cA_{\al}^j$,
\begin{align}\label{1-8e}
\og(Q)| 2^j Q|^{\f pq}
&\leq
c_3^{-p}[\omega]_{A_1(\rn)}
2^{-{j(1-\va)p} }  \int_Q |\nabla f(z)|^p
\og(z)\, dz\\\noz
&\leq
c_3^{-p}[\omega]_{A_1(\rn)}
2^{-{j(1-\va)p} }\|f\|_{L^p_{\omega}(\rn)}^p
\end{align}
which further implies that
$$\sup_{Q\in \cA_\al^j}l(Q) <\infty$$
with $l(Q)$ for any $Q\in\cA_\al^j
$ being the edge length of $Q$. Thus, every cube $Q\in \cA_\al^j$ is
contained in a dyadic cube in $\cA_\al^j$ that is
maximal with respect to the set inclusion.
For any $\alpha\in\{0,\frac{1}{3},\frac{2}{3}\}^n$,
we denote by the \emph{symbol} $\cA_{\al,\max}^j$
the collection of all dyadic cubes in $\cA_\al^j$
that are maximal with respect to the set inclusion.
Clearly, the maximal dyadic cubes in $\cA_{\al,
	\max}^j$ are pairwise disjoint. For any $(x,y)
\in E_f^{(1)}(1,q),$ since $(x,y)\in B_{j_{x,y}}\times2^{j_{x,y}}B_{j_{x,y}}$,
it follows that $(x,y)\in Q_{j_{x,y}}\times2^{j_{x,y}}Q_{j_{x,y}}$ for some
$Q_{j_{x,y}}\in \cA_{\al_{x,y},\max}^{j_{x,y}}.$ Thus, we have
\begin{align*}
	E_f^{(1)}(1,q)&\subset \bigcup_{j=0}^\infty \bigcup
	_{\alpha\in\{0,\frac{1}{3},\frac{2}{3}\}^n}
	\bigcup_{Q\in \cA_{\al,\max}^j} (Q\times 2^j Q),
\end{align*}
which implies that, for any $x\in\rn$,
\begin{equation}\label{1-6e}
	\int_{\rn} \mathbf{1}_{	E_f^{(1)}(1,q) } (x,y)\,
	dy\leq \sum_{j=0}^\infty \sum_{\alpha\in
		\{0,\frac{1}{3},\frac{2}{3}\}^n} \sum_{Q\in
		\cA_{\al,\max}^j}\mathbf{1}_Q(x) |2^j Q|.
\end{equation}
Now, we prove \eqref{1-3e} by considering the following two cases on $q.$

Case $1)$ $q\in(0,p]$. In this case, by \eqref{1-6e},  the Minkowski
inequality on $L^{\frac{p}{q}}(\rn),$
\eqref{1-8e}, and the fact that the dyadic cubes in
$\cA_{\al,\max}^j$ are pairwise disjoint, we obtain
\begin{align*}
	\mathrm{I}:&=\left\{\int_{\rn}\left[\int_{\rn}\mathbf{1}
	_{E_f^{(1)}(1,q)}(x,y)\,dy\right]^{\frac{p}{q}}\omega(x)
	\,dx\right\}^{\frac{q}{p}}\\
	&\leq\sum_{j=0}^\infty \sum_{\alpha\in\{0,
		\frac{1}{3},\frac{2}{3}\}^n}\left\{\int_{\rn}
	\left[\sum_{Q\in \cA_{\al,\max}^j} |2^j Q|
	\mathbf{1}_Q(x)\right]^{\frac{p}{q}}\omega(x)
	\,dx\right\}^{\frac{q}{p}}\noz\\
	&=\sum_{j=0}^\infty \sum_{\alpha\in\{0,\frac{1}{3},
		\frac{2}{3}\}^n} \left[\sum_{Q\in\cA_{\al,
			\max}^j} |2^j Q|^{\f pq} \og(Q)\right]^{\f qp}\\\noz
	&\lesssim[\omega]_{A_1(\rn)}^{q/p}\sum_{j=0}
	^\infty \sum_{\alpha\in\{0,\frac{1}{3},
		\frac{2}{3}\}^n} \left[ \sum_{Q\in
		\cA_{\al,\max}^j} 2^{- {j(1-\va)p} }
	\int_Q |\nabla f(z)|^p \og(z)\, dz \right]^{\f qp} \\
	&\lesssim[\omega]_{A_1(\rn)}^{q/p}\sum_{j=0}
	^\infty 2^{- {j(1-\va)q}} \left[ \int_{\rn}
	|\nabla f(z)|^p \og(z)\, dz \right]^{\f qp}\\
	&\lesssim [\omega]_{A_1(\rn)}^{q/p} \left[
	\int_{\rn} |\nabla f(z)|^p \og(z)\, dz
	\right]^{\f qp}.
\end{align*}
This gives the desired
estimate \eqref{1-3e} for any given $q\in(0,p]$.

Case $2)$ $q\in(p,\infty)$.
In this case, recall that, for any $r\in(0,1]$
and $\{a_j\}_{j\in\zz_+}\subset(0,\infty),$
\begin{align}\label{ppp}
	\left(\sum_{j\in\zz_+ }a_j\right)^r\leq \sum_{j\in\zz_+}a_j^r.
\end{align}
By this,
\eqref{1-6e}, \eqref{1-8e}, and the fact that
the dyadic cubes in $\cA_{\al,\max}^j$ are
pairwise disjoint, we conclude that
\begin{align*}
	&\int_{\rn} \left[ \int_{\rn} \mathbf{1}_{	
		E_f^{(1)}(1,q) } (x,y)\, dy\right]^{\f pq} \og(x)\, dx\\
	&\quad\leq\sum_{j=0}^\infty \sum_{\alpha\in\{0,
		\frac{1}{3},\frac{2}{3}\}^n} \sum _{Q\in \cA_{
			\al,\max}^j}\og(Q) |2^j Q|^{\f pq}\\
	&\quad\lesssim[\omega]_{A_1(\rn)} \sum_{j=0}^
	\infty \sum_{\alpha\in\{0,\frac{1}{3},
		\frac{2}{3}\}^n} \sum _{Q\in \cA_{\al,\max}^j}2
	^{- {j(1-\va)p} } \int_Q |\nabla f(z)|^p \og(z)\, dz\\
	&\quad\lesssim [\omega]_{A_1(\rn)}3^n\sum_{j=0}
	^\infty2^{- {j(1-\va)p} }\int_{\rn} |\nabla
	f(z)|^p \og(z)\, dz.
\end{align*}
This gives the desired estimate \eqref{1-3e}
for any given $q\in(p,\infty)$, which then completes
the proof of Lemma \ref{lem-1-4}.
\ep

\begin{lemma}\label{lem-1-5} Let $p\in[1,\infty)$
	and $q\in(0,\infty)$ satisfy $n(\f 1p-\f1q) <1$.
	Let $\omega\in A_1(\rn)$. Then there exist positive
	constants $C$ and $C_{([\omega]_{A_1(\rn)})}$
	such that, for any $f\in \dot{W}^{1,p}_\omega(\rn),$
	\begin{align}\label{1-10b}
		&\int_{\rn} \left[ \int_{\rn} \mathbf{1}_{E_f^{(2)}(1,q)}
		(x,y)\, dy\right]^{\f pq} \og(x)\, dx\\
		&\quad\leq C C_{([\omega]_{A_1(\rn)})}\int_{\RR^n}
		| \nabla f(x)|^p \og(x) \, dx,\noz
	\end{align}
	where $E_f^{(2)}(1,q)$ is the same as  in \eqref{wuyu2}, the positive
	constant $C$ is independent of $\omega$,  the
	positive constant $C_{([\omega]_{A_1(\rn)})}$
	increases as $[\omega]_{A_1(\rn)}$ increases, and $C_{(\cdot)}$ is
	continuous on $(0,\infty)$.
\end{lemma}
Before we prove Lemma \ref{lem-1-5}, we need the
following lemma, which can be found in \cite[p.\,18]{ref6}.
\begin{lemma}\label{lemma3.77}
	Let $p\in[1,\infty)$ and $\omega\in A_p(\rn).$ For any
	$g\in L^p_{\omega}(\rn)$ and $x\in\rn$, let
	\begin{align*} Rg(x):=\sum_{k=0}^{\infty}
		\frac{\cm^kg(x)}{2^k\|\cm\|^k_{L^p_{\omega}(\rn)\to L^p_{\omega}(\rn)}},
	\end{align*}
	where, for any $k\in\mathbb{N},\,\cm^k
	:=\cm\circ\cdots\circ\cm$ is the $k$ iterations
	of the Hardy--Littlewood maximal operator
	and $\cm^0g(x):=|g(x)|.$ Then, for any
	$g\in L^p_{\omega}(\rn)$ and $x\in\rn$,
	\begin{itemize}
		\item[{\rm(i)}] $|g(x)|\leq Rg(x);$
		\item[{\rm(ii)}] $Rg\in A_1(\rn)$ and $[Rg]_
		{A_1(\rn)}\leq2\|\cm\|_{L^p_{\omega}(\rn)\to L^p_{\omega}(\rn)},$ where $\|\cm\|_{L^p_{\omega}(\rn)\to L^p_{\omega}(\rn)}$ denotes the operator norm of
		$\cm$ mapping $L^p_{\omega}(\rn)$ to $L^p_{\omega}(\rn)$;
		\item[{\rm(iii)}] $\|Rg\|_{L^p_{\omega}(\rn)}\leq2\|g\|_{L^p_{\omega}(\rn)}.$
	\end{itemize}
\end{lemma}
The following lemma is just \cite[Lemma 2.6]{ref5}.
\begin{lemma}\label{lemma3.4}
	Let $X$ be a ball Banach function space. Then $X$ coincides with its
	second associate space $X''$. In other words, a
	function $f$ belongs to $X$ if and only if
	it belongs to $X''$ and,
	in that case,
	$$
	\|f\|_X=\|f\|_{X''}.
	$$
\end{lemma}
\begin{proof}[Proof of Lemma \ref{lem-1-5}]
	Let $\va\in (0,1)$ be a sufficiently small
	absolute constant. By an argument similar to that
	used in the proof of Lemma \ref{lem-1-4}, we know that there
	exist positive constants $C_1$ and $C_2$,
	depending only on $n,q,$ and $\va$, such that,
	for any $j\in\zz_+,$ $\alpha\in\{0,\frac{1}{3},
	\frac{2}{3}\}^n,$ and $Q\in\mathcal{A}^j_{\al,\max},$
	\begin{align*}
		C_1 2^{j(1-\va)} |2^j Q|^{\f 1q}<
		\fint_{Q} |\nabla f(z)|\, dz,
	\end{align*}
	\begin{equation}\label{1-8e1}
		|2^j Q|^{\f pq} \leq C_2[\omega]_{A_1(\rn)}
		2^{-{j(1-\va)p} } \f 1 {\og(Q)} \int_Q |
		\nabla f(z)|^p \og(z)\, dz,
	\end{equation}
	and
	\begin{align}\label{1-11e}
		E_f^{(2)}(1,q) \subset \bigcup_{j=0}^\infty \bigcup
		_{\alpha\in\{0,\frac{1}{3},\frac{2}{3}\}^n}
		\bigcup_{Q\in \cA_{\al,\max}^j}\lf( 2^j Q \times Q\r),
	\end{align}
	where $\cA_{\al,\max}^j$ is the same as in
	the proof of Lemma \ref{lem-1-4}.
	Using \eqref{1-11e}, we conclude that, for any $x\in\rn,$
	\begin{equation}\label{1-12e}
		\int_{\rn} \mathbf{1}_{E_f^{(2)}(1,q) }
		(x,y)\, dy\leq \sum_{j=0}^\infty
		\sum_{\alpha\in\{0,\frac{1}{3},\frac{2}{3}\}^n}
		\sum_{Q\in \cA_{\al,\max}^j}|Q|\mathbf{1}_{2^j Q} (x).
	\end{equation}
	Now, we prove \eqref{1-10b} by
	considering the following two cases on $q.$
	
	Case $1)$ $ q\in[p,\infty)$ and $n(\f 1p-\f1q) <1$.
	In this case, by \eqref{ppp}, \eqref{1-12e}, \eqref{1-8e1}, and
	Lemma \ref{Lemma2.1}(i), we find that
	\begin{align}\label{4566}
		&\int_{\rn} \left| \int_{\rn} \mathbf{1}_{E_f^{(2)}(1,q) }
		(x,y)\, dy\right|^{\f pq} \og(x)\, dx\\\noz
		&\quad\leq\sum_{j=0}^\infty \sum_{\alpha\in\{0,
			\frac{1}{3},\frac{2}{3}\}^n} \sum_{Q\in
			\cA_{\al,\max}^j}|Q|^{\f pq} \og (2^j Q)\\\noz
		&\quad\leq [\omega]_{A_1(\rn)}\sum_{j=0}^\infty
		\sum_{\alpha\in\{0,\frac{1}{3},\frac{2}{3}\}^n}
		\sum_{Q\in \cA_{\al,\max}^j}|2^j Q|^{\f pq}
		2^{j n( 1-\f pq)} \og ( Q)\\\noz
		&\quad\lesssim[\omega]_{A_1(\rn)}^2\sum_{j=0}^\infty
		\sum_{\alpha\in\{0,\frac{1}{3},\frac{2}{3}\}^n}
		\sum_{Q\in \cA_{\al,\max}^j} 2^{jp[n(\f1p-\f1q)-(1-\va)]} \int_Q |\nabla f(z)|^p \og(z)\, dz\\\noz
		&\quad\lesssim[\omega]_{A_1(\rn)}^23^n\sum_{j=0}
		^\infty 2^{jp[n(\f1p-\f1q)-(1-\va
			)]} \int_{\rn}
		|\nabla f(z)|^p \og(z)\, dz\\\noz
		&\quad\lesssim[\omega]_{A_1(\rn)}^2\int_{\rn}
		|\nabla f(z)|^p \og(z)\, dz,
	\end{align}
	where, in the fourth step, we took
	an $\va\in (0,1)$
	sufficiently small so that
	$n(\f1p-\f 1q) <1-\va $. This can be done because
	$n(\f 1p-\f1q) <1$. This finishes the proof of
	\eqref{1-10b} in this case.
	
	Case $2)$ $q\in(0,p)$. In this case, let $r:=\f pq$
	, $r':= \f r{r-1}$, and $\mu(x):=[\omega(x)]^{1-r'}$ for any $x\in\rn.$
	Since $\omega\in A_1(\rn)\subset A_r(\rn),$ it follows that $\mu\in A_{r'}(\rn)$ and
	\begin{align}\label{quan}
		[\omega^{1-r'}]_{A_{r'}(\rn)}^{\frac{1}{r'-1}}=[\omega]_{A_r(\rn)}\leq[\omega]_{A_1(\rn)}
	\end{align}
	(see, for instance, both (4) and (6) of \cite[ Proposition 7.1.5]{ref1}).
	It is known that $[L_\omega^r(\rn)]'=L^{r'}_{\mu}(\rn),$
	where $[L_\omega^r(\rn)]'$ denotes the
	associated space of $L_\omega^r(\rn)$   in
	Definition \ref{def-X'} (see \cite[Theorem 2.7.4]{DHHR}). From this,
	Lemma \ref{lemma3.4} with $X:=L^r_\omega(\rn)$, Definition \ref{def-X'} with $X:=L^{r'}_\mu(\rn)$, and
	Lemma \ref{lemma3.77}(i), we deduce that
	\begin{align}\label{icu1}
		&\left\{\int_{\rn}\left[\int_{\rn}\mathbf{1}_{E_f^{(2)}(1,q)}(x,y)\,dy\right]^r\omega(x)\,dx\right\}^{\f1r}\\\noz
		&\quad=\left\|\int_{\rn}\mathbf{1}_{E_f^{(2)}(1,q)}(\cdot,y)\,dy\right\|_{[L^r_\omega(\rn)]''}
		=\left\|\int_{\rn}\mathbf{1}_{E_f^{(2)}(1,q)}(\cdot,y)\,dy\right\|_{[L^{r'}_\mu(\rn)]'}\\\noz
		&\quad=\sup_{\|g\|_{L^{r'}_\mu(\rn)}=1}\int_{\rn}
		\int_{\rn} \mathbf{1}_{E_f^{(2)}(1,q) }
		(x,y)\, dyg(x)\,dx\\\noz
		&\quad\leq \sup_{\|g\|_{L^{r'}_\mu(\rn)}=1}\int_{\rn}
		\int_{\rn} \mathbf{1}_{E_f^{(2)}(1,q) }
		(x,y)\, dyRg(x)\,dx\\\noz
		&\quad\lesssim\sup_{\|g\|_{L^{r'}_\mu(\rn)}
			=1}
		[Rg]_{A_1(\rn)}^2\int_{\rn}|\nabla f(x)|^q Rg(x)\,dx,
	\end{align}
	where, in the last step, we used \eqref{4566} with $p:=q$ and $\omega:=Rg.$
	On the other hand, by the H\"older inequality, both (ii) and (iii) of Lemma \ref{lemma3.77},
	Lemma \ref{jidahanshu}, and \eqref{quan}, we find that, for any $g\in L^{r'}_{\mu}(\rn)$
	with $\|g\|_{L^{r'}_{\mu}(\rn)}=1$,
	\begin{align*}
		&[Rg]_{A_1(\rn)}^2\int_{\rn}|\nabla f(x)|^q Rg(x)\,dx\\
		&\quad\lesssim[Rg]_{A_1(\rn)}^2\|g\|_{L^{r'}_{\mu}(\rn)}
		\left\{\int_{\rn}|\nabla f(x)|^p\omega(x)\,dx\right\}^{\f1r}\\\noz
		&\quad\lesssim\|\cm\|_{L^{r'}_\mu(\rn)\to L^{r'}_\mu(\rn)}^2
		\left\{\int_{\rn}|\nabla f(x)|^p\omega(x)\,dx\right\}^{\f1r}\\\noz
		&\quad\lesssim [\mu]_{A_{r'}(\rn)}^{\frac{2}{r'-1}}
		\left\{\int_{\rn}|\nabla f(x)|^p\omega(x)\,dx\right\}^{\f1r}\\\noz
		&\quad\lesssim [\omega]_{A_1(\rn)}^2
		\left\{\int_{\rn}|\nabla f(x)|^p\omega(x)\,dx\right\}^{\f1r}.
	\end{align*}
	This, combined with \eqref{icu1}, implies that
	\begin{align}\label{icu}
		&\left\{\int_{\rn}\left[\int_{\rn}\mathbf{1}_{E_f^{(2)}(1,q)}(x,y)\,dy\right]^r\omega(x)\,dx\right\}^{\f1r}\\\noz
		&\quad\lesssim [\omega]_{A_1(\rn)}^2
		\left\{\int_{\rn}|\nabla f(x)|^p\omega(x)\,dx\right\}^{\f1r}.	
	\end{align}
	This finishes the proof of
	\eqref{1-10b} in this case.
	
	Let
	$$
	C_{([\omega]_{A_1(\rn)})}:=
	\begin{cases}
		[\omega]_{A_1(\rn)}^2, &q\in[p,\infty)\,\text{and}
		\displaystyle\,n\left(\f 1p-\f1q\right) <1,\\
		[\omega]_{A_1(\rn)}^{(2p)/q},&q\in(0,p).
	\end{cases}$$
	It is easy to see that $C_{([\omega]_{A_1(\rn)})}$
	increases as $[\omega]_{A_1(\rn)}$ increases. From both
	\eqref{4566} and \eqref{icu}, we deduce that
	$$\int_{\rn} \left[ \int_{\rn} \mathbf{1}_{	E_f^{(2)}(1,q) }
	(x,y)\, dy\right]^{\f pq} \og(x)\, dx
	\lesssim C_{([\omega]_{A_1(\rn)})}	\int_{\RR^n} |
	\nabla f(x)|^p \og(x) \, dx.$$
	This finishes the proof of Lemma \ref{lem-1-5}.
\end{proof}
Finally, we show Theorem \ref{thm41}.
\begin{proof}[Proof of Theorem \ref{thm41}]
	Without loss of generality, we may assume that
	$\ld=1$ because, otherwise, we can replace $f$ by $f/\ld$ for any $\ld\in(0,\infty)$. Let $A$ be the same as in Lemma \ref{lem-1-2a}.
	Since $|A|=0$, it follows that,
	to show Theorem \ref{thm41},
	it suffices to prove \eqref{zy} with
	$E_f(1,q)$ replaced by  $E_f(1,q)\setminus[(A\times\rn)\cup(\rn\times A)]$.
	On the other hand, it is easy to show that $$E_f(1,q)\setminus[(A\times\rn)
\cup(\rn\times A)] \subset E_f^{(1)}(1,q) \cup E_f^{(2)}(1,q),$$
	where $E_f^{(1)}(1,q)$ and $E_f^{(2)}(1,q)$ are  the same, respectively, as in \eqref{wuyu1} and \eqref{wuyu2}.
	Thus, it suffices to prove the corresponding upper
	estimates for both the sets $E_f^{(1)}(1,q)$ and $E_f^{(2)}(1,q)$, which are done, respectively, in
	Lemmas \ref{lem-1-4} and \ref{lem-1-5}. This finishes the
	proof of Theorem \ref{thm41}.
\end{proof}

\subsection {Proof of Theorem \ref{thm-4-1}: Lower Estimate}\label{subsection-2-1}

In this subsection, we show a generalization of
the lower estimate of Theorem \ref{thm-4-1} on
ball quasi-Banach function spaces, which plays
an essential role in the proof of
Theorem \ref{theorem3.9} below.
In what follows, we use the symbol
$C^2(\rn)$ to denote  the set
of all twice continuously
differentiable functions on $\rn.$
\begin{theorem}\label{theorem4.88}
Let $X$ be a ball quasi-Banach function space
and $q\in(0,\infty)$. Then, for any $f\in C^2(\rn)$ with $|\nabla f|\in
C_{\mathrm{c}}(\rn)$,
\begin{align}\label{8002}
	\liminf_{\ld\to\infty}\ld^q\left\|\int_{\rn}
	\mathbf{1}_{\{y\in\rn:\ (\cdot,y)\in E_f(\ld,q)\}}
	(y)\,dy\right\|_{X^{\f 1 q}}\geq\frac{K(q,n)}{n}
	\|\,|\nabla f|\,\|_{X}^q,
\end{align}
where $E_f(\ld,q)$ for any $\ld\in(0,\infty)$ is the same as  in \eqref{E}, and $K(q,n)$ is the same as in \eqref{kqn}. Moreover,
if $X$ is a ball Banach function space,
then, for any $f\in C^2(\rn)$ with $|\nabla f|\in
C_{\mathrm{c}}(\rn)$,
\begin{align}\label{yzr}
	\lim_{\ld\to\infty}\ld^q\left\|\int_{\rn}
	\mathbf{1}_{\{y\in\rn:\ (\cdot,y)\in E_f(\ld,q)\}}
	(y)\,dy\right\|_{X^{\f 1 q}}=\frac{K(q,n)}{n}
	\|\,|\nabla f|\,\|_{X}^q.
\end{align}
\end{theorem}
\begin{proof}
Let $q\in(0,\infty),$ $f\in C^2(\rn)$ with $|\nabla f|\in
C_{\mathrm{c}}(\rn)$,
and $M\in(0,\infty)$ be such that $\supp\, (|\nabla f|)\subseteq B(\mathbf{0},M).$
Let $K:=B(\mathbf{0},M+1).$ In order to show
\eqref{yzr}, we first prove \eqref{8002}.
For any $\lambda\in(0,\infty),\ x\in\rn,$
and $\xi\in\mathbb{S}^{n-1},$ let
$$F_f(x,\xi,\ld,q):=\left\{t\in(0,\infty):\
\frac{|f(x+t\xi)-f(x)|^q}{t^q}>\lambda^qt^n\right\}.$$
Then, by the Fubini Theorem, we have, for any
$\lambda\in(0,\infty)$ and $x\in\rn,$
	\begin{align}\label{8001}
	\int_{\rn}\mathbf{1}_{\{y\in\rn:\ (x,
	y)\in E_f(\ld,q)\}}(y)\,dy=\int_{\mathbb{S}^{n-1}}
	\int_0^{\infty}\mathbf{1}_{F_f(x,\xi,\ld,q)}(t)t^{n-1}\,dt\,d\sigma(\xi).
	\end{align}
Since $f\in C^2(\rn)$ with $|\nabla f|\in
C_{\mathrm{c}}(\rn)$, it follows
that there exist constants $L_1\in(\|\,|\nabla f|\,\|_{L^{\infty}(\rn)},\infty)$ and $L_2\in(0,
\infty)$ such that, for any $t\in(0,\infty),$ $x\in\rn,$ and $\xi\in\mathbb{S}^{n-1},$
\begin{align}\label{8003}
	|f(x+t\xi) -f(x)|&\leq L_1 t
\end{align}
and
\begin{align}\label{8004}
	|f(x+t\xi) -f(x)-t \nabla f(x) \cdot \xi|
	&\leq L_2t^2.
\end{align}
Let $\lambda\in(L_1,\infty).$ By \eqref{8003},
we conclude that, for any $t\in(0,\infty),$ $x\in\rn,$
and $\xi\in\mathbb{S}^{n-1},$
\begin{align}\label{8005}
	\frac{|f(x+t\xi)-f(x)|^q}{|t|^q}\leq L_1^q
\end{align}
and hence
\begin{align}\label{8006}
	F_f(x,\xi,\ld,q)=\left\{t\in(0,\(L_1/\lambda\)^{\frac{q}{n}}):\
	\frac{|f(x+t\xi)-f(x)|^q}{t^q}>\lambda^qt^n\right\}.
\end{align}
From \eqref{8004}, we deduce that, for any $t\in(0,
\infty),$ $x\in\rn,$ and $\xi\in\mathbb{S}^{n-1},$
$$|\xi\cdot\nabla f(x)|-tL_2\leq \frac{|f(x+t\xi)
-f(x)|}{t}\leq|\xi\cdot\nabla f(x)|+tL_2,$$
which, combined with both \eqref{8005} and \eqref{8006}, further
implies that, for any $t\in(0,(L_1/\lambda)
^{\frac{q}{n}}) ,$
$$A^-_f(x,\xi,\lambda,q)\leq\frac{|f(x+t\xi)-f(x)|
^q}{t^q}\leq A^+_f(x,\xi,\lambda,q),$$
where
$$A^-_f(x,\xi,\lambda,q):=\left[\max\left\{|\xi\cdot\nabla
f(x)|-(L_1/\lambda)^{\frac{q}{n}}L_2,0\right\}\right]^q$$
and
$$A^+_f(x,\xi,\lambda,q):=\min\left\{\left[|\xi\cdot\nabla
f(x)|+(L_1/\lambda)^{\frac{q}{n}}L_2\right]^q,L_1^q\right\}.$$
Using this, we conclude that, for any $x\in\rn$
and $\xi\in\mathbb{S}^{n-1},$
\begin{align}\label{8007}
	F_f^-(x,\xi,\ld,q)\subset F_f
	(x,\xi,\ld,q)\subset F_f^+(x,\xi,\ld,q),
\end{align}
where
$$F_f^{\pm}(x,\xi,\ld,q):=\{t\in(0,\infty):
\ A^{\pm}_f(x,\xi,\lambda,q)>\lambda^qt^n\}.$$
\par
We now show that, for any $\lambda\in(L_1,\infty),$
$\xi\in \mathbb{S}^{n-1},$ and $x\in K^{\complement},$
$$F_f(x,\xi,\ld,q)=\emptyset$$
and hence, for any $x\in\rn$,
\begin{align}\label{8008}
	&\int_{\mathbb{S}^{n-1}}\int_0^{\infty}\mathbf{1}
	_{F_f(x,\xi,\ld,q)}(t)t^{n-1}\,dt\,d\sigma(\xi)\\\noz
	&\quad=
	\int_{\mathbb{S}^{n-1}}\int_0^{\infty}\mathbf{1}
	_{F_f(x,\xi,\ld,q)}(t)t^{n-1}\,dt\,d\sigma(\xi)
	\mathbf{1}_{K}(x).
\end{align}
Indeed, by \eqref{8006}, for any $\lambda\in(L_1,
\infty),$ we have $F_f(x,\xi,\ld,q)\subset (0,1)
.$ Thus, for any $\lambda\in(L_1,\infty),$
$x\in K^{\complement},$ $\xi\in \mathbb{S}^{n-1},$
and $t\in F_f(x,\xi,\ld,q)\subset (0,1),$ we obtain
$x+t\xi\in [B(\mathbf{0},M)]^{\complement}$
and hence $f(x+t\xi)=f(x).$
This implies that, for any $\lambda\in(L_1,\infty),$
$\xi\in \mathbb{S}^{n-1},$ and $x\in K^{\complement},$
$$F_f(x,\xi,\ld,q)=\emptyset,$$
which completes the proof of \eqref{8008}.
By \eqref{8001}, \eqref{8007}, and Definition \ref{Debqfs}(ii) together
with $X$ being a ball quasi-Banach function space,  we have
\begin{align}\label{8025}
&\ld^q\left\|\int_{\rn}\mathbf{1}_{\{y\in\rn:
\ (\cdot,y)\in E_f(\ld,q)\}}(y)\,dy\right\|_{X^{\f 1 q}}\\
&\quad\geq\ld^q\left\|\int_{\mathbb{S}^{n-1}}\int_
0^{\infty}\mathbf{1}_{F^-_f(\cdot,\xi,\ld,q)
}(t)t^{n-1}\,dt\,d\sigma(\xi)\right\|_{X^{\frac{1}{q}}}\noz\\\noz
&\quad=\frac{1}{n}\left\|\int_{\mathbb{S}^{n-1}}A^-_f
(\cdot,\xi,\ld,q)\,d\sigma(\xi)\right\|_{X^{\frac{1}{q}}}.
\end{align}
Notice that the function $\lambda \mapsto A^{-}(x,
\xi,\lambda)$ is increasing on $(0,\infty)$ and
$$\lim_{\lambda\to\infty}A^{-}_f(x,\xi,\lambda,q)=|
\xi\cdot\nabla f(x)|^q,$$
and that $K(q,n)$ is independent of $e.$ Using this and  Definition \ref{Debqfs}(iii), and
letting $\ld\to\infty$ in
\eqref{8025},
we then conclude
\begin{align}\label{11223}
	&\liminf_{\ld\to\infty}\ld^q\left\|\int_{\rn}
	\mathbf{1}_{\{y\in\rn:\ (\cdot,y)\in E_f(\ld,q)
	\}}(y)\,dy\right\|_{X^{\f 1 q}}\\
    &\quad\geq
	\frac{1}{n}\left\|\int_{\mathbb{S}^{n-1}}|
	\xi\cdot\nabla f|^q\,d\sigma(\xi)
	\right\|_{X^{\frac{1}{q}}}\noz\\\noz
	&\quad=	\frac{1}{n}\left\|\int_{\mathbb{S}^{n-1}}\left|
	\xi\cdot\frac{\nabla f}{|\nabla f|}\right|^q\,d\sigma(\xi)|\nabla f|^q
	\right\|_{X^{\frac{1}{q}}}=\frac{K(q,n)}{n}\|\,|\nabla f|\,\|_{X}^q.
\end{align}
This finishes the proof of \eqref{8002}.

Next, we prove that, for any $f\in C^2(\rn)$ with $|\nabla f|\in
C_{\mathrm{c}}(\rn)$,
\begin{align}\label{8066}
	\limsup_{\ld\to\infty}\ld^q\left\|
	\int_{\rn}\mathbf{1}_{\{y\in\rn:\ (\cdot,y)
	\in E_f(\ld,q)\}}(y)\,dy\right\|_{X^{
	\f 1 q}}\leq\frac{K(q,n)}{n}\|\,|\nabla f|\,\|_{X}^q.
\end{align}
To this end,
recall that, for any $\theta\in(0,1)$,
there exists a positive constant $C_{\theta}$
such that, for any $a,b\in(0,\infty)$,
\begin{align}\label{key11}
	(a+b)^q\leq (1+\theta) a^q+C_{\theta }b^q
\end{align}
(see, for instance, \cite[p.\,699]{ref7}).
Obviously, we have, for any $x\in\rn$ and
$\xi\in\mathbb{S}^{n-1},$
$$
|\xi\cdot\nabla f(x)|\leq \|\,|\nabla f|\,\|_{L^\infty(\rn)}<L_1.
$$
From this and \eqref{key11}, we deduce that, for any
$\lambda\in(L_1,\infty)$ sufficiently large,
\begin{align*}
A^+_f(x,\xi,\lambda,q)&=[|\xi\cdot\nabla f(x)|
+(L_1/\lambda)^{\frac{q}{n}}L_2]^q\\
&\leq(1+\theta)
|\xi\cdot\nabla f(x)|^q+C_{\theta}(L_1/
\lambda)^{\frac{q^2}{n}}L_2^q,
\end{align*}
Where $C_\theta$ is a positive constant depending only on $\theta.$
This, together with \eqref{8001}, \eqref{8007},
\eqref{8008}, and \eqref{key11} with $q$ replaced by $\frac{1}{q}$,
implies that, for any $\ld\in(L_1,\infty)$, $x\in\rn,$
and $\theta,\eta\in(0,1)$,
\begin{align*}
&\ld\left[\int_{\rn}\mathbf{1}_{\{y\in\rn:\
(x,y)\in E_f(\ld,q)\}}(y)\,dy\right]^{\f 1 q}\\
&\quad\leq
\ld\left[\int_{\mathbb{S}^{n-1}}\int_0^
{\infty}\mathbf{1}_{F^+_f(x,\xi,\ld,q)}(t)
t^{n-1}\,dt\,d\sigma(\xi)\right]^{\f 1 q}\mathbf{1}_{K}(x)\\\noz
&\quad=
n^{-\f 1 q} \left\{\int_{\mathbb{S}^{n-1}}\left[|\xi
\cdot\nabla f|+(L_1/\lambda)^{\frac{q}{n}}
L_2\right]^qd\sigma(\xi)\right\}^{\frac{1}{q}}\mathbf{1}_{K}(x)\\\noz
&\quad\leq
n^{-\f 1 q}
\left[\int_{\mathbb{S}^{n-1}}
(1+\theta)|\xi\cdot \nabla f|^q
+C_\theta(L_1/\ld)^{\frac{q^2}{n}}L^q_2
\,d\sigma(\xi)\right]^{\frac{1}{q}}\mathbf{1}_{K}(x)\\
&\quad\leq
n^{-\f 1 q}(1+\eta)(1+\theta)^{\frac{1}{q}}
\left[\int_{\mathbb{S}^{n-1}}
|\xi\cdot  \nabla f|^q\,d\sigma(\xi)\right]^{\frac{1}{q}}\mathbf{1}_{K}(x)\\\noz
&\quad\quad+
n^{-\f 1 q}C_{\eta}(L_1/\ld)^{\frac{q}{n}}
\left[C_\theta L^q_2\omega_n\right]^{\frac{1}{q}}\mathbf{1}_{K}(x)\\\noz
&\quad=
n^{-\f 1 q}(1+\eta)(1+\theta)^{\frac{1}{q}}[K(q,n)]^{\frac{1}{q}}
|\nabla f(x)|\mathbf{1}_K(x)\\\noz
&\quad\quad+
n^{-\f 1 q}C_{\eta}(L_1/\ld)^{\frac{q}{n}}
\left[C_\theta L^q_2\omega_n\right]^{\frac{1}{q}}\mathbf{1}_{K}(x),
\end{align*}
which, combined with
the assumption that $X$ is
a ball Banach function space, further implies that
\begin{align*}
&\ld\left\|\left[\int_{\rn}\mathbf{1}_{\{y\in\rn:\
(\cdot,y)\in E_f(\ld,q)\}}(y)\,dy\right]^{\f 1 q}\right\|_{X}\\
&\quad\leq
n^{-\f 1 q}(1+\eta)(1+\theta)^{\frac{1}{q}}[K(q,n)]^{\frac{1}{q}}
\|\,|\nabla f|\,\|_{X}\\\noz
&\qquad+
n^{-\f 1 q}C_{\eta}(L_1/\ld)^{\frac{q}{n}}
\left[C_\theta L^q_2\omega_n\right]^{\frac{1}{q}}\|\mathbf{1}_K\|_{X},
\end{align*}
where $C_\theta$ is a positive constant depending only on $\theta$
and $C_\eta$ is a positive constant depending only on $\eta$.
Letting $\ld\to\infty$ and $\theta,\eta \to 0,$  we then find that
\begin{align*}
\limsup_{\ld\to\infty}\ld\left\|\int_{\rn}\mathbf{1}_{\{y\in\rn:\
(\cdot,y)\in E_f(\ld,q)\}}(y)\,dy\right\|_{X^{\f 1 q}}^{\frac{1}{q}}\leq
\left[\frac{K(q,n)}{n}\right]^{\frac{1}{q}}
\|\,|\nabla f|\,\|_{X}.
\end{align*}
This proves \eqref{8066}, which,  combined with \eqref{11223}, then
completes the proof of Theorem \ref{theorem4.88}.
\end{proof}
Applying
Corollary \ref{density2} with $X:=L^p_{\omega}(\rn)$, we obtain the following corollary; we omit the details here.
\begin{corollary}\label{jiaquan2}
	Let $p\in[1,\infty)$ and $\omega\in A_p(\rn)$. Then, for any $f\in\dot{W}^{1,p}_{\omega}(\rn)$,
	there exists  $\{f_{k}\}_{k\in\nn}\subset C^\infty(\rn)$ with
$|\nabla f_k|\in C_{\rm{c}}(\rn)$ for any $k\in\nn$ such that
	\begin{align*}
		\lim_{k\to\infty}\|f-f_k\|_{\dot{W}^{1,p}_{\omega}(\rn)}
		=
		\lim_{k\to\infty}\|\,|\nabla f-\nabla f_k|\,\|_{L^p_{\omega}(\rn)}=0.
	\end{align*}
\end{corollary}
As a  consequence of Theorem
\ref{theorem4.88} with $X$ replaced by $L^p
_{\omega}(\rn)$, we have the following conclusion.
\begin{corollary}\label{corollary3.99}
Let $p\in[1,\infty),$ $q\in(0,\infty)$ with $n(\frac{1}{p}-\frac{1}{q})<1$, and $
\omega\in A_{1}(\rn)$. Then, for any $f\in
\dot{W}^{1,p}_\omega(\rn),$
\begin{align}\label{ruci}
&\lim_{\ld\to\infty}\ld^p \int_{\RR^n}\left
[\int_{ \RR^n} \mathbf{1}_{E_f(\ld,q)}(x, y)\,
dy \right]^{\f pq} \og(x) \,dx\\\noz
&\quad=
\left[\frac
{K(q,n)}{n}\right]^{\frac{p}{q}}	\int_{\RR^n}
| \nabla f(x)|^p \og(x) \, dx.
\end{align}
\end{corollary}
\begin{proof} Using Theorem \ref{theorem4.88} with $X$ replaced by $L^p_{\omega}(\rn)$,
we find that \eqref{ruci} holds true for any $f\in C^2(\rn)$ with $|\nabla f|\in
C_{\mathrm{c}}(\rn)$.
Next, we show that \eqref{ruci} hold
true for any $f\in
\dot{W}^{1,p}_\omega(\rn)$.

Let $f\in
\dot{W}^{1,p}_\omega(\rn)$ and $C_{(q,n)}:=[\frac
{K(q,n)}{n}]^{\frac{1}{q}}.$
From  Corollary \ref{jiaquan2}, we infer  that there exists a sequence $\{f_k\}_{k\in\nn}\subset\ C^\infty(\rn)$ with $
|\nabla f_k |\in C_{\mathrm{c}}(\rn)$ for any $k\in\nn$ such that
\begin{align*}
\lim_{k\to\infty}\|\,|\nabla f-\nabla f_k|\,\|_{L^p_\omega(\rn)}=0.
\end{align*}
By the H\"older inequality, \eqref{key11} with $q$ replaced by $\frac{1}{q},$
Theorem \ref{theorem4.88} with $f$ replaced by $f_k$, and
Theorem \ref{thm41} with $f$ replaced by $f-f_k$, we obtain
\begin{align*}
&\|\,|\nabla f|\,\|_{L^p_\omega(\rn)}
-\|\,|\nabla f-\nabla f_k|\,\|_{L^p_\omega(\rn)}\\
&\quad\leq
\|\,|\nabla f_k|\, \|_{L^p_\omega(\rn)}
=C_{(q,n)}^{-1}\liminf_{\ld\to\infty}\ld\left\{\int_{\RR^n}\left
[\int_{ \RR^n} \mathbf{1}_{E_{f_k}(\ld,q)}(x, y)\,
dy \right]^{\f pq} \og(x) \,dx\right\}^{\frac{1}{p}}\\
&\quad\leq
C_{(q,n)}^{-1}\liminf_{\ld\to\infty}\ld\left\{\int_{\RR^n}\left
[\int_{ \RR^n} \mathbf{1}_{E_{f_k-f}(\delta\ld,q)}(x, y)+\mathbf{1}_{E_{f}((1-\delta)\ld,q)}(x, y)\,
dy \right]^{\f pq} \og(x)\, dx\right\}^{\frac{1}{p}}\\
&\quad\leq
C_{(q,n)}^{-1}C_\theta\liminf_{\ld\to\infty}\ld\left\{\int_{\RR^n}\left
[\int_{ \RR^n} \mathbf{1}_{E_{f_k-f}(\delta\ld,q)}(x, y)\,
dy \right]^{\f pq} \og(x)\, dx\right\}^{\frac{1}{p}}\\
&\quad\quad+
C_{(q,n)}^{-1}(1+\theta)\liminf_{\ld\to\infty}\ld\left\{\int_{\RR^n}\left
[\int_{ \RR^n} \mathbf{1}_{E_{f}((1-\delta)\ld,q)}(x, y)\,
dy \right]^{\f pq} \og(x)\, dx\right\}^{\frac{1}{p}}\\
&\quad\le \left\{C_1C_{([\omega]_{A_1(\rn)})}\right\}^{1/p}
C_{(q,n)}^{-1}C_\theta\delta^{-1}\|\,|\nabla f_k-\nabla f|\,\|_{L^p_\omega(\rn)}\\
&\quad\quad+C_{(q,n)}^{-1}(1+\theta)(1-\delta)^{-1}\liminf_{\ld\to\infty}\ld\left\{\int_{\RR^n}\left
[\int_{ \RR^n} \mathbf{1}_{E_{f}(\ld,q)}(x, y)\,
dy \right]^{\f pq} \og(x)\, dx\right\}^{\frac{1}{p}},
\end{align*}
where $C_1$ is the same as in Theorem \ref{thm41}.
Let $k\to\infty$, $\theta\to0$,
and $\delta\to0.$  Then we obtain
\begin{align}\label{20221954}
&\liminf_{\ld\to\infty}\ld^p \int_{\RR^n}\left
[\int_{ \RR^n} \mathbf{1}_{E_f(\ld,q)}(x, y)\,
dy \right]^{\f pq} \og(x) \,dx\\\noz
&\quad\geq \left[\frac
{K(q,n)}{n}\right]^{\frac{p}{q}}\int_{\RR^n}
| \nabla f(x)|^p \og(x) \, dx.
\end{align}
Similarly, we also have
\begin{align*}
&\limsup_{\ld\to\infty}\ld^p \int_{\RR^n}\left
[\int_{ \RR^n} \mathbf{1}_{E_f(\ld,q)}(x, y)\,
dy \right]^{\f pq} \og(x)
\, dx\\\noz
&\quad\leq
\left[\frac
{K(q,n)}{n}\right]^{\frac{p}{q}}	\int_{\RR^n}
| \nabla f(x)|^p \og(x) \, dx,
\end{align*}
which, together with \eqref{20221954},
further implies that \eqref{ruci} holds true.
This then finishes the proof of Corollary \ref{corollary3.99}.
\end{proof}
As a consequence of Corollary \ref{corollary3.99},
we obtain the lower estimate of Theorem \ref{thm-4-1}.
Now, we can complete the proof of Theorem \ref{thm-4-1}.

\begin{proof}[Proof of Theorem \ref{thm-4-1}]
As a consequence of Theorem \ref{thm41} and Corollary \ref{corollary3.99},
we immediately obtain the desired conclusions of this theorem,
which completes the proof of Theorem \ref{thm-4-1}.	
\end{proof}

\subsection{Proof of Theorem \ref{thm-4-2} }\label{sec:4-3}

\begin{proof}[Proof of Theorem \ref{thm-4-2}]
Assume that there exists a positive constant $C_1$
such that, for any $f\in C^1(\RR)$ satisfying
that $f'$ has compact support, 		
\begin{equation}\label{4-1}
	\sup_{\ld\in(0,\infty)} 	\ld^p \int_{\RR^2}
	\mathbf{1}_{E_f(\ld,p)} (x, y) \og(x)\, dx\,dy\leq C_1
	\int_{\RR}| f'(x)|^p \og(x) \, dx,
\end{equation}
where, for any $\ld\in(0,\infty)$, $p\in[1,\infty)$, and any measurable function $f$,
$E_f(\ld,p)$ is the same as  in \eqref{E} and the constant $C_1$ is independent of $f.$
We now show that $\omega\in A_p(\mathbb{R})$ with $p\in[1,\infty)$.
Observe that the inequality \eqref{4-1} is both
dilation and translation invariance; that is,
for any $\da\in(0,\infty)$ and $x_0\in \RR$, both the
weights $\og(\da x)$ and $\og(x-x_0)$ satisfy
\eqref{4-1} with the same constant $C_1$.
This, combined with Lemma \ref{Lemma2.1}(ii), further
implies that, to prove $\omega\in A_p(\rr)$, it
suffices to show that there exists a positive
constant $C$, depending only on $C_1$, such that,
for any nonnegative function $g\in L_{\text{loc}}^1(\RR)$,
\begin{equation}\label{4-2}
	\left[ \int_{-1}^1 g(x)\, dx \right]^p \leq \f
	C {\og([-1,1])} \int_{-1}^1 |g(x)|^p \og(x) \,dx.
\end{equation}

To show \eqref{4-2}, we first prove that, for any
$0\leq g\in C^\infty (\RR)$,
\begin{equation}\label{4-3}
	\left[ \int_{-1}^1 g(x)\, dx \right]^p \leq
	\frac{C_16^{p+1}}{4{\og(I_0)}}
	\int_{-4}^4 |g(x)|^p \og(x) \,dx,
\end{equation}
where $I_0:=[-3,-1]\cup[1,3]$.
Let $\eta\in C^\infty(\RR)$ be such that $\eta(x)\in[0,1]$
for any $x\in\rr,$
$\eta(x)=1$
for any $x\in[-3,3]$, and $\eta(x)=0$ for any
$x\in\mathbb{R}$ with $|x|\in[4,\infty)$. For any
$x\in\mathbb{R},$ let
$$f(x):= \int_{-\infty}^x g(t) \eta(t)\, dt.$$
Clearly, $f\in C^\infty(\RR)$ and $\supp
\,(f') \subset [-4,4]$.
Let $\ld:=6^{-1-\f 1p} \int_{-1}^{1} g(t)\, dt.$
Then, for any $x\in[-3,-1]$ and $ y\in[1,3]$, we have
\begin{align*}
	|f(y)-f(x)|&=\int_x^y g(t)\, dt \ge \int_{-1}
	^{1} g(t)\, dt=6^{1+\f 1p} \ld\ge \ld|x-y|^{\f1p+1}.
\end{align*}
This, together with the symmetry further implies that
\[\lf( [1,3] \times [-3,-1])\cup( [-3,-1]\times
[1,3]\r)\subset E_f(\ld,p). \]
Thus, using this and \eqref{4-1}, we conclude that
\begin{align*}
	4\ld^p \int_{I_0} \og(x)\, dx
	&\leq \ld^p\int_{\RR^2} \og(x) \mathbf{1}_{E_f(\ld,p)} (x,y)
	\, dx\,dy\\
	&\leq C_1 \int_{\RR} |f'(x)|^p \og(x)\,
	dx\leq C_1 \int_{-4}^4 |g(x)|^p \og(x)\, dx.
\end{align*}
This proves \eqref{4-3}.

Second, we show that, for any nonnegative locally
integrable function $g$,
\begin{equation}\label{4-4}
	\left[ \int_{-1}^1 g(x)\, dx \right]^p
	\leq \frac{C_16^{p+1}}{4{\og(I_0)}}
	\int_{-1}^1 |g(x)|^p \og(x) \,dx.
\end{equation}
Without loss of generality, we may assume that
$g$ is bounded because, otherwise, one may
replace $g$ by $\min\{g, n\}$ for any
$n\in\mathbb{N},$ and then apply
the monotone convergence theorem
to obtain the desired inequality.
Let $\vi\in C^\infty(\RR)$ be such that
$\vi(t)\ge 0$ for any $t\in \RR$, $\vi(t)=0$
for any $t\in\mathbb{R}$ with $|t|\geq1$, and
$\int_{\RR} \vi(t)\, dt =1$. For any $\va\in(0,\infty)$
and $t\in\mathbb{R}$, let
$\vi_\va(t):=\va^{-1} \vi(t/\va)$ and
\[ g_\va (t):=(g\mathbf{1}_{[-1,1]})\ast \vi_\va(t)
=\int_{-1}^1 g(u) \vi_\va (t-u)\, du.\]
Then $0\leq g_\va\in C_{\mathrm{c}}^\infty(\RR)$
and, using \eqref{4-3}, we obtain
\begin{align}\label{2.41x}
	\left[ \int_{-1}^1 g_\va(x)\, dx \right]^p \lesssim
	\f 1 {\og(I_0)} \int_{-4}^4 |g_\va(x)|^p \og(x) \,dx.
\end{align}
Since, for almost every $t\in\mathbb{R},$
$$\lim_{\va\to 0} g_\va (t)=g(t)\mathbf{1}_{[-1,1]}(t)$$
and
$$\sup_{\va\in(0,\infty)} \|g_\va\|_{L^\infty(\rr)}\leq
\sup_{\va\in(0,\infty)} \|\vi_\va
\|_{L^1(\rr)}\|g\|_{L^\infty(\rr)}<\infty,$$
from these, \eqref{2.41x}, and the
Lebesgue dominated convergence theorem, we deduce \eqref{4-4}.

Finally, we prove \eqref{4-2}. Let $g:=1$ and
$I:=[-1,1]$. By \eqref{4-4}, we find that
$\f {\og(I)} {\og(I_0)}\ge c_1$,
where $c_1:=\frac{2}{C_1 3^{p+1}}.$
Thus,
\[ \og(2I) \leq \og(I_0) +\og(I) \leq (1+1/c_1) \og(I).\]
Since \eqref{4-1} is both dilation and translation
invariance for the weight $\og$, it follows
that the inequality $\og(2I) \leq (1+1/c_1)
\og(I)$ holds true for any compact interval $I\subset\RR$.
By this, we find that
\begin{align*}
\omega([-1,1])
&\leq\omega([-4,0])+\omega([0,4])\\
&\leq (1+1/{c_1})\{\omega([-3,-1])+\omega([1,3])\}\\
&=(1+1/{c_1})\omega(I_0).
\end{align*}
This, combined with \eqref{4-4},
further implies that \eqref{4-2} holds true. This then finishes
the proof of Theorem \ref{thm-4-2}.
\end{proof}

\section{Estimates in Ball Banach Function Spaces}\label{sec:5}

In this section, we establish the Brezis--Van
Schaftingen--Yung formulae in ball Banach
function spaces (see Theorem \ref{theorem3.9} below).
As applications, we also obtain some fractional Sobolev and
Gagliardo--Nirenberg type inequalities in ball Banach
function spaces.

We begin with introducing the following concepts of homogeneous
(weak) Triebel--Lizorkin-type spaces.
\begin{definition}\label{np}
	Let $q\in(0,\infty),$ $s\in[0,\infty),$ and $X$
	be a ball quasi-Banach function space.
	\begin{itemize}
		\item[{\rm(i)}] The
		\emph{homogeneous Triebel--Lizorkin-type space} $\dot{\mathbf{F}}^s_{X,q}(\rn)$ is
		defined to be the set of all the measurable
		functions $f$ on $\rn$
		such that
		$$\|f\|_{\dot{\mathbf{F}}^s_{X,q}(\rn)}
		:=\left\|\left[\int_{\rn}\frac{|f(\cdot)-
			f(y)|^q}{|\cdot-y|^{n+sq}}\,dy\right]^{\frac{1}{q}}\right\|_X<\infty.$$
		
		\item[{\rm(ii)}] The \emph{homogeneous weak Triebel--Lizorkin-type space}
$\mathbf{W\dot{F}}^s_{X,q}(\rn)$ is
		defined to be the set of all the measurable functions $f$ on $\rn$
		such that
		$$\|f\|_{\mathbf{W\dot{F}}^s_{X,q}(\rn)}
		:=\sup_{\lambda\in(0,\infty)}\lambda\left\|\left[
		\int_{\rn}\mathbf{1}_{E_f(\ld,q,s)}(
		\cdot,y)\,dy\right]^{\frac{1}{q}}\right\|_X<\infty,$$
		where, for any $\ld\in(0,\infty),$
		\begin{align}\label{keydef}
			E_{f}(\ld,q,s):=\left\{
			(x,y)\in\rn\times\rn:\ |f(x)-f(y)|>\ld|x-y|^{\f nq+s}\right\}.
		\end{align}
	\end{itemize}
\end{definition}
Similarly to Brezis et al. \cite[(1.2)]{ref8} (see also \cite{bbm,ref7}),
we have the following conclusions on the
``drawback" of $\dot{\mathbf{F}}^s_{X,q}(\rn)$.

\begin{theorem}\label{ss1}
Let $X$ be a ball quasi-Banach function space and
$s,q\in(0,\infty)$. Assume
that $X^{\frac{1}{q}}$ is a ball Banach
function space.
If $s\min\{1,q\}\in[1,\infty)$ and
$f\in\dot{\mathbf{F}}^s_{X,q}(\rn)$, then $f$ is a constant function.
\end{theorem}

\begin{proof}
Let  $s,q\in(0,\infty)$
and $f\in\dot{\mathbf{F}}^s_{X,q}(\rn).$
By Lemma \ref{lemma3.4} and Definition \ref{def-X'}, we have
\begin{align}\label{mq}
\infty&>\left\|\int_{\rn}\frac{|f(\cdot)-f(y)|^q}
{|\cdot-y|^{n+sq}}\,dy\right
\|_{X^{1/q}}=\left\|
\int_{\rn}\frac{|f(\cdot)-f(y)|^q}
{|\cdot-y|^{n+sq}}\,dy\right\|_{(X^{1/q})''}\\\noz
&=\sup_{\|g\|_{(X^{1/q})'}=1}\int_{\rn}\int_{\rn}
\frac{|f(x)-f(y)|^q}{|x-y|^{n+sq}}g(x)\,dy\,dx\\\noz
&=\sup_{\|g\|_{(X^{1/q})'}=1}\int_{\rn}\int_{\rn}
\frac{|f(x)-f(x-h)|^q}{|h|^{n+sq}}g(x)\,dh\,dx.
\end{align}
For any $N\in(0,\fz)$,  let $g:=
\mathbf{1}_{B(\mathbf{0}, N)}/\|\mathbf{1}_{B(\mathbf{0}, N)}\|_{(X^{1/q})'}.$
Using \eqref{mq},
we conclude that, for any  $N\in(0,\infty)$,
\begin{align*}
\int_{|x|< N}\int_{\rn}\frac{|f(x)-f(x-h)|^q}{|h|^{n+sq}}
\,dh\,dx<\infty.
\end{align*}
From this, we deduce that,
for any $N\in(0,\infty)$
and $r\in(0,N)$,
\begin{align}\label{hgf22}
\infty&>\int_{|x|<
N}\int_{|h|< r}\frac{|f(x)-f(x-h)|^q}{|h|^{n+sq}}
\,dh\,dx\\\noz
&\geq\sum_{j=0}^{\infty}2^{j(n+sq)}r^{-(n+sq)}
\int_{2^{-(j+1)}r\le|h|<2^{-j}
r}\int_{|x|< N}
|f(x)-f(x-h)|^q\,dx\,dh.
\end{align}
We prove the present theorem by considering the following
two cases on both $s$ and $q$.

Case 1) $q\in[1,\infty)$ and $s\in[1,\infty)$.
Recall the discrete H\"older inequality that,
for any $m\in\zz_+$ and $\{a_j\}_{j=1}^m\subset(0,\infty),$
\begin{align}\label{lsholder}
	\left(\sum_{j=1}^{m}a_j\right)^q\leq m^{q-1}\left(\sum_{j=1}^ma_j^q\right).
\end{align}
By this, we obtain, for any $j\in\zz_+$,
\begin{align}\label{jiechu}
&\int_{2^{-1}r\le|h|<
r}\int_{|x|<N-r}
|f(x)-f(x-h)|^q
\,dx\,dh\\\noz
&\quad=2^{jn}
\int_{2^{-(j+1)}r\le|h|<2^{-j}
r}\int_{|x|<N-r}
|f(x)-f(x-2^jh)|^q
\,dx\,dh\\\noz
&\quad=2^{jn}
\int_{2^{-(j+1)}r\le|h|<2^{-j}
r}\int_{|x|<N-r}
\left|\sum_{i=0}^{2^j-1}[f(x-ih)-f(x-(i+1)h)]\right|^q
\,dx\,dh\\\noz
&\quad\lesssim2^{jn+jq-j}
\sum_{i=0}^{2^j-1}
\int_{2^{-(j+1)}r\le|h|<2^{-j}
r}\int_{|x|<N-r}
|f(x-ih)-f(x-(i+1)h)|^q
\,dx\,dh\\\noz
&\quad\lesssim2^{jn+jq-j}
\sum_{i=0}^{2^j-1}
\int_{2^{-(j+1)}r\le|h|<2^{-j}
r}\int_{|x|<N-r+i2^{-j}r}
|f(x)-f(x-h)|^q
\,dx\,dh\\\noz
&\quad\lesssim2^{jn+jq}
\int_{2^{-(j+1)}r\le|h|<2^{-j}
r}\int_{|x|< N}
|f(x)-f(x-h)|^q
\,dx\,dh,
\end{align}
which, combined with \eqref{hgf22},
further implies that, for any $N\in(0,\infty)$
and $r\in(0,N)$,
\begin{align*}
\sum_{j=0}^{\infty}2^{j(s-1)q}
\int_{2^{-1}r\le|h|<
r}\int_{|x|< N-r}
|f(x)-f(x-h)|^q
\,dx\,dh<\infty.
\end{align*}
From this and $(s-1)q\in[0,\infty)$,
we deduce that,
for any $N\in(0,\infty)$
and $r\in(0,N)$,
\begin{align*}
\int_{2^{-1}r\le|h|<r}\int_{|x|<N-r}
|f(x)-f(x-h)|^q
\,dx\,dh=0.
\end{align*}
Using this and letting $N\to\infty,$ we then obtain, for any $r\in(0,\infty),$
$$\int_{2^{-1}r\le|h|<r}\int_{\rn}
|f(x)-f(x-h)|^q
\,dx\,dh=0.$$
By this, we further conclude that
$$\int_{\rn}\int_{\rn}
|f(x)-f(x-h)|^q
\,dx\,dh=0,$$
which  implies that
$f$ is a constant
function on $\rn$. This finishes
the proof of the present theorem in this case.

Case 2) $q\in(0,1)$
and $sq\in[1,\infty)$. By an argument
similar to that used in the proof of \eqref{jiechu} with \eqref{lsholder}
replaced by \eqref{ppp}, we find that, for any $j\in\zz_+$,
\begin{align*}
	&\int_{2^{-1}r\le|h|<
		r}\int_{|x|<N-r}
	|f(x)-f(x-h)|^q
	\,dx\,dh\\\noz
	&\quad\lesssim2^{jn+j}
	\int_{2^{-(j+1)}r\le|h|<2^{-j}
		r}\int_{|x|< N}
	|f(x)-f(x-h)|^q
	\,dx\,dh,
\end{align*}
which, combined with \eqref{hgf22},
implies that, for any $N\in(0,\infty)$
and $r\in(0,N)$,
\begin{align*}
	\sum_{j=0}^{\infty}2^{j(sq-1)}
	\int_{2^{-1}r\le|h|<
		r}\int_{|x|<N-r}
	|f(x)-f(x-h)|^q
	\,dx\,dh<\infty.
\end{align*}
From this and $sq-1\in[0,\infty)$,
we deduce that,
for any $r\in(0,\infty)$,
\begin{align*}
	\int_{2^{-1}r\le|h|<
		r}\int_{\rn}
	|f(x)-f(x-h)|^q
	\,dx\,dh=0.
\end{align*}
By this and  the arbitrariness of $r$, we conclude that
$$\int_{\rn}\int_{\rn}
|f(x)-f(x-h)|^q
\,dx\,dh=0,$$
which further implies that
$f$ is a constant
function on $\rn$. This then finishes
the proof of   Theorem \ref{ss1}.
\end{proof}

\begin{remark} In Theorem \ref{ss1}, if $X:=L^p(\rn)$, $0<q\leq p<\infty$,
and $s\min\{1,q\}\in[1,\infty)$, then the conclusions of 
Theorem \ref{ss1} hold true with $X$ replaced by $L^p(\rn)$,
which, when $p=q\in[1,\infty)$ and $s=1$,
coincide with \cite[(1.2)]{ref8} (see also \cite{bbm,ref7}).
\end{remark}

One of the main targets in this section is
to prove the equivalence \eqref{1-13} in a
ball Banach function space $X$ under
some mild assumptions on both $X$ and  $p$.
Theorem \ref{ss1} justifies the use of the
semi-norm $ \|f\|_{\mathbf{W\dot{F}}^1_{X,q}(\rn)}$
instead of $ \|f\|_{\dot{\mathbf{F}}^1_{X,q}(\rn)}$
in the equivalence \eqref{1-13} as follows.

\begin{theorem}\label{theorem3.9}
Let $p\in[1,\infty)$ and $q\in(0,\infty)$
satisfy $n(\f 1p-\f1q) <1$. Assume that
$X$ is a ball Banach function space,
$X^{\f 1 p}$  a ball Banach function space,
and $\cm$  in \eqref{2-4-c}  bounded
on its associate space $(X^{1/p})'.$
\begin{enumerate}
\item[{\rm(i)}]
Then there exist positive constants $C_2$ and $C_{(\|\cm\|_{(X^{1/p})'
\to(X^{1/p})'})}$ such that, for
any $f\in C^{2}(\rn)$ with $|\nabla f|\in C_{\mathrm{c}}(\rn)$,
\begin{align}\label{1011b}
\left[\frac{K(q,n)}{n}\right]^{\f 1 q}\|\,|\nabla f|\,\|_X
\leq
\|f\|_{\mathbf{W\dot{F}}^1_{X,q}(\rn)}\leq C_2C_{(\|\cm\|_{(X^{1/p})'
\to(X^{1/p})'})}\|\,|\nabla f|\,\|_X,
\end{align}
where $K(q,n)$ is the same as in \eqref{kqn},
the positive constant $C_2$ depends only on $p,q$
and $n$,  and the positive constant $C_{(\|\cm\|_{(X^{1/p})'\to(X^{1/p})'})}$,
depending only on $\|\cm\|_{(X^{1/p})'\to(X^{1/p})'},p$, and $q$, increases as
$\|\cm\|_{(X^{1/p})'\to(X^{1/p})'}$ increases, and $C_{(\cdot)}$ is continuous on $(0,\infty)$. Moreover,
for any $f\in C^{2}(\rn)$ with $|\nabla f|\in C_{\mathrm{c}}(\rn)$,
\begin{align}\label{leiliu00}
	\lim_{\ld\to\infty}\ld^q\left\|\int_{\rn}
	\mathbf{1}_{E_f(\ld,q)}
	(\cdot,y)\,dy\right\|_{X^{\f 1 q}}=\frac{K(q,n)}{n}
	\|\,|\nabla f|\,\|_{X}^q,
\end{align}
where $E_f(\ld,q)$ for any $\lz\in (0,\fz)$ is the same as  in \eqref{E}.
\item[{\rm(ii)}]
Assume that $X$ has an absolutely
continuous norm. Then both \eqref{1011b}
and \eqref{leiliu00}
hold true for any $f\in \dot{W}^{1,X}(\rn
).$
\end{enumerate}
\end{theorem}
To prove Theorem \ref{theorem3.9}, we need the following conclusion whose proof is
a slight modification of \cite[p.\,18]{ref6}
via replacing $L^p_\omega(\rn)$ in \cite[p.\,18]{ref6} by $X$;
we omit the details.

\begin{lemma}\label{lemma3.7}
	Let $X$ be a ball Banach function space.
	Assume that the Hardy--Littlewood maximal
	operator $\cm$ is bounded on $X.$ For any
	$g\in X$ and $x\in\rn$, let
	\begin{align*} R_Xg(x):=\sum_{k=0}^{\infty}
		\frac{\cm^kg(x)}{2^k\|\cm\|^k_{X\to X}},
	\end{align*}
	where, for any $k\in\mathbb{N},\,\cm^k
	:=\cm\circ\cdots\circ\cm$ is $k$
	iterations
	of the Hardy--Littlewood maximal operator,
	and $\cm^0g(x):=|g(x)|.$ Then, for any
	$g\in X$ and $x\in\rn$,
	\begin{itemize}
		\item[{\rm(i)}] $|g(x)|\leq R_Xg(x);$
		\item[{\rm(ii)}] $R_Xg\in A_1(\rn)$ and $[R_Xg]_
		{A_1(\rn)}\leq2\|\cm\|_{X\to X},$ where $\|\cm\|_{X\to X}$ denotes the operator norm of
		$\cm$ mapping $X$ to $X$;
		\item[{\rm(iii)}] $\|R_Xg\|_X\leq2\|g\|_X.$
	\end{itemize}
\end{lemma}
As a consequence of Lemma \ref{lemma3.7}, we have the following conclusion.
\begin{lemma}\label{biaoda}
	Let $X$ be a ball Banach function space,
	$p\in[1,\infty),$ and $Y:=X^{\frac{1}{p}}$.
	Assume that
	$Y$ is a ball Banach function space
	and $\cm$  in \eqref{2-4-c}  bounded
	on its associate space $Y'.$
	Then, for any $f\in X$,
	\begin{align*}
		\|f\|_{X}
		\leq
		\sup_{\|g\|_{Y'}\leq 1}
		\left[\int_{\rn}|f(x)|^pR_{Y'}g(x)\,dx\right]^{\frac{1}{p}}
		\leq
		2^{\frac{1}{p}}\|f\|_{X}.
	\end{align*}
\end{lemma}
\begin{proof}
By Lemmas \ref{lemma3.4}, \ref{lemma3.5}, and \ref{lemma3.7}, we find that, for any
	$f\in X,$
	\begin{align*}
		\|f\|_{X}^p
		&=\|\,|f|^p\|_{Y}
		=\sup_{\|g\|_{Y'}\leq 1}\int_{\rn}|f(x)|^p|g(x)|\,dx\\
		&\leq
		\sup_{\|g\|_{Y'}\leq 1}\int_{\rn}|f(x)|^pR_{Y'}g(x)\,dx\\
		&\leq
		\sup_{\|g\|_{Y'}\leq 1}
		\|\,|f|^p\|_{Y}\|R_{Y'}g\|_{Y'}
		\leq
		2\|f\|_{X}^p.
	\end{align*}
	This finishes the proof of Lemma \ref{biaoda}.
\end{proof}
\begin{proof}[Proof of Theorem \ref{theorem3.9}]
We first prove that, for any $f\in \dot{W}^{1,X}(\rn)$,
\begin{align}\label{duicuo}
\sup_{\lambda\in(0,\infty)}\lambda\left\|\left[
\int_{\rn}\mathbf{1}_{E_f(\ld,q)}(
\cdot,y)\,dy\right]^{\frac{1}{q}}\right\|_X\leq C_2C_{(\|\cm\|_{(X^{1/p})'
\to(X^{1/p})'})}\|\,|\nabla f|\,\|_X,
\end{align}
where the positive constant
$C_2$ depends only on $p,q$
and $n$,   the positive constant $C_{(\|\cm\|_{(X^{1/p})'\to(X^{1/p})'})}$,
depending only on $\|\cm\|_{(X^{1/p})'\to(X^{1/p})'},p$, and $q$, increases as
$\|\cm\|_{(X^{1/p})'\to(X^{1/p})'}$ increases, and $C_{(\cdot)}$ is continuous on $(0,\infty)$.

Let $Y:=X^{\f 1 p}$.
Then both $Y$ and $Y'$ are ball Banach function spaces.
By Lemma \ref{biaoda}, we have, for any $\ld\in(0,\infty)$,
\begin{align}\label{1012}
&\left\|\left[\int_{\rn}\mathbf{1}_
{E_f(\ld,q)}(\cdot,y)\,dy\right]^{\frac{1}{q}}\right\|_X\\\noz
&\quad\leq
\sup_{\|g\|_{Y'}\leq 1}\left\{\int_{\rn}
\left[\int_{\rn} \mathbf{1}_{E_f(\ld,q)} (x,y)
\, dy\right]^{\f pq} R_{Y'}g(x)\,dx\right\}^{\f 1 p}.
\end{align}
Let $g\in Y'$ with $\|g\|_{Y'}\leq 1$.
Using both Lemma \ref{lemma3.7}(ii) and
Theorem \ref{thm-4-1} with $\og$ replaced by
$R_{Y'}g$, we find that there exist positive
constants $C$ and $C_{([R_{Y'}g]_{A_1(\rn)})}$
such that, for any $\ld\in(0,\infty)$,
\begin{align*}
&\lambda\left
\{\int_{\rn} \left[ \int_{\rn} \mathbf
{1}_{E_f(\ld,q)} (x,y)\, dy\right]^{\f pq}
R_{Y'}g(x)\,dx\right\}^{\f 1 p}\\\noz
&\quad \leq
CC_{([R_{Y'}g]_{A_1(\rn)})}
\left[\int_{\rn}|\nabla f(x)|^p
R_{Y'}g(x)\,dx\right]^{\f 1 p}\\\noz
&\quad\leq
CC_{(2\|\cm\|_{Y'\to Y'})}
\left[\int_{\rn}|\nabla f(x)|^p
R_{Y'}g(x)\,dx\right]^{\f 1 p},
\end{align*}
where $E_f(\ld,q)$  is the same as  in \eqref{E},
the positive constant $C_{(2\|\cm\|_{Y'\to Y'})}$
increases as $\|\cm\|_{Y'\to Y'}$ increases, $C_{(\cdot)}$ is continuous on $(0,\infty),$ and the
positive constant $C$ depends only on $p,q$,
and $n$.
From this, \eqref{1012}, and Lemma \ref{biaoda} again, we deduce that,
for any $\ld\in(0,\infty)$,
\begin{align*}
&\ld\left\|\left[\int_{\rn}\mathbf{1}_
{E_f(\ld,q)}(\cdot,y)\,dy\right]^{\frac{1}{q}}\right\|_X\\\noz
&\quad\leq
CC_{(2\|\cm\|_{Y'\to Y'})}
\sup_{\|g\|_{Y'}\leq 1}
\left[\int_{\rn}|\nabla f(x)|^p
R_{Y'}g(x)\,dx\right]^{\f 1 p}\\\noz
&\quad\leq
CC_{(2\|\cm\|_{Y'\to Y'})}2^{\f 1 p}
\|\,|\nabla f|\,\|_{X}.
\end{align*}
This proves \eqref{duicuo}.

By both \eqref{8002} and \eqref{duicuo}, we conclude that both \eqref{1011b} and
\eqref{leiliu00} hold true for any
$f\in C^2(\rn)$ with $|\nabla f|\in C_{\mathrm
{c}}(\rn)$. This proves (i).

Next, we prove (ii). Assume that
$X$ has an absolutely continuous norm.
By \eqref{duicuo}, we find that, to
show (ii),
it suffices to prove that \eqref{leiliu00} holds true for any $f\in\dot{W}^{1,X}(\rn)$. We first show that,
for any
$f\in \dot{W}^{1,X}(\rn)$,
\begin{align}\label{dui}
\liminf_{\ld\to\infty}\ld^q\left\|\int_{\rn}
\mathbf{1}_{E_f(\ld,q)}
(\cdot,y)\,dy\right\|_{X^{\f 1 q}}
\geq
\frac{K(q,n)}{n}
\|\,|\nabla f|\,\|_{X}^q.
\end{align}

Let $f\in \dot{W}^{1,X}(\rn)$. Using
Corollary \ref{density2}, we find that
there exists a sequence $\{f_k\}_{k\in\nn}\subset
C^\infty(\rn)$ with $|\nabla f_k|\in C_{\mathrm{c}}(\rn)$ for any $k\in\nn$ such that
\begin{align}\label{bijin}
\lim_{k\to\infty}\|\,|\nabla f-\nabla f_k|\,\|_X=0.
\end{align}
By  the assumption that $X$
is a ball Banach function space, we have
\begin{align}\label{dong}
\|\,|\nabla f|\,\|_{X}
\leq
\|\,|\nabla f_k|\,\|_{X}+\|\,|\nabla f-\nabla f_k|\,\|_{X}.
\end{align}
Let $C_{(q,n)}:=[\frac{K(q,n)}{n}]^{\frac{1}{q}}.$
From \eqref{yzr},
\eqref{key11} with $q$ replaced by $\f 1 q$,  the assumption that $X$
is a ball Banach function space,
and \eqref{duicuo},
we deduce that,
for any $\eta,\delta\in(0,1)$,
\begin{align*}
\|\,|\nabla f_k|\,\|_X^q
&=
C_{(q,n)}^{-1}\liminf_{\ld\to\infty}
\ld\left\|\left[\int_{\rn}
\mathbf{1}_{E_{f_k}(\ld,q)}
(\cdot,y)\,dy\right]^{\frac{1}{q}}\right\|_{X}\\\noz
&\leq
C_{(q,n)}^{-1}
\liminf_{\ld\to\infty}
\ld\left\|\left[\int_{\rn}
\mathbf{1}_{E_{f_k-f}((1-\delta)\ld,q)}
(\cdot,y)+
\mathbf{1}_{E_{f}(\delta\ld,q)}
(\cdot,y)
\,dy\right]^{\frac{1}{q}}\right\|_{X}\\\noz
&\leq
C_\eta C_{(q,n)}^{-1}\liminf_{\ld\to\infty}\ld
\left\|\left[\int_{\rn}
\mathbf{1}_{E_{f_k-f}((1-\delta)\ld,q)}
(\cdot,y)\,dy\right]^{\frac{1}{q}}\right\|_{X}\\\noz
&\quad+
(1+\eta) C_{(q,n)}^{-1}\liminf_{\ld\to\infty}\ld
\left\|\left[\int_{\rn}
\mathbf{1}_{E_{f}(\delta\ld,q)}
(\cdot,y)\,dy\right]^{\frac{1}{q}}\right\|_{X}\\\noz
&\leq
C_\eta(1-\delta)^{-1}C_{(q,n)}^{-1}\sup_{\ld\in(0,\infty)}\ld
\left\|\left[\int_{\rn}
\mathbf{1}_{E_{f_k-f}(\ld,q)}
(\cdot,y)\,dy\right]^{\frac{1}{q}}\right\|_{X}\\\noz
&\quad+
(1+\eta)\delta^{-1} C_{(q,n)}^{-1}\liminf_{\ld\to\infty}\ld
\left\|\left[\int_{\rn}
\mathbf{1}_{E_{f}(\ld,q)}
(\cdot,y)\,dy\right]^{\frac{1}{q}}\right\|_{X}\\\noz
&\leq
c_0
C_\eta(1-\delta)^{-1}C_{(q,n)}^{-1}
\|\,|\nabla f-\nabla f_k|\,\|_X\\\noz
&\quad
+(1+\eta)\delta^{-1} C_{(q,n)}^{-1}\liminf_{\ld\to\infty}\ld
\left\|\left[\int_{\rn}
\mathbf{1}_{E_{f}(\ld,q)}
(\cdot,y)\,dy\right]^{\frac{1}{q}}\right\|_{X},
\end{align*}
where $c_0:=C_2C_{(\|\cm\|_{(X^{1/p})'\to(X^{1/p})'})}$ and
$C_\eta$ is a positive constant
depending only on $\eta.$
Using this and \eqref{dong}, and letting $k\to\infty$, we conclude that
\begin{align*}
\|\,|\nabla f|\,\|_{X}
\leq
(1+\eta)\delta C_{(q,n)}^{-1}\liminf_{\ld\to\infty}\ld
\left\|\left[\int_{\rn}
\mathbf{1}_{E_{f}(\ld,q)}
(\cdot,y)\,dy\right]^{\frac{1}{q}}\right\|_{X}.
\end{align*}
Let  $\eta\to0$ and $\delta\to1.$  We then  prove \eqref{dui}.

Finally, we
show that
\begin{align}\label{zuihou}
\limsup_{\ld\to\infty}\ld^q\left\|\int_{\rn}
\mathbf{1}_{E_f(\ld,q)}
(\cdot,y)\,dy\right\|_{X^{\f 1 q}}
\leq
\frac{K(q,n)}{n}
\|\,|\nabla f|\,\|_{X}^q.
\end{align}
By \eqref{yzr}, we find that
\eqref{zuihou} holds true for any $f\in C^2(\rn)$ with $|\nabla f|\in C_{\mathrm{c}}(\rn)$.
Using \eqref{zuihou} with $f:=f_k$, \eqref{duicuo} with $f:=f-f_k$,
 \eqref{key11}, and the assumption
that $X$ is a ball Banach function space,
we have, for any $k\in\nn$,
\begin{align*}
&\limsup_{\ld\to\infty}\ld\left\|
\left[\int_{\rn}\mathbf{1}_{E_f(\ld,q)}
(\cdot,y)\,dy\right]^{\frac{1}{q}}\right\|_{X}\\\noz
&\quad\leq
C_\theta\limsup_{\ld\to\infty}\ld\left\|
\left[\int_{\rn}\mathbf{1}_{E_{f-f_k}((1-\delta)\ld,q)}
(\cdot,y)\,dy\right]^{\frac{1}{q}}\right\|_{X}\\\noz
&\quad\quad+
(1+\theta)\limsup_{\ld\to\infty}\ld\left\|
\left[\int_{\rn}\mathbf{1}_{E_{f_k}(\delta\ld,q)}
(\cdot,y)\,dy\right]^{\frac{1}{q}}\right\|_{X}\\\noz
&\quad\leq
C_\theta(1-\delta)^{-1}\sup_{\ld\in(0,\infty)}\ld\left\|
\left[\int_{\rn}\mathbf{1}_{E_{f-f_k}(\ld,q)}
(\cdot,y)\,dy\right]^{\frac{1}{q}}\right\|_{X}\\\noz
&\quad\quad+
(1+\theta)\delta^{-1}\limsup_{\ld\to\infty}\ld\left\|
\left[\int_{\rn}\mathbf{1}_{E_{f_k}(\ld,q)}
(\cdot,y)\,dy\right]^{\frac{1}{q}}\right\|_{X}\\\noz
&\quad\leq
C_\theta(1-\delta)^{-1}c_0
\|\,|\nabla f-\nabla f_k|\,\|_{X}+
(1+\theta)\delta^{-1}C_{(q,n)}
\|\,|\nabla f_k|\,\|_{X},
\end{align*}
where $c_0:=C_2C_{(\|\cm\|_{(X^{1/p})'\to(X^{1/p})'})}$ and $C_\theta$ is a positive
constant depending only on $\theta.$
Letting $k\to\infty,$ $\theta\to0$,
and $\delta\to1$, and using \eqref{bijin},
we conclude that
\begin{align*}
\limsup_{\ld\to\infty}\ld\left\|
\left[\int_{\rn}\mathbf{1}_{E_f(\ld,q)}
(\cdot,y)\,dy\right]^{\frac{1}{q}}\right\|_{X}
\leq
C_{(q,n)}
\|\,|\nabla f|\,\|_{X}.
\end{align*}
Thus, \eqref{zuihou} holds true,
which, combined with \eqref{dui}, further implies that \eqref{leiliu00} holds true for any $f\in\dot{W}^{1,X}(\rn)$.
This finishes the proof of Theorem
\ref{theorem3.9}.
\end{proof}

\begin{remark}
\begin{itemize}
\item[{\rm(i)}] In Theorem \ref{theorem3.9}(ii), if $X:=L^p(\rn)$ with
$p\in[1,\infty)$ and if $q\in(0,\infty)$ satisfies $n(\frac{1}{p}-\frac{1}{q})<1$,
then the conclusions of Theorem \ref{theorem3.9}(ii)
hold true with $X$ replaced by
$L^p(\rn)$, which, when $p=q\in [1,\fz)$,
is just \eqref{bsy} in
\cite{ref8,BSSY21}.

\item[{\rm(ii)}]
Let $p\in[1,\infty)$ and
$X$ be the same as in Theorem \ref{theorem3.9}(ii).
Assume that $q_1,q_2\in(0,\infty)$ satisfy $n(\f1p-\f1{q_1})<1$ and $n(\f1p-\f1{q_2})<1.$
From Theorem \ref{theorem3.9}(ii), it follows that
$$\mathbf{W\dot{F}}^1_{X,q_1}(\rn)\cap  \dot{W}^{1,X}(\rn
)=\mathbf{W\dot{F}}^1_{X,q_1}(\rn)\cap  \dot{W}^{1,X}(\rn
)$$
with equivalent quasi-norms. Thus, when $q\in(0,\infty)$
satisfies $n(\f 1p-\f1q) <1,$ the space  $\mathbf{W\dot{F}}^1_{X,q}(\rn)\cap  \dot{W}^{1,X}(\rn
)$ is independent of $q.$
\end{itemize}
\end{remark}
When $p=1$ and the Hardy--Littlewood maximal operator $\cm$ in \eqref{2-4-c} is not known to be bounded
on its associate space $X',$ Theorem \ref{theorem3.9} does not work anymore in this case;
instead of this, we have the following conclusion.

\begin{theorem}\label{theorem3}
Let	$X$ be a ball Banach function space and $q\in(0,\infty)$ with $n(1-\f1q) <1$.
Assume that
there exists a sequence $\{\theta_m\}_{m\in\nn}\subset(0,1)$
such that $\lim_{m\to\infty}\theta_m=1$ and, for any $m\in\nn$,
$X^{\f 1 {\theta_m}}$ is a ball Banach function space, $\cm$  in \eqref{2-4-c}  bounded
on its associate space $(X^{1/{\theta_m}})',$
 and
\begin{align}\label{sjx}
\lim_{m\to\infty}
\|\cm\|_{(X^{1/{\theta_m}})'\to(X^{1/{\theta_m}})'}
<\infty.
\end{align}
\begin{enumerate}
\item[{\rm(i)}]
Then, for any $f\in C^2(\rn)$ with $|\nabla f|\in C_{\mathrm{c}}(\rn)$,
\begin{align}\label{sjx2}
	\|f\|_{\mathbf{W\dot{F}}^1_{X,q}(\rn)}\sim\|\,|\nabla f|\,\|_X,
\end{align}
where the positive equivalence constants
are independent of $f$.
Moreover,
for any  $f\in C^2(\rn)$ with $|\nabla f|\in C_{\mathrm{c}}(\rn)$,
\begin{align}\label{leiliu}
	\lim_{\ld\to\infty}\ld^q\left\|\int_{\rn}
	\mathbf{1}_{E_f(\ld,q)}
	(\cdot,y)\,dy\right\|_{X^{\f 1 q}}=\frac{K(q,n)}{n}
	\|\,|\nabla f|\,\|_{X}^q,
\end{align}
where $E_f(\ld,q)$ for any $\lz\in (0,\fz)$ is the same as  in \eqref{E}.
\item [{\rm(ii)}]
Assume that
the centered ball average operators $\{B_r\}_{r\in(0,\infty)}$ are uniformly bounded on $X$ and
$X$ has an absolutely continuous norm.
Then, for any $f\in \dot{W}^{1,X}(\rn)$, both \eqref{sjx2} and
\eqref{leiliu} hold true.
\end{enumerate}
\end{theorem}

\begin{proof}
	We first prove (i).
From Theorem \ref{theorem4.88}, it follows that, for any
$f\in C^2(\rn)$ with $|\nabla f|\in C_{\mathrm{c}}(\rn)$, \eqref{leiliu}
holds true.
Thus, to show \eqref{sjx2},
it remains to prove that, for any $f\in C^2(\rn)$ with $|\nabla f|\in C_{\mathrm{c}}(\rn)$,
\begin{align}\label{013}
\|f\|_{\mathbf{W\dot{F}}^1_{X,q}(\rn)}
\lesssim\|\,|\nabla f|\,\|_X.
\end{align}
To this end, it suffices to show that,
for any  $f\in C^2(\rn)$ with $|\nabla f|\in C_{\mathrm{c}}(\rn)$ and for any $\ld\in(0,\infty),$
\begin{align}\label{010}
\lambda\left\|\lf|\lf\{y\in\rn:\ (\cdot,y)\in E_f(\ld,q)\r\}\r|
^{\frac{1}{q}}\right\|_{X}\lesssim
\|\,|\nabla f|\,\|_{X},
\end{align}
where $E_f(\ld,q)$ for any $\lz\in (0,\fz)$ is the same as  in \eqref{E}.
Let  $f\in C^2(\rn)$ with $|\nabla f|\in C_{\mathrm{c}}(\rn)$.
By Theorem \ref{theorem3.9}
and the facts that, for any
$m\in\nn$, $X^{\f 1 {\theta_m}}$ is a ball Banach
function space and that $\cm$ as in \eqref{2-4-c}
is bounded on its associate
space $(X^{1/{\theta_m}})'$,
we conclude that, for any $m\in\nn$,
\begin{align*}
\sup_{\lambda\in(0,\infty)}\lambda\left
\|\lf|\lf\{y\in\rn:\ (\cdot,y)\in E_f(\ld,q)\r\}\r|
^{\frac{1}{q}}\right\|_{X^{\f 1 {\theta_m}}}\lesssim
C_{(\|\cm\|_{(X^{1/{\theta_m}})'\to(X^{1/{\theta_m}})'})}\|\,|\nabla f|\,\|_{X^{\f 1 {\theta_m}}},
\end{align*}
which further implies that, for any
$\lambda\in(0,\infty)$,
\begin{align}\label{hg}
\lambda^{1/{\theta_m}}\left
\|\lf|\lf\{y\in\rn:\ (\cdot,y)\in E_f(\ld,q)\r\}\r|
^{\frac{1}{ q\theta_m}}\right\|_{X}\lesssim
C_{(\|\cm\|_{(X^{1/{\theta_m}})'
\to(X^{1/{\theta_m}})'})}^{1/{\theta_m}}
\lf\|\,|\nabla f|^{\f 1 {\theta_m}}\r\|_{X}.
\end{align}
From this, \eqref{hg}, Definition \ref{Debqfs}(iii), \eqref{sjx}, and
the fact that $C_{(\cdot)}$ is continuous
on $(0,\infty),$ together with $X$ being a
quasi-Banach space,
we deduce that, for any $\lambda\in(0,\infty)$,
\begin{align*}
&\left
\|\lf|\lf\{y\in\rn:\ (\cdot,y)\in E_f(\ld,q)\r\}\r|
^{\frac{1}{q}}\right\|_X\\
&\quad=
\left
\|\lim_{m\to\infty}\inf_{j\geq m}\lf|
\lf\{y\in\rn:\ (\cdot,y)\in E_f(\ld,q)\r\}\r|
^{\frac{1}{ q\theta_j}}\right\|_{X}\\
&\quad=
\lim_{m\to\infty}\left
\|\inf_{j\geq m}\lf|\lf\{y\in\rn:\ (\cdot,y)\in E_f(\ld,q)\r\}\r|
^{\frac{1}{ q\theta_j}}\right\|_{X}\\
&\quad\lesssim
\liminf_{m\to\infty}\left
\|\lf|\lf\{y\in\rn:\ (\cdot,y)\in E_f(\ld,q)\r\}\r|
^{\frac{1}{q\theta_m}}\right\|_{X}\\
&\quad
\lesssim\liminf_{m\to\infty}\lambda^{-
\f 1 {\theta_{m}}}
C_{(\|\cm\|_{(X^{1/{\theta_m}})'
\to(X^{1/\theta_m})'})}^{\f 1 {\theta_m}}
\lf\|\,|\nabla f|^{\f 1 {\theta_m}}\r\|_{X}\\
&\quad
\lesssim\lambda^{-1}
C_{(\lim_{m\to\infty}\|\cm\|_{(X^{1/\theta_m})'\to(X^{1/\theta_m})'})}
\limsup_{m\to\infty}
\lf\|\,|\nabla f|^{\f 1 {\theta_{m}}}\r\|_{X}\\
&\quad
\lesssim\lambda^{-1}\limsup_{m\to\infty}
\lf\|\,|\nabla f|^{\f 1 {\theta_{m}}}\r\|_{X}\\\noz
&\quad\sim
\lambda^{-1}\limsup_{m\to\infty}
\lf\|\lf[\frac{|\nabla f|}{\|\,|\nabla f|\,\|_
{L^{\infty}(\rn)}+1}\r]^{\f 1 {\theta_{m}}}\r\|_{X}
[\|\,|\nabla f|\,\|_{L^{\infty}(\rn)}+1]^{\f 1 {\theta_{m}}}\\
&\quad
\lesssim\lambda^{-1}\limsup_{m\to\infty}
\lf\|\frac{|\nabla f|}{\|\,|\nabla f|\,\|_{L^{\infty}(\rn)}+1}\r\|_{X}
[\|\,|\nabla f|\,\|_{L^{\infty}(\rn)}+1]^{\f 1 {\theta_{m}}}
\sim
\lambda^{-1}\lf\|\,|\nabla f|\,\r\|_{X},
\end{align*}
which  implies \eqref{010} and hence \eqref{013} hold true.
This finishes the proof of (i).

Next, we prove (ii). We first show that,
for any $f\in \dot{W}^{1,X}(\rn)$,
\begin{align}\label{luyang}
\sup_{\ld\in(0,\infty)}\ld
\left\|\left[\int_{\rn}\mathbf{1}_{E_f(\ld,q)}(\cdot,y)\,dy\right]^{\f 1 q}\right\|_X
\lesssim
\|\,|\nabla f|\,\|_X.
\end{align}
By Theorem \ref{density}, we conclude that there exists a sequence
$\{f_k\}_{k\in\nn}\subset C^\infty(\rn)$ with
$\{|\nabla f_k|\}_{k\in\nn}\subset C_{\mathrm{c}}(\rn)$  such that
\begin{align}\label{1912}
\lim_{k\to\infty}
\|\,|\nabla f-\nabla f_k|\,\|_{X}=0
\ \ \text{and}\ \
\lim_{k\to\infty}\|(f-f_k)\mathbf{1}_{B(\mathbf{0},R)}\|_{X}=0
\end{align}
for any $R\in(0,\infty)$.
For any $N\in\nn,$ let $$E_N:=\left\{(x,y)\in B(\mathbf{0},N)\times
B(\mathbf{0},N):\ |x-y|\geq \frac{1}{N}\right\}.$$
It is easy to show that, for any $k\in\nn$, $\ld\in(0,\infty)$, and
$\delta\in(0,1)$,
\begin{align*}
E_f(\ld,q)\subset E_{f_k}((1-\delta) \ld,q)\cup E_k^{(1)}(2^{-1}\delta\ld,q)\cup E_k^{(2)}(2^{-1}\delta\ld,q),
\end{align*}
where
\begin{align*}
E_k^{(1)}(2^{-1}\delta\ld,q):=\left\{(x,y)\in\rn\times\rn:\
|f(x)-f_k(x)|>\frac{\delta\ld}{2}|x-y|^{1+\frac{n}{q}}\right\}
\end{align*}
and
\begin{align*}
E_k^{(2)}(2^{-1}\delta\ld,q):=\left\{(x,y)\in\rn\times\rn:\
|f(y)-f_k(y)|>\frac{\delta\ld}{2}|x-y|^{1+\frac{n}{q}}\right\}.
\end{align*}
Using this and applying \eqref{key11} with $q$ replaced by $\frac{1}{q}$,
we obtain, for any $\ld\in(0,\infty),$ $N\in\nn$,
$k\in\nn,$ and $\delta,\theta\in(0,1)$,
\begin{align}\label{wuyu}
&\ld\left\|
\left[\int_{\rn}\mathbf{1}_{E_f(\ld,q)\cap E_N}(\cdot,y)\,dy\right]^{\f 1 q}\right\|_{X}\\\noz
&\quad\leq
\ld(1+\theta)\left\|\left[\int_{\rn}\mathbf{1}_{E_{f_k}((1-\delta)\ld,q)\cap E_N}
(\cdot,y)\,dy\right]^{\f 1 q}\right\|_{X}\\\noz
&\quad\quad+
\ld C_{\theta}\left\|\left[\int_{\rn}
\mathbf{1}_{E_{k}^{(1)}(2^{-1}\delta\ld,q)\cap E_N}(\cdot,y)+\mathbf{1}_{E_{k}^{(2)}(2^{-1}\delta\ld,q)\cap E_N}(\cdot,y)\,dy\right]^{\f 1 q}\right\|_{X}\\\noz
&\quad=:\mathrm{I}+\mathrm{II},
\end{align}
where $C_\theta$ is a positive constant depending only on $\theta.$
Applying \eqref{013} with $f:=f_k$, we have, for any $\ld\in(0,\infty)$, $k\in\nn,$ and $\delta,\theta\in(0,1)$,
\begin{align}\label{wei}
\mathrm{I}
&\leq
\ld(1+\theta)
\left\|
\left[\int_{\rn}\mathbf{1}_{E_{f_k}((1-\delta)\ld,q)}(\cdot,y)\,dy\right]^{\f 1 q}
\right\|_{X}\\\noz
&\lesssim
(1-\delta)^{-1}(1+\theta)\|\,|\nabla f_k|\,\|_{X}\\\noz
&\lesssim
(1-\delta)^{-1}(1+\theta)
(\|\,|\nabla f|\,\|_{X}+\|\,|\nabla f-\nabla f_k|\,\|_{X}).
\end{align}
Next, we estimate $\mathrm{II}$.
It is obvious that,  for any $x\in\rn,$
\begin{align}\label{fuzhuang}
&\int_{\rn}\mathbf{1}_{E_{k}^{(1)}(2^{-1}\delta\ld,q)\cap E_N}(x,y)\,dy\\\noz
&\quad\leq
(2\ld^{-1}\delta^{-1})^q
\int_{\rn}
\frac{|f(x)-f_k(x)|^q}{|x-y|^{q+n}}\mathbf{1}_{E_N}(x,y)\,dy\\\noz
&\quad\leq
(2\ld^{-1}\delta^{-1})^qN^{q+N}|B(\mathbf{0},N)||f(x)-f_k(x)|^q\mathbf{1}_{B(\mathbf{0},N)}(x).
\end{align}
On the other hand, by Lemma \ref{lemma3.5}, we find that, for any
$x\in\rn$,
\begin{align*}
&\int_{\rn}\mathbf{1}_{E_{k}^{(2)}(2^{-1}\delta\ld,q)\cap E_N}(x,y)\,dy\\\noz
&\quad\leq
\frac{2}{\delta\ld}\int_{\rn}
\frac{|f(y)-f_k(y)|}{|x-y|^{1+\frac{n}{q}}}\mathbf{1}_{E_N}(x,y)\,dy\\\noz
&\quad\leq
\frac{2N^{1+\frac{n}{q}}}{\delta\ld}\int_{B(\mathbf{0},N)}|f(y)-f_k(y)|\,dy
\mathbf{1}_{B(\mathbf{0},N)}(x)\\\noz
&\quad\leq
\frac{2N^{1+\frac{n}{q}}}{\delta\ld}
\|(f-f_k)\mathbf{1}_{B(\mathbf{0},N)}\|_{X}\|\mathbf{1}_{B(\mathbf{0},N)}\|_{X'}\mathbf{1}_{B(\mathbf{0},N)}(x).
\end{align*}
This, together with \eqref{fuzhuang},  further implies that,
for any $\ld\in(0,\infty),$ $N\in\nn$,
$k\in\nn,$ and $\delta,\theta\in(0,1)$,
\begin{align*}
\mathrm{II}
&\leq
\ld C_{\theta}2^{\f 1 q}
\left\|\left[\int_{\rn}\mathbf{1}_{E_{k}^{(1)}(2^{-1}\delta\ld,q)\cap E_N}(\cdot,y)\,dy\right]^{\f 1 q}\right\|_{X}\\\noz
&\quad+
\ld C_{\theta}2^{\f 1 q}
\left\|\left[\int_{\rn}\mathbf{1}_{E_{k}^{(2)}(2^{-1}\delta\ld,q)\cap E_N}(\cdot,y)\,dy\right]^{\f 1 q}\right\|_{X}\\\noz
&\leq
C_{\theta}2^{\f 1 q}
2\delta^{-1}N^{1+\frac{n}{q}}|B({\mathbf 0},N)|^{\frac{1}{q}}\|(f-f_k)\mathbf{1}_{B(\mathbf{0},N)}\|_{X}\\\noz
&\quad+
\ld C_{\theta}2^{\f 1 q}\left(\frac{2N^{1+\frac{n}{q}}}{\delta\ld}\right)^{\frac{1}{q}}\|(f-f_k)\mathbf{1}_{B(\mathbf{0},N)}\|_{X}^{\f 1 q}\|\mathbf{1}_{B(\mathbf{0},N)}\|_{X'}^{\f 1 q}\|\mathbf{1}_{B(\mathbf{0},N)}\|_{X}.
\end{align*}
Using this, \eqref{wei}, \eqref{wuyu},
and \eqref{1912}, and
letting $k\to\infty$, we conclude that,
for any $\ld\in(0,\infty)$, $N\in\nn$,  and $\delta, \theta\in(0,1)$,
\begin{align*}
\ld\left\|
\left[\int_{\rn}\mathbf{1}_{E_f(\ld,q)\cap E_N}(\cdot,y)\,dy\right]^{\f 1 q}\right\|_{X}
\lesssim
(1-\delta)^{-1}(1+\theta)\|\,|\nabla f|\,\|_{X}.
\end{align*}
Let $\delta\to0$,  $\theta\to0,$
and $N\to\infty$.
Then we find that \eqref{luyang} holds true.

By \eqref{luyang} and a density argument
similar to that used in the proof
of \eqref{dui}, we conclude that
\eqref{leiliu} holds true for any
$f\in \dot{W}
^{1,X}(\rn).$ This finishes the proof
of (ii) and hence Theorem \ref{theorem3}.
\end{proof}

\begin{remark}
\begin{itemize}
\item[{\rm(i)}] In Theorem \ref{theorem3}(ii), if $X:=L^p(\rn)$ with $p\in[1,\infty)$
and if $q\in(0,\infty)$ satisfies $n(1-\frac{1}{q})<1$, then the conclusions of
Theorem \ref{theorem3}(ii) hold true with $X$ replaced by $L^p(\rn)$.
Moreover, when $q:=1$ and $X:=L^1(\rn)$, \eqref{sjx2} is just \eqref{bsy}  with $p=1$
in \cite{ref8,BSSY21}.
\item[{\rm(ii)}]
Theorem \ref{theorem3} is used to solve the endpoint case
of concrete examples of ball Banach function spaces in Section \ref{sec:6}.
For example, when $X:=L^{r(\cdot)}(\rn)$
with $\widetilde{r}_-=1$ (see Subsection \ref{s6.3} below 
for the precise definitions of both $L^{r(\cdot)}(\rn)$ and $\widetilde{r}_-$),
since it is still unclear whether or not the Hardy--Littlewood maximal operator
$\cm$ is  bounded on $[L^{r(\cdot)}(\rn)]'$, it follows that
Theorem \ref{theorem3.9} does not work anymore in this case,
while Theorem \ref{theorem3} is applicable in this case.
\item[{\rm(iii)}]
We should point out that
we do not need the assumption \eqref{sjx} in the proof of
the lower estimate of \eqref{sjx2}.
\end{itemize}
\end{remark}
As a consequence of Theorem \ref{theorem3.9},
we obtain the following fractional Sobolev type inequality on ball Banach function spaces.
\begin{corollary}\label{corollary3.111}
Let $q_1\in[1,\infty],$
$\ta\in (0,1)$, and
$q\in [1, q_1]$ satisfy
$ \frac{1}{q}=\frac{1-\theta}{q_1}+\theta$.
Let $X$ be a ball Banach function space. Assume that
there exists a sequence $\{\al_m\}\in(0,1)$ such that, for any $m\in\nn$, $\lim_{m\to\infty}\al_m=1$,
$X^{\f 1 {\al_m}}$ is a ball Banach function space, $\cm$  in \eqref{2-4-c}  bounded
on its associate space $(X^{1/{\al_m}})',$
and
\begin{align}\label{72121}
	\lim_{m\to\infty}
	\|\cm\|_{(X^{1/{\al_m}})'\to(X^{1/{\al_m}})'}
	<\infty.
\end{align}
\begin{itemize}
\item[{\rm(i)}]
Let $q_1\in[1,\infty).$ Then there exists a positive constant $C$ such
that, for any $f\in C^2(\rn)$ with $|\nabla f|\in C_{\mathrm{c}}(\rn)$,
\begin{align}\label{jjdjr}
	\|f\|_{\mathbf{W\dot{F}}^{\theta}_{X^q,q}(\rn)}\leq C\|f
	\|_{X^{q_1}}^{1-\ta}\|\,|\nabla f|\,\|_X^{\ta}.
\end{align}
Assume further that  $X$ has an absolutely continuous norm,
$X^{q_1}$ is a ball Banach function space, and
the centered ball average operators $\{B_r\}_{r\in(0,\infty)}$ are uniformly bounded  on $X$.
Then \eqref{jjdjr} holds true for any $f\in \dot{W}^{1,X}(\rn)$.
\item[{\rm(ii)}]
Let $q_1=\infty.$ Then there exists a positive constant $C$ such
that, for any $f\in C^2(\rn)$ with $|\nabla f|\in C_{\mathrm{c}}(\rn)$,
\begin{align}\label{henghengheng}
	\|f\|_{\mathbf{W\dot{F}}^{\ta}_{X^q,q}(\rn)}\leq C\|f
	\|_{L^\infty(\rn)}^{1-\ta}\|\,|\nabla f|\,\|_X^{\ta}.
\end{align}
Assume further that
the centered ball average operators $\{B_r\}_{r\in(0,\infty)}$ are uniformly bounded on $X$ and
$X$ has an absolutely continuous norm.
Then \eqref{henghengheng} holds true for any $f\in \dot{W}^{1,X}(\rn)$.
\end{itemize}
\end{corollary}
\begin{proof}
We first prove $\rm(ii).$
Let $f\in\dot{W}^{1,X}(\rn)$.
If $\|f\|_{L^\infty(\rn)}=0$
or $\|f\|_{L^\infty(\rn)}=\infty$,
then \eqref{henghengheng} holds true automatically. Now, we assume that $\|f\|_{L^\infty(\rn)}\in (0,\infty)$.
Let $q_1:=\infty.$ Then $q\theta=1$.
For any $\ld,r,s\in(0,\infty)$, let
$E_{f}(\ld,r,s)$ be the same as  in \eqref{keydef}.
Since, for any $\ld\in(0,\infty),$
$$ E_f\left(\ld, \frac 1\ta, \ta\right) \subset E_f\left( \f {\ld^{1/\ta}
}{[2\|f\|_{L^\infty(\rn)}]^{(1-\ta)/\ta}}, 1,1\right),$$
from Definition \ref{Debqfs}(ii), we deduce that
\begin{align}\label{ha}
\|f\|_{\mathbf{W\dot{F}}^{\ta}_{X^q,q}(\rn)}^{1/\ta}	
&=\sup_{\ld\in(0,\infty)}\lambda^{1/\ta}\left\|
\int_{\rn}\mathbf{1}_{E_f(\ld,q,\ta)}
(\cdot,y)\,dy\right\|_X\\\noz
&\leq\sup_{\ld\in(0,\infty)}\lambda^{1/\ta}
\left\|\int_{\rn}\mathbf{1}
_{E_f\left(\f {\ld^{1/\ta}
}{[2\|f\|_{L^\infty(\rn)}]^{(1-\ta)/\ta}}, 1,1\right)}
(\cdot,y)\,dy\right\|_X\\\noz
&=
[2\|f\|_{L^\infty(\rn)}]^{(1-\ta)/\ta}
\sup_{\ld\in(0,\infty)}\lambda\left\|
\int_{\rn}\mathbf{1}_{E_f(\ld, 1,1)}
(\cdot,y)\,dy\right\|_X.
\end{align}

If $f\in C^2(\rn)$ with $|\nabla f|\in C_\mathrm{c}(\rn)$, then,
applying Theorem \ref{theorem3}(i) with $q:=1$,
we find that
\begin{align}\label{haa}
\sup_{\ld\in(0,\infty)}\lambda\left\|
\int_{\rn}\mathbf{1}_{E_f(\ld, 1,1)}
(\cdot,y)\,dy\right\|_X
\lesssim
\|\,|\nabla f|\,\|_{X}.
\end{align}
This, together with \eqref{ha}, further implies that \eqref{henghengheng}
holds true for any $f\in C^2(\rn)$ with $|\nabla f|\in C_\mathrm{c}(\rn)$.

If the centered ball average operators $\{B_r\}_{r\in(0,\infty)}$ are uniformly bounded on $X$ and
$X$ has an absolutely continuous norm,
then, by  Theorem \ref{theorem3}(ii) with $q:=1$, we conclude that \eqref{haa} holds true
for any $f\in \dot{W}^{1,X}(\rn)$.
This, combined with \eqref{ha}, further
implies that \eqref{henghengheng} holds true for any $f\in \dot{W}^{1,X}(\rn)$.
This finishes the proof of (ii) of this
corollary.

Next, we prove $\rm(i).$ Assume $q_1\in[1,\infty)$.
We first show that  \eqref{jjdjr}
holds true for any $f\in C^2(\rn)$
with $|\nabla f|\in C_{\mathrm{c}}(\rn)$.
Let $f\in C^2(\rn)$  with $|\nabla f|\in C_{\mathrm{c}}(\rn)$
and $A\in(0,\infty)$ be a constant
which is specified later. Since, for any $x, y\in\rn,$
$$\frac{|f(x)-f(y)|}{|x-y|^{\frac{n}{q}+\theta}}=
\left[\frac{|f(x)-f(y)|}{|x-y|^{\frac{n}{q_1}
		}}\right]^{1-\theta}\left[\frac{|f(x)-f(y)|}
{|x-y|^{n+1}}\right]^{\theta},$$
we deduce that, for any $\ld\in(0,\infty),$
\begin{align*}
	&E_f(\ld, q,\theta) \subset \left[E_f \left(A^{-\ta}\ld,
	q_1, 0\right) \cup E_f\left(A^{1-\ta}\ld, 1,1\right)\right].
\end{align*}
This, together with the fact that $\|\cdot\|_X$ is a quasi-norm, further implies that, for any $\ld\in(0,\infty)$,
\begin{align}\label{5-7dd}
&\left\|\int_{\rn}\mathbf{1}_{E_f(\ld,q,\theta)}
(\cdot,y)\,dy\right\|_{X}^{\f1q}\\\noz
&\quad\lesssim\left\|\int_{\rn}\mathbf{1}_{E_f\left(A^{-\ta} \ld, q_1, 0\right)}(\cdot,y)\,dy\right\|
_{X}^{\f1q}+ \left\|\int_{\rn}\mathbf{1}_{E_f\left(A^{1-\ta}\ld,
1,1\right)}(\cdot,y)\,dy\right\|_{X}^{\f1q}\\\noz
&\quad\lesssim \left(\frac{A^{\theta}\mathrm{G}}{\ld}\right)^{\frac{q_1}{q}}
+\left(\frac{\mathrm{H}}{A^{1-\theta}\ld}\right)^{\f1q},
\end{align}
where
\begin{align*}
	\mathrm{G}:&=\sup_{\ld\in(0,\infty)}\lambda\left\|
	\int_\rn\mathbf{1}_{E_f(\ld, q_1, 0)}
	(\cdot,y)\,dy\right\|
	_{X}^{\f1{q_1}}
\end{align*}
and
\begin{align*}
	\mathrm{H}:&=\sup_{\ld\in(0,\infty)}\lambda
	\left\|\int_{\rn}\mathbf{1}_{E_f(\ld,1,1)}
	(\cdot,y)\,dy\right\|_{X}.
\end{align*}
Choose an $A\in(0,\infty)$ such that
$$\left(\frac{A^{\theta}\mathrm{G}}{\ld}\right)^{\frac{q_1}{q}}
=\left(\frac{\mathrm{H}}{A^{1-\theta}\ld}\right)^{\f1q}.$$
This,  combined with \eqref{5-7dd}, then implies that
\begin{align}\label{youyong}
\left\|\int_{\rn}\mathbf{1}_{E_f(\ld,q,\theta)}
(\cdot,y)\,dy\right\|_{X}^{\f1q}
&\lesssim\left(\frac{A^{\theta}\mathrm{G}}{\ld}\right)^{\frac{q_1}{q}}
+\left(\frac{\mathrm{H}}{A^{1-\theta}\ld}\right)^{\f1q}\\\noz
&\sim\left(\frac{A^{\theta}\mathrm{G}}{\ld}\right)^{\frac{q_1}{q}}
\sim\ld^{-1}\mathrm{G}^{1-\theta}\mathrm{H}^{\theta}.
\end{align}
Using this and applying Theorem \ref{theorem3}(i) with $q:=1$, we conclude that, for any $f\in C^2(\rn)$
with $|\nabla f|\in C_{\mathrm{c}}(\rn)$,
\begin{align}\label{hahaha}
\sup_{\ld\in(0,\infty)}\lambda\left\|\left|
\lf\{y\in\rn:\ |f(x)-f(y)|>\lambda
|x-y|^{\f n {q}+\ta}\r\}\right|
^{\frac{1}{q}}\right\|_{X^{q}}
\lesssim\mathrm{G}^{1-\theta}
\|\,|\nabla f|\,\|_{X}^{\theta}.
\end{align}

Next, we prove that, for any $f\in C^2(\rn)$ with $|\nabla f|\in C_{\mathrm{c}}(\rn)$,
\begin{align}\label{wuwuwu}
\mathrm{G}\lesssim\|f\|_{X^{q_1}}.
\end{align}
To this end, by Lemma \ref{lemma3.4} and Lemma \ref{lemma3.7}(i), we have, for any $\ld\in(0,\infty)$ and $m\in\nn,$
\begin{align}\label{xixixi}
&\left\|\int_\rn\mathbf{1}_{E_f(\ld, q_1, 0)}
(\cdot,y)\,dy\right\|_{X^{1/{\al_m}}}\\\noz
&\quad=\sup_{\|g\|_{(X^{1/{\al_m}})'}=1}\int_{\rn}\int_\rn\mathbf{1}_{E_f(\ld, q_1, 0)}
(x,y)\,dyg(x)\,dx\\\noz
&\quad\leq\sup_{\|g\|_{(X^{1/{\al_m}})'}=1}\int_{\rn}\int_\rn\mathbf{1}_{E_f(\ld, q_1, 0)}
(x,y)\,dyR_{(X^{1/{\al_m}})'}g(x)\,dx.
\end{align}
For any $x,y\in\rn,$ let
$$D_f(\ld,q_1)_1:=\{(x,y)\in\rn\times\rn:\ |f(x)|>\ld|x-y|^{\f n{q_1}}/2\}$$
and
$$D_f(\ld,q_1)_2:=\{(x,y)\in\rn\times\rn:\ |f(y)|>\ld|x-y|^{\f n{q_1}}/2\}.$$
Observe that, for any $\ld\in(0,\infty),$
$E_f(\ld, q_1, 0)\subset D_f(\ld,q_1)_1\cup D_f(\ld,q_1)_2.$ From this, both (ii) and (iii) of Lemma \ref{lemma3.7}, and the definition
of $A_1(\rn),$ we deduce that, for any  $\ld\in(0,\infty),$ $m\in\nn,$ and $g\in (X^{1/{\al_m}})'$ with $\|g\|_{(X^{1/{\al_m}})'}=1,$
\begin{align*}
&\int_{\rn}\int_\rn\mathbf{1}_{E_f(\ld, q_1, 0)}
(x,y)\,dyR_{(X^{1/{\al_m}})'}g(x)\,dx\\
&\quad\leq\int_{\rn}\int_{\{y\in\rn:\ (x,y)\in D_f(\ld,q_1)_1\}}\,dyR_{(X^{1/{\al_m}})'}g(x)\,dx\\
&\qquad+\int_{\rn}\int_{\{x\in\rn:\ (x,y)\in D_f(\ld,q_1)_2\}}R_{(X^{1/{\al_m}})'}g(x)\,dx\,dy\\
&\quad\lesssim\ld^{-q_1}\int_\rn|f(x)|^{q_1}
R_{(X^{1/{\al_m}})'}g(x)\,dx\\\noz
&\quad\quad+[R_{(X^{1/{\al_m}})'}g]_{A_1(\rn)}\ld^{-q_1}\int_\rn|f(y)|^{q_1}
R_{(X^{1/{\al_m}})'}g(y)\,dy\\
&\quad\lesssim\left[1+\|\cm\|_{(X^{1/{\al_m}})'\to (X^{1/{\al_m}})'}\right]\ld^{-q_1}\|f\|
_{X^{q_1/{\al_m}}}^{q_1}
\|R_{(X^{1/{\al_m}})'}g\|_{(X^{1/{\al_m}})'}\\
&\quad\lesssim\left[1+\|\cm\|_{(X^{1/{\al_m}})'
\to(X^{1/{\al_m}})'}\right]\ld^{-q_1}\|f\|_{X^{q_1/{\al_m}}}^{q_1}.
\end{align*}
By this and \eqref{xixixi}, we find that
there exists a positive constant $C$ such
that,
for any   $m\in\nn$ and
$f\in C_{\mathrm{c}}^2(\rn),$
\begin{align}\label{6.21}
\ld\left\|\int_\rn\mathbf{1}_{E_f(\ld, q_1, 0)}
(\cdot,y)\,dy\right\|_{X^{1/{\al_m}}}^{1/{q_1}}
\leq C\left[1+\|\cm\|_{(X^{1/{\al_m}})'\to (X^{1/{\al_m}})'}\right]\|f\|_{X^{q_1/{\al_m}}}.
\end{align}
 Using this, Definition \ref{Debqfs}(iii), \eqref{6.21},
and \eqref{sjx}, together with $X$ being a
quasi-Banach space,
we conclude that, for any $\lambda\in(0,\infty)$,
\begin{align*}
&\left\|\int_\rn\mathbf{1}_{E_f(\ld, q_1, 0)}
(\cdot,y)\,dy\right\|_{X}\\
&\quad=\left\|\lim_{m\to\infty}\inf_{j\geq m}\left[\int_\rn\mathbf{1}_{E_f(\ld, q_1, 0)}
(\cdot,y)\,dy\right]^{1/\al_{j}}\right\|_{X}\\
&\quad=\lim_{m\to\infty}\left\|\inf_{j\geq m}\left[\int_\rn\mathbf{1}_{E_f(\ld, q_1, 0)}
(\cdot,y)\,dy\right]^{1/\al_{j}}\right\|_{X}\\
&\quad\leq\liminf_{m\to\infty}\left\|\left[\int_\rn\mathbf{1}_{E_f(\ld, q_1, 0)}
(\cdot,y)\,dy\right]^{1/\al_{m}}\right\|_{X}\\
&\quad\leq\liminf_{m\to\infty}\left[1+\|\cm\|_{(X^{1/\al_{m}})'\to(X^{1/\al_{m}})'}
\right]^{q_1/\al_m}\ld^{-q_1/\al_m}\left\||f|^{\f{q_1}{\al_m}}\right\|_X\\
&\quad\lesssim\ld^{-q_1}\limsup_{m\to\infty}
\left\||f|^{\f{q_1}{\al_m}}\right\|_X\\
&\quad\sim\ld^{-q_1}\limsup_{m\to\infty}
\left\|\,\left|\f{f}{\|f\|_{L^\infty(\rn)}}\right|^{\f{q_1}{\al_m}}\right\|_X\|f\|_{L^\infty(\rn)}^{\f{q_1}{\al_m}}\\
&\quad\lesssim\ld^{-q_1}\limsup_{m\to\infty}
\left\|\,\left|\f{f}{\|f\|_{L^\infty(\rn)}}\right|^{q_1}\right\|_X\|f\|_{L^\infty(\rn)}^{\f{q_1}{\al_m}}
\lesssim\ld^{-q_1}
\left\||f|^{q_1}\right\|_X.
\end{align*}
This implies that \eqref{wuwuwu} holds true. From both
\eqref{wuwuwu} and \eqref{hahaha}, it follows that \eqref{jjdjr} holds true
for  any $f\in C^2(\rn)$
with $|\nabla f|\in C_{\mathrm{c}}(\rn)$.

Assume that $X$ has an absolutely continuous norm,
$X^{q_1}$ is a ball Banach function space, and
the centered ball average operators $\{B_r\}_{r\in(0,\infty)}$ are uniformly bounded  on  $X$.
It remains to show that \eqref{jjdjr} holds true for any
$f\in \dot{W}^{1,X}(\rn)$ under these additional assumptions.

Indeed, by both \eqref{youyong}
and Theorem \ref{theorem3}(ii) with $q:=1$, we find that \eqref{hahaha}
holds true for any $f\in \dot{W}^{1,X}(\rn)$. This implies that,
to prove that \eqref{jjdjr} holds true for any $f\in \dot{W}^{1,X}(\rn)$,
it suffices to show that \eqref{wuwuwu}
holds true for any $f\in  \dot{W}^{1,X}(\rn)\cap
X^{q_1}$.
Let $f\in \dot{W}^{1,X}(\rn)\cap X^{q_1}.$
By the assumption that  the centered ball average operators $\{B_r\}_{r\in(0,\infty)}$ are uniformly bounded  on $X$,
it is easy to show that the centered ball average operators $\{B_r\}_{r\in(0,\infty)}$
are also uniformly bounded  on $X^{q_1}$.
Moreover, $X$ has an absolutely continuous norm implies that
it is obvious that $X^{q_1}$ also has an absolutely continuous norm.
From these and the assumption that
$X^{q_1}$ is a ball Banach function space,
it follows that $C^\infty_{\mathrm{c}}(\rn)$ is dense in $X^{q_1}$ (see, for instance,
\cite[Corollary 3.10]{dgpyyz}). Thus, there  exists
a sequence $\{f_k\}_{k\in\nn}\subset C^\infty_{\mathrm{c}}(\rn)$ such that
\begin{align}\label{youlai}
\lim_{k\to\infty}\|f-f_k\|_{X^{q_1}}=0.
\end{align}
For any $N\in\nn$, let
\begin{align*}
E_N:=\left\{(x,y)\in B(\mathbf{0},N)\times B(\mathbf{0},N):\
|x-y|\geq N^{-1}\right\}.
\end{align*}
Notice that, for any $k\in\nn$ and $\ld\in(0,\infty)$,
\begin{align*}
E_f(\ld,q_1,0)
\subset
E_{f_k}(3^{-1}\ld,q_1,0)
\cup
E_{k}^{(1)}(3^{-1}\ld,q_1,0)
\cup
E_{k}^{(2)}(3^{-1}\ld,q_1,0),
\end{align*}
where
\begin{align*}
E_{k}^{(1)}(3^{-1}\ld,q_1,0)
:=\left\{(x,y)\in\rn\times\rn:\
|f(x)-f_k(x)|>3^{-1}\ld|x-y|^{\frac{n}{q_1}}
\right\}
\end{align*}
and
\begin{align*}
E_{k}^{(2)}(3^{-1}\ld,q_1,0)
:=\left\{(x,y)\in\rn\times\rn:\
|f(y)-f_k(y
)|>3^{-1}\ld|x-y|^{\frac{n}{q_1}}
\right\}.
\end{align*}
This, together with the assumption that $X$ is a ball Banach function
space and the fact that $\frac{1}{q_1}\in (0,1]$, further implies that, for any $\ld\in(0,\infty)$ and $k,N\in\nn$,
\begin{align}\label{huxi}
&\ld\left\|\int_{\rn}\mathbf{1}_{E_f(\ld,q_1,0) \cap
E_N}(\cdot,y)\,dy\right\|_{X}^{\frac{1}{q_1}}\\\noz
&\quad\leq
\ld\left\|\int_{\rn}\mathbf{1}_{E_{f_k}(3^{-1}\ld,q_1,0) \cap
E_N}(\cdot,y)\,dy\right\|_{X}^{\frac{1}{q_1}}\\\noz
&\qquad+
\ld\left\|\int_{\rn}\mathbf{1}_{E_{k}^{(1)}(3^{-1}\ld,q_1,0) \cap
E_N}(\cdot,y)\,dy\right\|_{X}^{\frac{1}{q_1}}\\\noz
&\qquad+
\ld\left\|\int_{\rn}\mathbf{1}_{E_{k}^{(2)}(3^{-1}\ld,q_1,0) \cap
E_N}(\cdot,y)\,dy\right\|_{X}^{\frac{1}{q_1}}\\
&\quad=:\mathrm{I}+\mathrm{II}+\mathrm{III}.\noz
\end{align}
Applying \eqref{wuwuwu} with $f:=f_k$ and using
 the assumption that $X^{q_1}$ is a ball Banach function space,
we find that, for any $\ld\in(0,\infty)$ and $k,N\in\nn$,
\begin{align}\label{yueqiu1}
\mathrm{I}
\leq
\ld\left\|\int_{\rn}\mathbf{1}_{E_{f_k}(3^{-1}\ld,q_1,0)}(\cdot,y)\,dy\right\|_{X}^{\frac{1}{q_1}}
\lesssim
\|f_k\|_{X^{q_1}}
\lesssim
\|f_k-f\|_{X^{q_1}}+\|f\|_{X^{q_1}}.
\end{align}
Observe that, for any $\ld\in(0,\infty)$,  $k,N\in\nn$, and
$x\in\rn,$
\begin{align*}
\ld^{q_1}\int_{\rn}\mathbf{1}_{E_{k}^{(1)}(3^{-1}\ld,q_1,0)\cap E_N}(x,y)\,dy
&\leq
3^{q_1}\int_{\rn}\frac{|f(x)-f_k(x)|^{q_1}}{|x-y|^n}\mathbf{1}_{E_N}(x,y)\,dy\\\noz
&\leq
3^{q_1}N^n|B(\mathbf{0},N)||f(x)-f_k(x)|^{q_1}.
\end{align*}
Thus, for any $\ld\in(0,\infty)$ and $k,N\in\nn$,
\begin{align}\label{yueqiu3}
\mathrm{II}
\leq
3[N^n|B(\mathbf{0},N)|]^{\frac{1}{q_1}}\|f-f_k\|_{X^{q_1}}.
\end{align}
By Lemma  \ref{lemma3.5}, we find that,  for any $\ld\in(0,\infty)$,  $k,N\in\nn$, and
$x\in\rn,$
\begin{align*}
&\ld^{q_1}\int_{\rn}\mathbf{1}_{E_{k}^{(2)}(3^{-1}\ld,q_1,0)\cap E_N}(x,y)\,dy\\\noz
&\quad\leq
3^{q_1}\int_{\rn}\frac{|f(y)-f_k(y)|^{q_1}}{|x-y|^n}\mathbf{1}_{E_N}(x,y)\,dy\\\noz
&\quad\leq
3^{q_1}N^n\int_{B(\mathbf{0},N)}|f(y)-f_k(y)|^{q_1}\,dy\mathbf{1}_{B(\mathbf{0},N)}(x)\\\noz
&\quad\leq
3^{q_1}N^n\|f-f_k\|_{X^{q_1}}^{q_1}
\|\mathbf{1}_{B(\mathbf{0},N)}\|_{X'}
\mathbf{1}_{B(\mathbf{0},N)}(x)
\end{align*}
and hence
\begin{align*}
\mathrm{III}
\leq
3[N^n\|\mathbf{1}_{B(\mathbf{0},N)}\|_{X}
\|\mathbf{1}_{B(\mathbf{0},N)}\|_{X'}]^{\frac{1}{q_1}}\|f-f_k\|_{X^{q_1}}.
\end{align*}
From this, \eqref{yueqiu3}, \eqref{yueqiu1}, and \eqref{huxi},
it follows that, for any $\ld\in(0,\infty)$ and $k,N\in\nn$,
\begin{align*}
&\ld\left\|\int_{\rn}\mathbf{1}_{E_f(\ld,q_1,0) \cap
E_N}(\cdot,y)\,dy\right\|_{X}^{\frac{1}{q_1}}\\\noz
&\quad\lesssim
\left\{1+[N^n|B(\mathbf{0},N)|]^{\frac{1}{q_1}}+[N^n\|\mathbf{1}_{B(\mathbf{0},N)}\|_{X}
\|\mathbf{1}_{B(\mathbf{0},N)}\|_{X'}]^{\frac{1}{q_1}}\right\}\|f-f_k\|_{X^{q_1}}\\\noz
&\quad\quad+\|f\|_{X^{q_1}}.
\end{align*}
Using this and \eqref{youlai}, and
letting $k\to\infty$ and $N\to\infty$,
we find that \eqref{wuwuwu} holds true for any $f\in X^{q_1}$, which completes
the proof of (ii) and hence Corollary \ref{corollary3.111}.
\end{proof}
\begin{remark}\label{r3.10}

\begin{itemize}
\item[{\rm(i)}]
Let $X:=L^p(\rn)$ with $p\in[1,\infty).$
Then  Corollary \ref{corollary3.111}
holds true with $X$ replaced by $L^p(\rn)$.
In particular, if $q_1:=\infty$, $q\in[1,\fz]$ and $\theta:=1/q,$  $X:=L^1(\rn)$,
and $f\in C^\infty_{\mathrm{c}}(\rn)$,
then, in this case, Corollary \ref{corollary3.111}(ii)
is just \cite[Corollary 5.1]{ref8}; furthermore, in this case, if $n=1$, 
by Corollary \ref{corollary3.111}(ii), we find that,
for any $f\in C^\infty_{\mathrm{c}}(\rr)$,
\begin{align*}
\left\|\frac{|f(x)-f(y)|}{|x-y|^{\frac{2}{q}}}\right\|_{L^{q,\infty}(\rr\times\rr)}
\lesssim
\|f\|_{L^\infty(\rr)}^{1-\theta}
\|f'\|_{L^1(\rr)}^{\theta}
\lesssim
\|f'\|_{L^1(\rr)},
\end{align*}
which is just \cite[Corollary 4.1]{ref8}.
\item[{\rm(ii)}] Let $q_1:=\infty.$ In this case, when we replace
the weak type norm $\|\cdot\|_{\mathbf{W\dot{F}}^{\ta}_{X^q,q}(\rn)}$ in \eqref{henghengheng} by the strong
type norm $\|\cdot\|_{\dot{\mathbf{F}}^{\ta}_{X^q,q}(\rn)}$, \eqref{henghengheng} may not hold true
(see, for instance, \cite[(5.3)]{ref8}). In this sense, \eqref{henghengheng} with $q_1=\infty$
seems to be sharp.

\item[{\rm(iii)}] Let $q_1\in[1,\infty).$ We should point out that, if
the weak type norm $\|\cdot\|_{\mathbf{W\dot{F}}^{\ta}_{X^q,q}(\rn)}$ in \eqref{jjdjr} is replaced by the strong
type norm $\|\cdot\|_{\dot{\mathbf{F}}^{\ta}_{X^q,q}(\rn)},$   it is still unclear  whether or not Corollary
\ref{corollary3.111}(i) still holds true.
\item[{\rm(iv)}]
Let $X$ be the same as in Theorem \ref{theorem3.9}(ii).
Then both \eqref{jjdjr} and \eqref{henghengheng} hold true
 for any $f\in\dot{W}^{1,X}(\rn)$,
which can be proved by a slight modification of the proof of Corollary \ref{corollary3.111}
with Theorem \ref{theorem3} replaced by Theorem \ref{theorem3.9}(ii); we omit the details.
\end{itemize}
\end{remark}

Similarly, using Theorem \ref{theorem3.9}, we obtain the following fractional
Gagliardo--Nirenberg type inequality on ball Banach function spaces.
\begin{corollary}\label{corollary1001}
Let $s_1 \in[0,1)$, $q_1\in(1,\infty),$ and
$\ta\in (0,1)$. Let $s\in (s_1, 1)$ and
$q\in (1, q_1)$ satisfy $s=(1-\theta) s_1+\theta$
and $ \frac{1}{q}=\frac{1-\theta}{q_1}+\theta$.
Let $X$ be a ball Banach function space. Assume that
there exists a sequence $\{\al_m\}\in(0,1)$ such that, for any $m\in\nn$, $\lim_{m\to\infty}\al_m=1$,
$X^{\f 1 {\al_m}}$ is a ball Banach function space, $\cm$  in \eqref{2-4-c}  bounded
on its associate space $(X^{1/{\al_m}})',$
and \eqref{72121} holds true.
Then there exists a positive constant $C$ such
that, for	any $f\in C^2(\rn)$ with $|\nabla f|\in C_\mathrm{c}(\rn)$,
\begin{align}\label{zhongyang}
\|f\|_{\mathbf{W\dot{F}}^s_{X^q,q}(\rn)}
\leq
C\|f\|^{1-\theta}_{\mathbf{W\dot{F}}^{s_1}_{X^{q_1},q_1}(\rn)}
\|\,|\nabla f|\,\|_{X}^{\theta}
\leq
C\|f\|^{1-\theta}_{\dot{\mathbf{F}}^{s_1}_{X^{q_1},q_1}(\rn)}
\|\,|\nabla f|\,\|_{X}^{\theta}.
\end{align}
Assume further that
the centered ball average operators $\{B_r\}_{r\in(0,\infty)}$ are uniformly bounded on $X$ and
$X$ has an absolutely continuous norm.
Then \eqref{zhongyang} holds true for any $f\in \dot{W}^{1,X}(\rn)$.
\end{corollary}
\begin{proof}
Let $f\in \mathbf{\dot{F}}^{s_1}_{X^{q_1},q_1}(\rn)\cap \mathbf{W\dot{F}}^{1}_{X,1}(\rn).$
For any
$\ld,r,s\in(0,\infty)$ and $x\in \rn$,
let $E_f(\ld, r, s)$ be the same as  in \eqref{keydef}. Let
\begin{align*}
\mathrm{G_1}:&=\sup_{\ld\in(0,\infty)}\lambda\left\|
\left[\int_\rn\mathbf{1}_{E_f(\ld, q_1, s_1)}
(\cdot,y)\,dy\right]^{1/q_1}\right\|
_{X^{q_1}}=\|f\|_{\mathbf{W\dot{F}}^{s_1}_{X^{q_1},q_1}(\rn)}
\end{align*}
and
\begin{align*}
\mathrm{H_1}:&=\sup_{\ld\in(0,\infty)}\lambda
\left\|\int_{\rn}\mathbf{1}_{E_f(\ld,1,1)}
(\cdot,y)\,dy\right\|_{X}
=\|f\|_{\mathbf{W\dot{F}}^{1}_{X,1}(\rn)}.
\end{align*}
By an argument similar to that used in the proof of \eqref{youyong}
with $E_f(\ld,q,\theta)$ and $E_f(\ld,q_1,0)$ replaced, respectively,
by $E_f(\ld,q,s)$ and $E_f(\ld,q_1,s_1)$, we conclude that
\begin{align*}
\left\|\int_{\rn}\mathbf{1}_{E_f(\ld, q,s)}
(\cdot,y)\,dy\right\|_{X}\lesssim\frac{1}{\lambda^q}\mathrm{G_1}^
{q(1-\theta)}\mathrm{H_1}^{q\theta}.
\end{align*}
From this and the fact that $\|f\|_{\mathbf{W\dot{F}}^{s_1}_{X^{q_1},q_1}(\rn)}\leq
\|f\|_{\dot{\mathbf{F}}^{s_1}_{X^{q_1},q_1}(\rn)}$, we deduce that
\begin{align}\label{qianniu}
\|f\|_{\mathbf{W\dot{F}}^s_{X^q,q}(\rn)}^q
&=\sup_{\ld
\in(0,\infty)}\lambda^q\left\|\int_{\rn}\mathbf{1}_{
E_f(\ld,q,s)}(\cdot,y)\,dy\right\|_{X}
\lesssim \mathrm{G}_1^{q(1-\theta)}\mathrm{H}_1^{q\theta}\\\noz
&\sim
\|f\|_{\mathbf{W\dot{F}}^{s_1}_{X^{q_1},q_1}(\rn)}^{q(1-\theta)}
\|f\|_{\mathbf{W\dot{F}}^{1}_{X,1}(\rn)}^{q\theta}
\lesssim
\|f\|_{\mathbf{\dot{F}}^{s_1}_{X^{q_1},q_1}(\rn)}^{q(1-\theta)}
\|f\|_{\mathbf{W\dot{F}}^{1}_{X,1}(\rn)}^{q\theta}\noz.
\end{align}

If $f\in C^2(\rn)$ with $|\nabla f|\in C_\mathrm{c}(\rn)$, then,
applying Theorem \ref{theorem3}(i) with $q:=1$,
we find that
\begin{align}\label{qiannianjiang}
\|f\|_{\mathbf{W\dot{F}}^{1}_{X,1}(\rn)}
\lesssim
\|\,|\nabla f|\,\|_{X}.
\end{align}
This, together with \eqref{qianniu}, further
implies that \eqref{zhongyang} holds true for any $f\in C^2(\rn)$ with $|\nabla f|\in C_\mathrm{c}(\rn)$.

If the centered ball average operators $\{B_r\}_{r\in(0,\infty)}$ are uniformly bounded on $X$ and
$X$ has an absolutely continuous norm,
then, by  Theorem \ref{theorem3}(ii) with $q:=1$, we find that \eqref{qiannianjiang} holds true
for any $f\in \dot{W}^{1,X}(\rn)$.
This, combined with \eqref{qianniu},
further implies that \eqref{zhongyang} holds true for any $f\in \dot{W}^{1,X}(\rn)$.
This finishes the proof of Corollary \ref{corollary1001}.
\end{proof}

\begin{remark}\label{2022721}
\begin{itemize}
\item[{\rm(i)}] In Corollary \ref{corollary1001}, if $X:=L^p(\rn)$ with $p\in[1,\infty)$,
then the conclusions of Corollary \ref{corollary1001} hold true with  $X$ 
replaced by $L^p(\rn)$. Moreover, when   $X:=L^1(\rn)$ and $f\in C^\infty_{\mathrm{c}}(\rn)$,
Corollary \ref{corollary1001} in this case is just \cite[Corollary 5.2]{ref8}.

\item[{\rm(ii)}]
Let $X$ be the same as in Theorem \ref{theorem3.9}(ii).
Then, similarly to Remark \ref{r3.10}(iv), we find that \eqref{zhongyang} holds true
for any $f\in\dot{W}^{1,X}(\rn)$,
which can be proved by a slight modification of the proof of Corollary \ref{corollary1001}
with Theorem \ref{theorem3} replaced by Theorem \ref{theorem3.9}(ii); we omit the details.
\end{itemize}
\end{remark}
\section{Applications}\label{sec:6}
In this section, we apply
Theorems \ref{theorem3.9} and \ref{theorem3},
Corollaries \ref{corollary3.111} and \ref{corollary1001},
respectively, to six concrete examples of ball
Banach function spaces, namely, Morrey spaces
(see Subsection \ref{s6.1} below), mixed-norm Lebesgue
spaces (see Subsection \ref{s6.2} below),
variable Lebesgue spaces (see Subsection
\ref{s6.3} below), weighted Lebesgue spaces
(see Subsection \ref{s6.4} below), Orlicz spaces
(see Subsection \ref{s6.5} below), and Orlicz-slice
spaces (see Subsection \ref{s6.6} below).
\subsection{Morrey Spaces\label{s6.1}}
For any $0<r\le \alpha<\infty$, the \emph
{Morrey space $M_r^\alpha(\rn)$} is defined to be the set of all the
measurable functions $f$ on $\rn$ with the finite semi-norm
$$\|f\|_{M_r^\alpha(\rn)}
:=\sup_{B\in\BB}|B|^{1/\alpha-1/r}\|f\|_{L^r(B)}.
$$
These spaces were introduced in 1938 by Morrey
\cite{Mo} in order to study the regularity of
solutions to partial differential equations. They
have important applications in the theory of
elliptic partial differential equations,
potential theory, and harmonic analysis (see,
for instance, \cite{a15,cf,HMS20,HMS17,HSS18,HS17,JW,sdh20,sdh20II,tyy19}).
As was indicated in \cite[p.\,86]{ref3},
the Morrey space $M_r^\alpha(\rn)$ for any $1\leq r\leq\al<\infty$
is a ball 	Banach function space, but is not
a Banach function space in the terminology of
Bennett and Sharpley \cite{ref4}.
Using Theorems \ref{theorem3.9}(i)
and \ref{theorem3}(i),  we obtain the following conclusions.
\begin{theorem}\label{bb}
Let $1\leq r\leq \al<\infty$ and $q\in(0,\infty)$
satisfy $n(\frac{1}{r}-\frac{1}{q})<1$.
Then, for any $f\in C^2(\rn)$ with
$|\nabla f|\in C_{\mathrm{c}}(\rn)$,
\begin{align}\label{new}
&\sup_{\sub{\ld\in(0,\infty)\\B\in\BB}} \lambda
|B|^{\f 1\al-\f1r} \left[\int_B \left|\left
\{y\in \rn:\ |f(x)-f(y)|>\lambda|x-y|^{
\frac{n}{q}+1}\right\}\right|^{\f rq} \, dx
\right]^{1/r}\\\noz
&\quad\sim\|\,|\nabla f|\,\|_{M_r^\al(\rn)},
\end{align}
where the positive equivalence constants are independent of $f.$ Moreover, for any $f\in C^2(\rn)$ with
$|\nabla f|\in C_{\mathrm{c}}(\rn)$,
\begin{align}\label{new20}
	&\lim_{\ld\to\infty}\sup_{B\in\BB} \lambda
	|B|^{\f 1\al-\f1r} \left[\int_B \left|\left
	\{y\in \rn:\ |f(x)-f(y)|>\lambda|x-y|^{
	\frac{n}{q}+1}\right\}\right|^{\f rq} \, dx
	\right]^{1/r}\\\noz
	&\quad=
	\left[\frac{K(q,n)}{n}\right]^{\f 1 q}\|\,|\nabla f|\,\|_{M_r^\al(\rn)}.
\end{align}
\end{theorem}

\begin{proof}
We prove \eqref{new} by considering
the following
two cases on $r.$

Case $1)$ $r\in(1,\infty).$ In this case,
since $n(\f1r-\f1q)<1,$ it follows that there exists
a
$p\in[1,r)$ satisfying $n(\f1p-\f1q)<1.$
It is known that, for any $1<r\le \alpha<\infty$, the
associate space $X'$ of the Morrey space
$X:=M_r^\alpha(\rn)$ is a block space, on which
the Hardy--Littlewood maximal operator $\cm$ is
bounded (see, for instance, \cite[Theorem 4.1]{ST},
\cite[Theorem 3.1]{ch14}, and \cite[Lemma 5.7]{H15}).
By this and \cite[p.\,86]{ref3},
we find that
$X^{1/p}=M_{r/p}^{\al/p}(\rn)$ is a ball Banach
function space and $\cm$  bounded on
its associate space $(X^{1/p})'$.
Thus, all the
assumptions of Theorem \ref{theorem3.9}(i)
are satisfied for $X:=M_r^\alpha(\rn)$ with $1<
r\leq \al<\infty,$ and  $p\in[1,r)$.
By Theorem \ref{theorem3.9}(i) with $X$ replaced by $M_r^\alpha(\rn)$, we find that both
\eqref{new} and \eqref{new20} hold true. This finishes the proof of
this theorem in this case.

Case $2)$ $r=1.$ In this case, we then have $n(1-\f1q)<1.$
By
the proof of \cite[Theorem 3.1]{ch14},
we find that, when $1\leq
r\leq \al<\infty$ and $\theta\in(0,1)$, for any $f\in (M_{r/\theta}^{\al/\theta}(\rn))',$
$$\|\cm f\|_{[M_{r/\theta}^{\al/\theta}(\rn)]'}
\lesssim\frac{r}{\theta}\|f\|_{[M_{r/\theta}^{\al/\theta}(\rn)]'},$$
where the implicit positive constant depends only
on $n$.
Thus, all the assumptions of Theorem \ref{theorem3}(i) are
satisfied for $X:=M_1^\alpha(\rn)$.
In this case,
by Theorem \ref{theorem3}(i) with $X:=M_1^\alpha(\rn)$, we find that both
\eqref{new} and \eqref{new20} hold true. This then finishes the proof of this case and hence Theorem \ref{bb}.
\end{proof}
\begin{remark}
\begin{itemize}
\item[{\rm(i)}]
Let $r\in[1,\infty)$ and $q=\al=r$. In this case,
Theorem \ref{bb} is just \cite[Theorem 1.1]{ref8}.
\item[{\rm(ii)}]
Let $1\leq r<\alpha<\infty.$
Since the Morrey space does not have an absolutely continuous norm,  it is still unclear  whether or not \eqref{new} holds true for any $f\in \dot{W}^{1,X}(\rn)$ with $X:=M^\al_r(\rn)$.
\end{itemize}
\end{remark}
Using Corollaries \ref{corollary3.111}
and \ref{corollary1001}, we obtain the following conclusions.
\begin{corollary}\label{bb2}	
Let $1\leq r\leq\alpha<\infty,$ $q_1\in[1,\infty],$ and
$\ta\in (0,1)$. Let
$q\in [1, q_1]$ satisfy
$ \frac{1}{q}=\frac{1-\theta}{q_1}+\theta$.
\begin{itemize}
\item[{\rm(i)}] If $q_1\in[1,\infty),$ then there exists a positive constant $C$ such
that, for any $f\in C^2(\rn)$ with $|\nabla f|\in C_{\mathrm{c}}(\rn)$,
\begin{align*}
	&\sup_{\sub{\ld\in(0,\infty)\\B\in\BB}}
	\lambda|B|^{\frac{1}{q\alpha}-\frac{1}{qr}}
	\left[\int_B \left|	\lf\{y\in \rn:\ |f(x)-f(y)|
	>\lambda|x-y|^{\f n{q} +\ta}\r\}\right|^{
		r}\, dx\right]^{1/qr}\\
	&\quad\leq
	C\|f\|_{M^{q_1\alpha}_{q_1r}}^{1-\ta}\|\,|\nabla f|\,\|
	_{M_{r}^{\alpha}(\rn)}^{\ta}.
\end{align*}
\item[{\rm(ii)}] If $q_1=\infty,$ then there exists a positive constant $C$ such
that, for any $f\in C^2(\rn)$ with $|\nabla f|\in C_{\mathrm{c}}(\rn)$,
\begin{align*}
	&\sup_{\sub{\ld\in(0,\infty)\\B\in\BB}}
	\lambda|B|^{\frac{1}{q\alpha}-\frac{1}{qr}}
	\left[\int_B \left|	\lf\{y\in \rn:\ |f(x)-f(y)|
	>\lambda|x-y|^{\f n{q} +\ta}\r\}\right|^{
		r}\, dx\right]^{1/qr}\\
	&\quad\leq
	C\|f\|_{L^\infty(\rn)}^{1-\ta}\|\,|\nabla f|\,\|
	_{M_{r}^{\alpha}(\rn)}^{\ta}.
\end{align*}
\end{itemize}
\end{corollary}
\begin{corollary}\label{bb3}
Let $1\leq r\leq \al <\infty$, $s_1\in (0,1)$,
$q_1\in(1,\infty)$, and $\theta\in(0,1).$
Let $s\in(s_1,1)$ and $q\in(1,r)$ satisfy
$s=(1-\theta) s_1+\theta$ and $\frac{1}{q}=
\frac{1-\theta}{q_1}+\theta.$
Then, for any $f\in C^2(\rn)$ with $|\nabla f|\in C_{\mathrm{c}}(\rn)$,
\eqref{zhongyang} holds
true with $X$ replaced by $M_{r}^{\alpha}(\rn).$
\end{corollary}
\begin{remark}
\begin{itemize}
\item[{\rm(i)}]
We point out that the Gagliardo--Nirenberg
type inequality in the Sobolev--Morrey space
related to the Riesz potential was established
by Sawano et al. in \cite{sw13}.
The Gagliardo--Nirenberg type inequalities
in the Sobolev-Morrey spaces, as given in
Corollaries \ref{bb2} and \ref{bb3}, appear new.
\item[{\rm(ii)}]
Let $1\leq r<\alpha<\infty.$
Since the Morrey space does not have an absolutely continuous norm,  it is still unclear  whether or not Corollaries \ref{bb2} and \ref{bb3} hold true for any $f\in \dot{W}^{1,X}(\rn)$ with $X:=M^\al_r(\rn)$.
\end{itemize}
\end{remark}
\subsection{Mixed-norm Lebesgue Spaces\label{s6.2}}
For a given vector $\vec{r}:=(r_1,\ldots,r_n)
\in(0,\infty]^n$, the \emph{mixed-norm Lebesgue
space $L^{\vec{r}}(\rn)$} is defined to be the
set of all the measurable functions $f$ on
$\rn$ with the finite quasi-norm
$$
\|f\|_{L^{\vec{r}}(\rn)}:=\lf\{\int_{\rr}
\cdots\lf[\int_{\rr}|f(x_1,\ldots,
x_n)|^{r_1}\,dx_1\r]^{\frac{r_2}{r_1}}
\cdots\,dx_n\r\}^{\frac{1}{r_n}},
$$
where the usual modifications are made when $r_i=
\infty$ for some $i\in\{1,\ldots,n\}$.
In the remainder of this subsection,
let $r_-:=\min\{r_1, \ldots , r_n\}.$
The study of mixed-norm Lebesgue spaces
can be traced back to H\"ormander \cite{H1}
and Benedek and Panzone \cite{BP}.
For more studies
on  mixed-norm Lebesgue spaces,
we refer the reader to
\cite{CGN17,HLYY,hlyy19,hy}
for the Hardy space associated with mixed-norm Lebesgue spaces,
to \cite{cgn17bs, gjn17, GN16}
for the Triebel and Besov spaces
associated with mixed-norm Lebesgue spaces,
to \cite{CGN17,CGN19MS}
for the (anisotropic) mixed-norm Lebesgue space,
and to \cite{tn,noss20}
for the mixed Morrey space.

From the definition of $L^{\vec{r}}(\rn)$, we easily deduce that
$L^{\vec{r}}(\rn)$, where $\vec{r}\in(0,\infty)^n$,
is a ball quasi-Banach function space.
But, $L^{\vec{r}}(\rn)$ may not be a quasi-Banach function space
(see, for instance, \cite[Remark 7.20]{ref5}).
When $X:=L^{\vec{r}}(\rn)$, we denote
$\dot{W}^{1,X}(\rn)$ simply by
$\dot{W}^{1,\vec{r}}(\rn)$.
Using Theorem \ref{theorem3.9}(ii), we obtain the
following conclusions.
\begin{theorem}\label{thm-6-5}
Let $\vec{r}:=(r_1,\ldots,r_n)\in(1,\infty)^n$
and $q\in(0,\infty)$ satisfy $n(\frac{1}{
r_{-}}-\f1q)<1$.
Then, for any
$f\in \dot{W}^{1,\vec{r}}(\rn)$,
\begin{align}\label{new2}
\sup_{\ld\in(0,\infty)}\lambda\left
\|\lf|\lf\{y\in\rn:\ |f(\cdot)-f(y)|>
\lambda|\cdot-y|^{\frac{n}{q}+1}
\r\}\r|^{\frac{1}{q}}\right\|_{L^{\vec{r}}
(\rn)}\sim\|\,|\nabla f|\,\|_{L^{\vec{r}}(\rn)},
\end{align}
where the positive equivalence constants are independent of $f.$ Moreover, 	for any
$f\in \dot{W}^{1,\vec{r}}(\rn)$,
\begin{align}\label{dengdeng}
&\lim_{\ld\to \infty}
\lambda\left
\|\lf|\lf\{y\in\rn:\ |f(\cdot)-f(y)|>
\lambda|\cdot-y|^{\frac{n}{q}+1}
\r\}\r|^{\frac{1}{q}}\right\|_{L^{\vec{r}}
(\rn)}\\\noz
&\quad=
\left[\frac{K(q,n)}{q}\right]^{\frac{1}{q}}\|\,|\nabla f|\,\|_{L^{\vec{r}}(\rn)}.
\end{align}
\end{theorem}
\begin{proof}
Since $n(\f1{r_-}-\f1q)<1,$ it follows that there
exists
a $p\in[1,r_-)$ satisfying $n(\f1p-\f1q)<1.$
Using \cite[Lemma 4.1]{gp65}, we find that
$L^{\vec{r}}(\rn)$ has an
absolutely continuous norm.
By the definition of $L^{\vec{r}}(\rn),$ we have
\begin{align*}
\left[L^{\vec{r}}(\rn)\right]^{\frac{1}{p}}=L^{\frac{\vec{r}}{p}}(\rn)
\end{align*}
and hence $[L^{\vec{r}}(\rn)]^{\frac{1}{p}}$ is a ball Banach  function space
(see, for instance, \cite[Remark 2.8(iii)]{HLYY}).
This, together with \cite[Theorems 1 and 2]{BP}, further implies  that
\begin{equation*}
\left(\left[L^{\vec{r}}(\rn)\right]^{\frac{1}{p}}\right)'=L^{\left(\frac{\vec{r}}{p}\right)'}(\rn),
\end{equation*}
where  $(\frac{\vec{r}}{p})':=(r_1^*,\ldots,r_n^*)$ with $\frac{p}{r_i}+\frac{1}{r_i^*}=1$
for any $i\in\{1,\ldots,n\}$.
Moreover,  from \cite[Lemma 3.5]{HLYY},
we infer that the Hardy--Littlewood maximal operator $\cm$ is
bounded on $([L^{\vec{r}}(\rn)]^{\frac{1}{p}})'$.
Thus, all the assumptions of Theorem
\ref{theorem3.9}(ii) are satisfied
for $X:=L^{\vec{r}}(\rn)$.
By Theorem \ref{theorem3.9}(ii) with $X:=L^{\vec{r}}(\rn)$,
we conclude that both \eqref{new2} and \eqref{dengdeng} hold true for any
$f\in \dot{W}^{1,\vec{r}}(\rn)$. This finishes the proof of Theorem \ref{thm-6-5}.
\end{proof}
Using
Corollaries \ref{corollary3.111} and
\ref{corollary1001} and Remarks \ref{r3.10}(iv) and \ref{2022721}(iii), we obtain the
following conclusions whose proofs are similar to
that of Theorem \ref{thm-6-5}; we omit the details here.
\begin{corollary}\label{gnsm1}
Let $\vec{r}:=(r_1,\ldots,r_n)\in(1,\infty)^n,$
 $q_1\in[1,\infty],$ and
$\ta\in (0,1)$. Let
$q\in [1, q_1]$ satisfy
$ \frac{1}{q}=\frac{1-\theta}{q_1}+\theta$.
\begin{itemize}
\item[{\rm(i)}] If $q_1\in[1,\infty),$ then there exists a positive constant
$C$ such that, for any $f\in \dot{W}^{1,\vec{r}}(\rn)$,
\begin{align*}
	&\sup_{\ld\in(0,\infty)}\lambda\left
	\|\left|\lf\{y\in\rn:\ |f(\cdot)-f(y)|>\lambda
	|\cdot-y|^{\f n {q}+\ta}\r\}\right|
	\right\|_{L^{\vec{r}}(\rn)}^{\f1q}\\
	&\quad\leq C
	\|f\|_{L^{q_1\vec{r}}(\rn)}^{1-\ta}\|\,|\nabla f|\,
	\|_{L^{\vec{r}}(\rn)}^{\ta}.
\end{align*}
\item[{\rm(ii)}] If $q_1=\infty,$ then there exists a positive constant
$C$ such that, for any $f\in \dot{W}^{1,\vec{r}}(\rn)$,
\begin{align*}
	&\sup_{\ld\in(0,\infty)}\lambda\left
	\|\left|\lf\{y\in\rn:\ |f(\cdot)-f(y)|>\lambda
	|\cdot-y|^{\f n {q}+\ta}\r\}\right|
	\right\|_{L^{\vec{r}}(\rn)}^{\f1q}\\
	&\quad\leq C
	\|f\|_{L^{\infty}(\rn)}^{1-\ta}\|\,|\nabla f|\,
	\|_{L^{\vec{r}}(\rn)}^{\ta}.
\end{align*}
\end{itemize}
\end{corollary}
\begin{corollary}\label{gnsm2}
Let $\vec{r}:=(r_1,\ldots,r_n)\in(1,\infty)^n$,
$s_1\in (0,1)$, $q_1\in(1,\infty)$,
and $\theta\in(0,1).$ Let $s\in(s_1,1)$
and $q\in(1,r_{-})$ satisfy
$s=(1-\theta) s_1+\theta$ and $ \frac{1}{q}=\frac{1-\theta}{q_1}
+\theta.$
Then, for any $f\in \dot{W}^{1,\vec{r}}(\rn)$, \eqref{zhongyang} holds
true with $X$ replaced by $L^{\vec{r}}(\rn).$
\end{corollary}
\begin{remark}
\begin{itemize}
\item[{\rm(i)}] To the best of our knowledge, the Gagliardo--Nirenberg
type inequalities of
Corollaries \ref{gnsm1} and \ref{gnsm2} on the mixed-norm Sobolev  space are totally new.

\item[{\rm(ii)}] In both Theorem \ref{thm-6-5} and Corollaries \ref{gnsm1} and \ref{gnsm2},
the reason why the case  $r_{-}=1$ was
excluded is that it is still unclear  whether or not \eqref{sjx}
holds true with $X$ replaced by  $L^{\vec{r}}(\rn)$ in this case.
\end{itemize}	
\end{remark}

\subsection{Variable Lebesgue Spaces\label{s6.3}}

Let $r:\ \rn\to(0,\infty)$ be a nonnegative
measurable function. Let
$$
\widetilde{r}_-:=\underset{x\in\rn}{\essinf}\,r(x)\ \text{and}\
\widetilde r_+:=\underset{x\in\rn}{\esssup}\,r(x).
$$
A function $r:\ \rn\to(0,\infty)$ is said to be \emph{globally
log-H\"older continuous} if there exists an
 $r_{\infty}\in\rr$
and a positive constant $C$ such that, for any
$x,y\in\rn$,
$$
|r(x)-r(y)|\le \frac{C}{\log(e+1/|x-y|)}\ \ \text{and}\ \
|r(x)-r_\infty|\le \frac{C}{\log(e+|x|)}.
$$
The \emph{variable Lebesgue space
$L^{r(\cdot)}(\rn)$} associated with the function
$r:\ \rn\to(0,\infty)$ is defined to be the set
of all the measurable functions $f$ on $\rn$ with
the finite quasi-norm
$$
\|f\|_{L^{r(\cdot)}(\rn)}:=\inf\lf\{\lambda
\in(0,\infty):\ \int_\rn\lf[\frac{|f(x)|}
{\lambda}\r]^{r(x)}\,dx\le1\r\}.
$$
It is known that $L^{r(\cdot)}(\rn)$ is a ball quasi-Banach function space
(see, for instance, \cite[Section 7.8]{ref3}). In particular,
if $1<\widetilde r_-\le \widetilde r_+<\infty$,
then $(L^{r(\cdot)}(\rn), \|\cdot\|_{
L^{r(\cdot)}(\rn)})$ is
a Banach function space in the terminology of
Bennett and Sharpley \cite{ref4} and hence
also a ball Banach function space.  For related results on
variable Lebesgue spaces, we refer the reader to
\cite{CUF, CUW, DHR,KR,N1,N2,NS12}.
We first show that the average operators \{$B_s\}_{s\in(0,\infty)}$
are uniformly on $L^{r(\cdot)}(\rn)$.
\begin{lemma}\label{piaoliang}
Let $r:\ \rn\to(0,\infty)$ be globally
log-H\"older continuous and
$1\leq\wt r_{-}\leq \wt r_{+}<\infty$.
Then there exists a positive constant
$C$ such that, for any   $s\in(0,\infty)$ and $f\in L^{r(\cdot)}(\rn)$,
\begin{align}\label{ku}
\|B_{s}(f)\|_{L^{r(\cdot)}(\rn)}
\leq
C\|f\|_{L^{r(\cdot)}(\rn)}.
\end{align}
\end{lemma}
\begin{proof}
Let $m\in\nn$ satisfy that $\int_{\rn}(e+|y|)^{-m}\,dy\leq 1.$
Since $r(\cdot)$ is globally
log-H\"older continuous and
$1\leq\wt r_{-}\leq \wt r_{+}<\infty$,
it follows that $\frac{1}{r(\cdot)}$
is also  globally
log-H\"older continuous (see, for instance, \cite[Proposition 2.3(5)]{CUF}).
From this and \cite[Theorem 4.2.4]{DHHR}, we deduce that there exists a
$\beta\in(0,1)$, depending only on both $m$ and $r(\cdot)$, such that, for any
$s\in(0,\infty)$, $f\in L^{r(\cdot)}(\rn)$ with
$\|f\|_{L^{r(\cdot)}(\rn)}\leq 1,$ and  $x\in\rn,$
\begin{align}\label{chenmo}
\left|\beta B_s(f)(x)\right|^{r(x)}
&=\left[\beta \fint_{B(x,s)}|f(y)|\,dy\right]^{r(x)}\\\noz
&\leq \fint_{B(x,s)}|f(y)|^{r(y)}\,dy+
2^{-1}(e+|x|)^{-m}\\
&\quad+2^{-1}\fint_{B(x,s)}
(e+|y|)^{-m}\,dy.\noz
\end{align}
On the other hand, observe that,
for any $f\in L^{r(\cdot)}(\rn)$ with
$\|f\|_{L^{r(\cdot)}(\rn)}\leq 1,$
$\int_{\rn}|f(x)|^{r(x)}\,dx
\leq1.$
This, together with both the assumption that
$\wt r_{-}\in[1,\infty)$ and \eqref{chenmo},  further implies that,
for any $s\in(0,\infty)$ and $f\in L^{r(\cdot)}(\rn)$ with
$\|f\|_{L^{r(\cdot)}(\rn)}\leq 1,$
\begin{align*}
&\int_{\rn}\left|2^{-1}\beta B_s(f)(x)\right|^{r(x)}\,dx\\\noz
&\quad\leq
2^{-1}\int_{\rn}\left|\beta B_s(f)(x)\right|^{r(x)}\,dx
\leq
2^{-1}\int_{\rn}\fint_{B(x,s)}|f(y)|^{r(y)}\,dy\,dx\\\noz
&\quad\quad+
2^{-2}\int_{\rn}(e+|x|)^{-m}\,dx
+
2^{-2}\int_{\rn}\fint_{B(x,s)}
(e+|y|)^{-m}\,dy\,dx\\\noz
&\quad=
2^{-1}\int_{\rn}|f(x)|^{r(x)}\,dx
+
2^{-1}\int_{\rn}(e+|x|)^{-m}\,dx
\leq 1
\end{align*}
and hence  $\|B_s(f)\|_{L^{r(\cdot)}(\rn)}\leq2\beta^{-1}.$
By this and a scaling argument, we conclude that \eqref{ku} holds true with $C:=2\beta^{-1}$. This finishes
the proof of Lemma \ref{piaoliang}.
\end{proof}
When $X:=L^{\widetilde{r}(\cdot)}(\rn)$, we denote
$\dot{W}^{1,X}(\rn)$ simply by
$\dot{W}^{1,r(\cdot)}(\rn)$.
Using both
Theorems \ref{theorem3.9}(ii) and \ref{theorem3}(ii),
we obtain the following conclusions.
\begin{theorem}\label{1615}
Let $r:\ \rn\to(0,\infty)$ be globally
log-H\"older continuous.
Assume that $1\leq
\wt r_{-}\leq \wt r_{+}<\infty$ and $q\in(0,
\infty)$ satisfy $n(\frac{1}{\wt r_{-}}-
\frac{1}{q})<1.$ Then, for
any $f\in \dot{W}^{1,r(\cdot)}(\rn)$,
\begin{align}\label{xiatian0}
\sup_{\ld\in(0,\infty)}\lambda\left
\|\lf|\lf\{y\in\rn:\ |f(\cdot)-f(y)|>
\lambda|\cdot-y|^{\frac{n}{q}+1}\r\}\r|
^{\frac{1}{q}}\right\|_{L^{r(\cdot)}
(\rn)}\sim\|\,|\nabla f|\,\|_{L^{r(\cdot)}(\rn)},
\end{align}	
where the positive equivalence constants are independent of $f.$
Moreover,  for
any $f\in \dot{W}^{1,r(\cdot)}(\rn)$,
\begin{align}\label{xiatian}
&\lim_{\ld\to\infty}\lambda\left
\|\lf|\lf\{y\in\rn:\ |f(\cdot)-f(y)|>
\lambda|\cdot-y|^{\frac{n}{q}+1}\r\}\r|
^{\frac{1}{q}}\right\|_{L^{r(\cdot)}(\rn)}\\
&\quad=\left[\frac{K(q,n)}{n}\right]^{\frac{1}{q}}\|\,|\nabla f|\,\|_{L^{r(\cdot)}(\rn)}.\noz
\end{align}
\end{theorem}
\begin{proof}
We prove the present theorem by considering the following two cases on both $\widetilde{r}_1$ and $\widetilde{r}_+.$

Case 1) $1<
\wt r_{-}\leq \wt r_{+}<\infty$.
In this case,
since $n(\frac{1}{\wt r_{-}}-
\frac{1}{q})<1,$ it follows that there exists a $p\in(1,\widetilde{r}_-)$ such that $n(\frac{1}{p}-\frac{1}{q})<1$.
By \cite[p.\,73]{CUF}, we conclude that
$L^{r(\cdot)}(\rn)$
has an absolutely continuous norm.
By the definition of $L^{r(\cdot)}(\rn)$,
we find that
\begin{align*}
\left[L^{r(\cdot)}(\rn)\right]^{\frac{1}{p}}=L^{\frac{r(\cdot)}{p}}(\rn).
\end{align*}
Using this and the dual theorem on the variable Lebesgue space (see, for instance, \cite[Theorem 2.80]{CUF}),
we have
\begin{align*}
\left(\left[L^{r(\cdot)}(\rn)\right]^{\frac{1}{p}}\right)'=
L^{\left(\frac{r(\cdot)}{p}\right)'}(\rn),
\end{align*}
where $$\frac{1}{\left(\frac{r(x)}{p}\right)'}+\frac{1}{\frac{r(x)}{p}}=1$$
for any $x\in\rn.$ This, together with
\cite[Theorem 2.80]{CUF}, further implies that
the Hardy--Littlewood maximal operator $\cm$ is bounded
on $([L^{r(\cdot)}(\rn)]^{\frac{1}{p}})'$. Thus, all the assumptions
of Theorem \ref{theorem3.9}(ii) are
satisfied for $X:=L^{r(\cdot)}(\rn)$
and hence, by  Theorem \ref{theorem3.9}(ii), we find that  both $\eqref{xiatian0}$
and \eqref{xiatian} hold true for any
$f\in \dot{W}^{1,r(\cdot)}(\rn)$.

Case 2) $1=
\wt r_{-}\leq \wt r_{+}<\infty$.
In this case, for any given $\theta\in(\frac{1}{2},1),$
 by the proof of \cite[Theorem 4.3.8]{DHHR}, we find that,
 for any $f\in[L^{r(\cdot)/\theta}(\rn)]',$
 $$\|\cm f\|_{[L^{r(\cdot)/\theta}(\rn)]'}
 \lesssim \frac{r^+}{\theta}\|f\|_{[L^{r(\cdot)/\theta}(\rn)]'},$$
 where the implicit positive constant depends only
 on both $n$ and $r(\cdot)$.  Thus, \eqref{sjx} holds true for $X:=L^{r(\cdot)}(\rn)$.
 From this and Lemma \ref{piaoliang}, we deduce that all the assumptions of Theorem
 \ref{theorem3}(ii) are satisfied for
 $X:=L^{r(\cdot)}(\rn)$. Applying Theorem
 \ref{theorem3}(ii) with
 $X:=L^{r(\cdot)}(\rn)$, we conclude that both \eqref{xiatian0} and \eqref{xiatian}
 hold true for any $f\in W^{1,r(\cdot)}(\rn)$. This finishes the proof of Theorem \ref{1615}.
\end{proof}
Using Corollaries \ref{corollary3.111} and \ref{corollary1001},
we obtain the following conclusions  whose proofs are similar to
that of Theorem \ref{1615}; we omit the details here.
\begin{corollary}\label{vsgn1}
Let $r:\ \rn\to(0,\infty)$ be globally
log-H\"older continuous.
Let $1\leq\wt r_{-}\leq \wt r_{+}<\infty,$
$q_1\in[1,\infty],$ and
$\ta\in (0,1)$. Let
$q\in [1, q_1]$ satisfy
$ \frac{1}{q}=\frac{1-\theta}{q_1}+\theta$.
\begin{itemize}
\item[{\rm(i)}] If $q_1\in[1,\infty),$ then there exists a
positive constant $C$ such that,
for any $f\in W^{1,r(\cdot)}(\rn)$,
\begin{align*}
	&\sup_{\ld\in(0,\infty)}\lambda\left\|\,\left|
	\lf\{y\in\rn:\ |f(\cdot)-f(y)|>\lambda|\cdot-y|^{
		\f n {q}+\ta}\r\}\right|\,
	\right\|_{L^{r(\cdot)}(\rn)}^{\f1q}\\
	&\quad\leq C\|
	f\|_{L^{q_1r(\cdot)}(\rn)}^{1-\ta}\|\,|\nabla f|\,\|_{
		L^{r(\cdot)}(\rn)}^{\ta}.
\end{align*}
\item[{\rm(ii)}] If $q_1=\infty,$ then there exists a
positive constant $C$ such that,
for any $f\in W^{1,r(\cdot)}(\rn)$,
\begin{align*}
	&\sup_{\ld\in(0,\infty)}\lambda\left\|\,\left|
	\lf\{y\in\rn:\ |f(\cdot)-f(y)|>\lambda|\cdot-y|^{
		\f n {q}+\ta}\r\}\right|\,
	\right\|_{L^{r(\cdot)}(\rn)}^{\f1q}\\
	&\quad\leq C\|
	f\|_{L^{\infty}(\rn)}^{1-\ta}\|\,|\nabla f|\,\|_{
		L^{r(\cdot)}(\rn)}^{\ta}.
\end{align*}
\end{itemize}
\end{corollary}
\begin{corollary}\label{vsgn2}
Let $r:\ \rn\to(0,\infty)$ be globally log-H\"older
continuous. Let $1\leq \wt r_{-}\leq \wt r_{+}<\infty$,
$s_1\in (0,1)$, $q_1\in(1,\infty)$, and
$\theta\in(0,1).$ Let $s\in(s_1,1)$ and
$q\in(1,\wt r_{-})$ satisfy
$s=(1-\theta) s_1+\theta$ and $\frac{1}{q}
=\frac{1-\theta}{q_1}+\theta.$
Then, for
any $f\in W^{1,r(\cdot)}(\rn)$,
\eqref{zhongyang} holds
true with $X$ replaced by $L^{r(\cdot)}(\rn).$
\end{corollary}
\begin{remark}
We point out that a different Gagliardo--Nirenberg
type inequality was established
in the variable Sobolev spaces related to the Riesz
potential in \cite{kc09,mnss13}. However,
Corollaries \ref{vsgn1} and \ref{vsgn2} appear new.
\end{remark}
\subsection{Weighted Lebesgue Spaces\label{s6.4}}
Recall that, for any given $1\leq r\leq \infty$ and any given
weight $\og$ on $\RR^n$, $L_\og^r(\rn)$ denotes
the weighted Lebesgue space with respect to
the measure $\og(x)\,dx$ on $\rn$ (see Definition  \ref{twl}) and,
when $X:=L^r_\omega(\rn)$, $\dot{W}^{1,X}(\rn)$ is denoted simply by
$\dot{W}^{1,r}_{\omega}(\rn)$ (see, Definition \ref{202276}). It is worth
pointing out that a weighted Lebesgue space with an $A_\infty(\rn)$-weight may
not be a Banach function space; see \cite[Section 7.1]{ref3}.
Using
Theorems \ref{theorem3.9}(ii), we obtain the following conclusions.
\begin{theorem} \label{thm-6-12}
Let $1\leq p\leq r <\infty$, $\omega\in A_{r/p}(\rn)$,
and $q\in(0,\infty)$ satisfy $n(\f 1p-\f1q) <1.$ Then,
for any $f\in \dot{W}^{1,r}_\omega(\rn)$,
\begin{align}\label{quanquan0}
&\sup_{\ld\in(0,\infty)}\lambda\left[\int_{\rn}
\left| \lf\{y\in \rn:\ |f(x)-f(y)|>\lambda|x-y|
^{\frac{n }{q}+1}\r\}\right|^{\f rq} \og(x) \,
dx\right]^{1/r}\\\noz
&\quad\sim\|\,|\nabla f|\,\|_{L^{r}_{\omega}(\rn)},
\end{align}
where the positive equivalence constants are independent of $f.$ Moreover, for any $f\in \dot{W}^{1,r}_\omega(\rn),$
\begin{align}\label{quanquan}
&\lim_{\ld\to\infty}\lambda\left[\int_{\rn}
\left| \lf\{y\in \rn:\ |f(x)-f(y)|>\lambda|x-y|
^{\frac{n }{q}+1}\r\}\right|^{\f rq} \og(x) \,
dx\right]^{1/r}\\\noz
&\quad=
\left[\frac{K(q,n)}{n}\right]^{\frac{1}{q}}\|\,|\nabla f|\,\|_{L^{r}_{\omega}(\rn)}.
\end{align}
\end{theorem}
\begin{proof}
By \cite[Theorem 1.34]{r87}, we find that $L^r_\omega(\rn)$ has an absolutely continuous norm.
Using the definition of $L^r_{\omega}(\rn),$ we know that
\begin{align}\label{1722}
\left[L^{r}_\omega(\rn)\right]^{\f 1 p}
=L^{\frac{r}{p}}_\omega(\rn).
\end{align}
We consider two cases based on the size of $p$. If $p\in[1,r)$, then, from both
\eqref{1722} and \cite[Lemma 4.2]{ss11}, we infer that
\begin{align}\label{1726}
\left(\left[L^{r}_\omega(\rn)\right]^{\f 1 p}\right)'=
L^{(r/p)'}_{\omega^{1-(r/p)'}}(\rn).
\end{align}
By both the assumption that $\omega\in A_{r/p}(\rn)$
and \cite[Proposition 7.1.5(4)]{ref1}, we conclude  that
\begin{equation*}
	\omega^{1-(r/p)'}\in A_{(r/p)'}(\rn).
\end{equation*}
Using this, \eqref{1726}, and Lemma \ref{jidahanshu},
we find that the Hardy--Littlewood maximal operator $\cm$ is bounded on
$([L^{r}_\omega(\rn)]^{\f 1 p})'$. Thus, in this case, all the assumptions
of Theorem \ref{theorem3.9}(ii) are satisfies for
$X:=L^r_\omega(\rn)$.

If $p=r$, applying the conclusion in \cite[p.\,9]{ins}
and using the assumption that $\omega\in A_{r/p}(\rn)$, we obtain
\begin{equation*}
	\left(\left[L^{r}_{\omega}(\rn)\right]^{1/p}\right)'=L^{\infty}_{\omega^{-1}}(\rn).
\end{equation*}
This, combined with both \cite[Theorem 3.1(b)]{AJ}
and \cite[p.\,9]{ins}, further implies  that
$\cm$ is bounded on $([L^{r}_{\omega}(\rn)]^{1/p})'$.
Thus, in this case, all the assumptions of Theorem \ref{theorem3.9}(ii) are satisfies for
$X:=L^r_\omega(\rn)$.

By Theorem
\ref{theorem3.9}(ii) with $X:=L^r_\omega(\rn)$, we conclude that
both \eqref{quanquan} and \eqref{quanquan0}
hold true. This then finishes the proof of Theorem \ref{thm-6-12}.
\end{proof}
By \cite[Lemma 3.15]{dgpyyz},
we find that the centered ball average operators are uniformly bounded on $L^p_{\omega}(\rn)$ with $p\in[1,\infty)$ and $\omega\in A_p(\rn)$. Using this,
Corollaries \ref{corollary3.111}
and \ref{corollary1001},
and Remarks \ref{r3.10}(iv) and \ref{2022721}(iii), we obtain the following conclusions   whose proofs are similar to
that of Theorem \ref{thm-6-12}; we omit the details here.
\begin{corollary}\label{wsgn1}
Assume that $p\in[1,\infty)$, $\omega\in A_{p}(\rn)$,
$q_1\in[1,\infty],$ and
$\ta\in (0,1)$. Let
$q\in [1, q_1]$ satisfy
$ \frac{1}{q}=\frac{1-\theta}{q_1}+\theta$.
\begin{itemize}
\item[{\rm(i)}] If $q_1\in[1,\infty),$ then there exists
a positive constant $C$ such that,
for any $f\in \dot{W}^{1,p}_\omega(\rn)$,
\begin{align*}
	&\sup_{\ld\in(0,\infty)}\lambda\left[\int_{\rn}
	\left|\lf\{y\in\rn:\ |f(x)-f(y)|>\lambda|x-y|
	^{\f n {q}+\ta}\r\}\right|^p
	\omega(x)\,dx\right]^{\f1{pq}}\\
	&\quad\leq C\|
	f\|_{L^{pq_1}_{\omega}(\rn)}^{1-\ta}\|\,|\nabla f|\,\|_{
		L^p_{\omega}(\rn)}^{\ta}.
\end{align*}
\item[{\rm(ii)}] If $q_1=\infty,$ then there exists
a positive constant $C$ such that,
for any $f\in \dot{W}^{1,p}_\omega(\rn)$,
\begin{align*}
	&\sup_{\ld\in(0,\infty)}\lambda\left[\int_{\rn}
	\left|\lf\{y\in\rn:\ |f(x)-f(y)|>\lambda|x-y|
	^{\f n {q}+\ta}\r\}\right|^p
	\omega(x)\,dx\right]^{\f1{pq}}\\
	&\quad\leq C\|
	f\|_{L^{\infty}(\rn)}^{1-\ta}\|\,|\nabla f|\,\|_{
		L^p_{\omega}(\rn)}^{\ta}.
\end{align*}
\end{itemize}
\end{corollary}
\begin{corollary}\label{wsgn2}
Let  $s_1\in (0,1)$, $q_1
\in(1,\infty)$, and $\theta\in(0,1).$
Let $s\in(s_1,1)$ and $q\in(1,q_1)$ satisfy
$s=(1-\theta) s_1+\theta$ and $\frac{1}{q}=
\frac{1-\theta}{q_1}+\theta.$
Assume that $\omega\in A_{p}(\rn)$.
Then,
for any $f\in \dot{W}^{1,p}_\omega(\rn)$, \eqref{zhongyang} holds true with
$X$ replaced by $L^p_{\omega}(\rn).$
\end{corollary}
\begin{remark}
We pointed out that the Gagliardo--Nirenberg
type inequality in the weighted Sobolev space
related to the Riesz potential was obtained
in \cite{dv07,n15}. However, to the best of
our knowledge, the Gagliardo--Nirenberg
type inequalities of both
Corollaries \ref{wsgn1} and \ref{wsgn2} on the weighted Sobolev space are totally new.
\end{remark}
\subsection{Orlicz Spaces\label{s6.5}}
First, we describe briefly some necessary
concepts and facts on  Orlicz spaces.
A non-decreasing function $\Phi:\ [0,\infty)
\ \to\ [0,\infty)$ is called an \emph{Orlicz function}
if $\Phi(0)= 0$, $\Phi(t)>0$ for any $t\in(0,\infty)$, and
$\lim_{t\to\infty}\Phi(t)=\infty$. An Orlicz
function $\Phi$ is said to be of \emph{lower}
(resp., \emph{upper}) \emph{type} $r$ for some
$r\in\rr$ if there exists a positive constant
$C_{(r)}$ such that,
for any $t\in[0,\infty)$ and
$s\in(0,1)$ (resp., $s\in[1,\infty)$),
 $$\Phi(st)\le C_{(r)} s^r
\Phi(t).$$

In the remainder of this subsection, we always assume that
$\Phi:\ [0,\infty)\ \to\ [0,\infty)$
is an Orlicz function with positive lower
type $r_{\Phi}^-$ and positive upper type
$r_{\Phi}^+$. The \emph{Orlicz norm} $\|f\|_{L^\Phi(\rn)}$ of a measurable
function $f$ on $\rn$ is then defined by setting
$$\|f\|_{L^\Phi(\rn)}:=\inf\lf\{\lambda\in
(0,\infty):\ \int_{\rn}\Phi\lf(\frac{|f(x)|}
{\lambda}\r)\,dx\le1\r\}.$$
Accordingly, the \emph{Orlicz space $L^\Phi(\rn)$}
is defined to be the set of all the measurable
functions $f$ on $\rn$ with finite norm $\|f
\|_{L^\Phi(\rn)}$.
It is easy to prove that the Orlicz space $L^\Phi(\rn)$ is a
quasi-Banach function space (see \cite[Section 7.6]{ref3}).
For related results on
Orlicz spaces, we refer the reader to
\cite{NS14,dfmn21,rr02}.
We first prove the following lemma.
\begin{lemma}\label{brorlicz}
Let $\Phi$ be an Orlicz function with positive lower
type $r_{\Phi}^-\in[1,\infty)$ and positive upper type $r_{\Phi}^+$.
Then there exists a positive constant $C$ such that, for any $r\in(0,\infty)$ and any $g\in L^\Phi(\rn),$
\begin{align}\label{jintui}
\|B_r(g)\|_{L^\Phi(\rn)}
\leq
C\|g\|_{L^\Phi(\rn)},
\end{align}
where $B_r$ is the same as in \eqref{pingjun}.
\end{lemma}
\begin{proof}
Let $r\in(0,\infty)$ and $g\in L^\Phi(\rn).$
Since $\Phi$ is of lower type $r_{\Phi}^{-}\in[1,\infty)$, it follows that, for any $0<t_1<t_2<\infty$,
\begin{align*}
\Phi(t_1)\leq
C_{(r_\Phi^-)}\left(\frac{t_1}{t_2}\right)^{r_\Phi^-}\Phi(t_2)\leq
C_{(r_\Phi^-)}\frac{t_1}{t_2}\Phi(t_2)
\end{align*}
and hence
\begin{align*}
\frac{\Phi(t_1)}{t_1}\leq
C_{(r_\Phi^-)}\frac{\Phi(t_2)}{t_2}.
\end{align*}
From this and \cite[Lemma 1.1.1]{kk91},
we deduce that,
for any   $\ld\in(0,\infty)$ and $x\in\rn,$
\begin{align*}
	\Phi\left(\frac{B_r(g)(x)}{\ld}\right)\lesssim \frac{1}{|B(x,r)|}\int_{B(x,r)}\Phi
	\left(\frac{|g(y)|}{\ld}\right)\,dy,
\end{align*}
where the implicit positive  constant
depends only on $\Phi.$
This, together with the Tonelli theorem, further implies that, for any $r\in(0,\infty)$ and $\ld\in(0,\infty),$
\begin{align*}
	\int_{\rn}\Phi\left(\frac{B_r(g)(x)}{\ld}\right)\,dx
	&\lesssim
	\int_{\rn}\frac{1}{|B(x,r)|}\int_{B(x,r)}\Phi
	\left(\frac{|g(y)|}{\ld}\right)\,dy\,dx\\\noz
	&\sim\int_{\rn}\Phi\left(\frac{|g(y)|}{\ld}\right)\,dx
\end{align*}
and hence
\begin{align*}
	\|B_r(g)\|_{L^\Phi(\rn)}\lesssim \|g\|_{L^\Phi(\rn)}.
\end{align*}
This finishes the proof of Lemma \ref{brorlicz}.
\end{proof}

When $X:=L^{\Phi}(\rn)$, we denote
$\dot{W}^{1,X}(\rn)$ simply by
$\dot{W}^{1,\Phi}(\rn)$.
Using both Theorems \ref{theorem3.9}(ii)
and \ref{theorem3}(ii),  we obtain the following conclusions.
\begin{theorem}\label{xiaole}
Let $\Phi$ be an Orlicz function with positive lower
type $r_{\Phi}^-$ and positive upper type $r_{\Phi}^+$.
Let $1\leq r_{\Phi}^-\leq r_{\Phi}^+<\infty$ and $q
\in(0,\infty)$ satisfy $n(\frac{1}{r_{\Phi}^-}-\f1q)<
1.$ Then, for any $f\in \dot{W}^{1,\Phi}(\rn)$,
\begin{align}\label{1752}
\sup_{\ld\in(0,\infty)}\ld\left\|
\left|\left\{y\in\rn:\ |f(\cdot)-f(y)|>\ld|\cdot-y|^{\f nq+1}\right\}\right|^{\f1q}
\right\|_{L^\Phi(\rn)}\sim \|\,|\nabla f|\,\|_{L^\Phi(\rn)},
\end{align}
where  the positive equivalence constants are independent of $f.$ Moreover, for any $f\in \dot{W}^{1,\Phi}(\rn)$,
\begin{align}\label{1751}
&\lim_{\ld\to\infty}\ld\left\|
\left|\left\{y\in\rn:\ |f(\cdot)-f(y)|>\ld|\cdot-y|^{\f nq+1}\right\}\right|^{\f1q}
\right\|_{L^\Phi(\rn)}\\\noz
&\quad=
\left[\frac{K(q,n)}{n}\right]^{\f 1 q}\|\,|\nabla f|\,\|_{L^\Phi(\rn)}.
\end{align}
\end{theorem}
\begin{proof} We  prove
the present theorem by considering the following two cases on both $r_{\Phi}^-$ and $r_{\Phi}^+$.

Case 1) $1< r_{\Phi}^-\leq r_{\Phi}^+<\infty$. In this case,
since  $n(\frac{1}{r_{\Phi}^-}-\f1q)<
1,$ it follows that there exists
a $p\in[1,r_{\Phi}^-)$ such that
$n(\frac{1}{p}-\f1q)<
1.$
Let $\Phi_p(t):=\Phi(t^{\frac{1}{p}})$
for any $t\in[0,\infty)$. By the
definition of $L^\Phi(\rn)$, we have
\begin{align*}
\left[L^\Phi(\rn)\right]^{\f 1 p}=L^{\Phi_{p}}(\rn).
\end{align*}
Moreover, using the proof of \cite[Lemma 4.5]{ZYYW},
\cite[Theorem 1.2.1]{kk91},
and the dual theorem of $L^{\Phi}(\rn)$ (see,
for instance, \cite[Theorem 13]{rr02}),
we further conclude that, if $1<r_{\Phi}^-\leq r_{\Phi}^+
<\infty$ and $p\in[1,r_{\Phi}^-)$,
then $L^{\Phi}(\rn)$ has an absolutely continuous norm
and $\cm$ is bounded on $([L^{\Phi}(\rn)]^{\frac{1}{p}})'$.
Thus, all the assumptions of Theorem
\ref{theorem3.9}(ii) are satisfied
for $X:=L^\Phi(\rn)$ and hence,
by Theorem
\ref{theorem3.9}(ii) with $X:=L^\Phi(\rn)$, we find that both \eqref{1752}
and \eqref{1751} hold true.

Case 2) $1= r_{\Phi}^-\leq r_{\Phi}^+<\infty$.
In this case, by the proof of
\cite[Theorem 1.2.1]{kk91}, we find that, when $1= r_{\Phi}^-\leq r_{\Phi}^+<\infty$ and $\theta\in(0,1),$
for any $f\in([L^{\Phi}(\rn)]^{\f1\theta})',$
$$\|\cm f\|_{([L^{\Phi}(\rn)]^{\f1\theta})'}
\lesssim (3C_{r_{\Phi}^+})^{\frac{3r_{\Phi}^+}{\theta}}
\|f\|_{([L^{\Phi}(\rn)]^{\f1\theta})'},$$
where the implicit positive  constant depends only on $n$. Thus, \eqref{sjx}
holds true for $X:=L^\Phi(\rn).$
From this and Lemma \ref{brorlicz}, we deduce that all the assumptions of Theorem \ref{theorem3}(ii)
are satisfied for $X:=L^\Phi(\rn)$. By Theorem \ref{theorem3}(ii)
for $X:=L^\Phi(\rn)$, we conclude that both \eqref{1752}
and \eqref{1751} hold true for any $f\in \dot{W}^{1,\Phi}(\rn)$.
This finishes the proof of Theorem \ref{xiaole}.
\end{proof}
Using  Corollaries \ref{corollary3.111}
and \ref{corollary1001}, we obtain the following conclusions whose proofs are similar to
that of Theorem \ref{xiaole}; we omit the details here.
\begin{corollary}\label{os1}
Assume that  $\Phi$ is an Orlicz function with positive
lower type $r_{\Phi}^-$ and positive upper type $
r_{\Phi}^+$. Let $1\leq r_{\Phi}^-\leq r_{\Phi}
^+<\infty,$ $q_1\in[1,\infty],$ and
$\ta\in (0,1)$. Let
$q\in [1, q_1]$ satisfy
$ \frac{1}{q}=\frac{1-\theta}{q_1}+\theta$.
\begin{itemize}
\item[{\rm(i)}] If $q_1\in[1,\infty),$ then there exists a positive constant $C$ such that, for any $f\in \dot{W}^{1,\Phi}(\rn)$,
\begin{align*}
	&\sup_{\ld\in(0,\infty)}\lambda\left\|\,\left|
	\lf\{y\in\rn:\ |f(\cdot)-f(y)|>\lambda|\cdot-y|^{\f n {q}
		+\ta}\r\}\right|\,\right\|_{L^\Phi(\rn)}^{\f1q}	\\
	&\quad\leq C\|f\|_{L^{q_1\Phi}(\rn)}^{1-\ta}\|\,|\nabla f|\,
\|_{L^\Phi(\rn)}^{\ta}.
\end{align*}
\item[{\rm(ii)}] If $q_1=\infty,$ then there exists a positive constant $C$
such that, for any $f\in \dot{W}^{1,\Phi}(\rn)$,
\begin{align*}
	&\sup_{\ld\in(0,\infty)}\lambda\left\|\,\left|
	\lf\{y\in\rn:\ |f(\cdot)-f(y)|>\lambda|\cdot-y|^{\f n {q}
		+\ta}\r\}\right|\,\right\|_{L^\Phi(\rn)}^{\f1q}\\
	&\quad\leq C\|f\|_{L^{\infty}(\rn)}^{1-\ta}\|\,|\nabla
	f|\,\|_{L^\Phi(\rn)}^{\ta}.
\end{align*}
\end{itemize}
\end{corollary}
\begin{corollary}\label{os2}
Let $\Phi$ be an Orlicz function with positive
lower type $r_{\Phi}^-$ and positive upper
type $r_{\Phi}^+$. Let $1\leq r_{\Phi}^-\leq r_{\Phi}
^+<\infty$, $s_1\in (0,1)$, $q_1\in(1,\infty)$,
and $\theta\in(0,1).$ Let $s\in(s_1,1)$ and
$q\in(1,r_{\Phi}^-)$ satisfy
$s=(1-\theta) s_1+\theta$ and $\frac{1}{q}=\frac{1-\theta}{q_1}+\theta.$
Then, for any $f\in \dot{W}^{1,\Phi}(\rn)$,
\eqref{zhongyang} holds true
with $X$ replaced by $L^\Phi(\rn).$
\end{corollary}
\begin{remark}
We pointed out that the Gagliardo--Nirenberg type
inequality in the Sobolev--Orlicz space related to
the Riesz potential was obtained in
\cite{kp,KP06,mnss12}. However, to the best
of our knowledge, the Gagliardo--Nirenberg
type inequalities of both
Corollaries \ref{os1} and \ref{os2}
on the Sobolev--Orlicz space are totally new.
\end{remark}
\subsection{Orlicz-Slice Spaces\label{s6.6}}
First, we give the definition of the Orlicz-slice
spaces and describe briefly some related facts.
Throughout this subsection, we assume
that $\Phi: [0,\infty)\to [0,\infty)$
is an Orlicz function with positive
lower type $r_{\Phi}^-$ and positive upper
type $r_{\Phi}^+$. For any given $t,r\in(0,\infty)$,
the \emph{Orlicz-slice space}
$(E_\Phi^r)_t(\rn)$ is defined to be the set of all the
measurable functions $f$ on $\rn$ with the finite
quasi-norm
$$
\|f\|_{(E_\Phi^r)_t(\rn)} :=\lf\{\int_{\rn}
\lf[\frac{\|f\mathbf{1}_{B(x,t)}\|_{L^\Phi(\rn)}}
{\|\mathbf{1}_{B(x,t)}\|_{L^\Phi(\rn)}}\r]
^r\,dx\r\}^{\frac{1}{r}}.
$$
The Orlicz-slice spaces were introduced in
\cite{ZYYW} as a generalization of
both the slice space of	Auscher and Mourgoglou
\cite{AM2014,APA} and the Wiener amalgam space
in \cite{h75,knt,h19}. According to
both \cite[Lemma 2.28]{ZYYW} and \cite[Remark 7.41(i)]{ref5},
the Orlicz-slice space $(E_\Phi^r)_t(\rn)$ is a
ball Banach function space, but in general is not a
Banach function space.
We first prove the following lemma.
\begin{lemma}\label{brslice}
Let $t\in(0,\infty),\,r\in[1,\infty),$ and
$\Phi$ be an Orlicz function with positive lower
type $r_{\Phi}^-\in[1,\infty)$ and positive upper type
$r_{\Phi}^+$. Then there exists a
positive constant $C$ such that,
for any $s\in(0,\infty)$ and $g\in (E_\Phi^r)_t(\rn)$,
\begin{align}\label{k1}
\|B_s(g)\|_{(E_\Phi^r)_t(\rn)}
\leq
C\|g\|_{(E_\Phi^r)_t(\rn)},
\end{align}
where $B_s$ is the same as in \eqref{pingjun} with $r$ replaced by $s.$
\end{lemma}
\begin{proof}
Let $s\in(0,\infty)$ and $g\in (E_\Phi^r)_t(\rn)$.
We consider the following two cases  on both $s$ and $t$.

Case $1)$ $s\in(0,t).$ In this case, by an argument similar to that used in the proof of \eqref{jintui}, we conclude that,
for any $x\in\rn,$
\begin{align*}
	\left\|B_s(g)\mathbf{1}_{B(x,t)}\right\|_{L^\Phi(\rn)}
	\lesssim
	\left\|g\mathbf{1}_{B(x,2t)}\right\|_{L^\Phi(\rn)}.
\end{align*}
From this and the proof of \cite[Theorem 2.20]{ZYYW}, we infer that
\begin{align*}
	\|B_s(g)\|_{(E_\Phi^r)_t(\rn)}
	&=\lf\{\int_{\rn}
	\lf[\frac{\|B_s(g)\mathbf{1}_{B(x,t)}\|_{L^\Phi(\rn)}}
	{\|\mathbf{1}_{B(x,t)}\|_{L^\Phi(\rn)}}\r]
	^r\,dx\r\}^{\frac{1}{r}}\\\noz
	&\lesssim
	\lf\{\int_{\rn}
	\lf[\frac{\|g\mathbf{1}_{B(x,2t)}\|_{L^\Phi(\rn)}}
	{\|\mathbf{1}_{B(x,t)}\|_{L^\Phi(\rn)}}\r]
	^r\,dx\r\}^{\frac{1}{r}}\sim \|g\|_{(E_\Phi^r)_t(\rn)}.
\end{align*}
This proves \eqref{k1} in this case.

Case $2)$ $s\in[t,\infty).$ In this case, by the Fubini theorem, we have, for any $x\in\rn,$
\begin{align*}
	\left\|B_s(g)\mathbf{1}_{B(x,t)}\right\|_{L^\Phi(\rn)}
	&=\left\|\fint_{B(\cdot,s)}\fint_{B(z,t)}|g(z)|\,d\xi
	\,dz\mathbf{1}_{B(x,t)}\right\|_{L^\Phi(\rn)}\\\noz
	&\lesssim
	\left\|\fint_{B(\cdot,2s)}\fint_{B(\xi,t)}|g(z)|\,dz
	\,d\xi\mathbf{1}_{B(x,t)}\right\|_{L^\Phi(\rn)}\\\noz
	&\lesssim
	\left\|\fint_{B(x,4s)}\fint_{B(\xi,t)}|g(z)|\,dz
	\,d\xi\mathbf{1}_{B(x,t)}\right\|_{L^\Phi(\rn)}\\\noz
	&\sim
	\fint_{B(x,4s)}\fint_{B(\xi,t)}|g(z)|\,dz
	\,d\xi\left\|\mathbf{1}_{B(x,t)}\right\|_{L^\Phi(\rn)}.
\end{align*}
From this, the H\"older inequality, and the proof of
\cite[Theorem 2.20]{ZYYW}, we deduce that
\begin{align*}
	\|B_s(g)\|_{(E_\Phi^r)_t(\rn)}
	&=\lf\{\int_{\rn}
	\lf[\frac{\|B_s(g)\mathbf{1}_{B(x,t)}\|_{L^\Phi(\rn)}}
	{\|\mathbf{1}_{B(x,t)}\|_{L^\Phi(\rn)}}\r]
	^r\,dx\r\}^{\frac{1}{r}}\\\noz
	&\lesssim
	\left\{\int_{\rn}
	\fint_{B(x,4s)}
	\left[\fint_{B(\xi,t)}|g(z)|\,dz
	\,\right]^r\,d\xi
	\,dx\right\}^{\frac{1}{r}}\\\noz
	&\sim
	\left\{\int_{\rn}\left[\fint_{B(x,t)}|g(z)|\,dz\right]^r\,dx\right\}^{\frac{1}{r}}
	\lesssim
	\|g\|_{(E_\Phi^r)_t(\rn)}.
\end{align*}
This proves \eqref{k1} in this case,
which then completes the proof of Lemma
\ref{brslice}.
\end{proof}
When $X:=(E_\Phi^r)_t(\rn)$,
we  denote $\dot{W}^{1,X}(\rn)$
simply by $\dot{W}^{1,(E_\Phi^r)_t}(\rn)$.
Using both Theorems \ref{theorem3.9}(ii)
and \ref{theorem3}(ii),  we obtain the following conclusions.
\begin{theorem}\label{kan}
Let $t\in(0,\infty),\,r\in[1,\infty),$ and
$\Phi$ be an Orlicz function with positive lower
type $r_{\Phi}^-$ and positive upper type
$r_{\Phi}^+$.
Let $1\leq r_{\Phi}^-\leq
r_{\Phi}^+<\infty$ and $q\in(0,\infty)$
satisfy $n(\frac{1}{\min\{r_{\Phi}^-,r\}}-\f1q)<1$.
Then, for any $f\in \dot{W}^{1,(E_\Phi^r)_t}(\rn),$
\begin{align}\label{613}
&\sup_{\ld\in(0,\infty)}\lambda\left
\|\lf|\lf\{y\in\rn:\ |f(\cdot)-f(y)|>\lambda|
\cdot-y|^{\frac{n}{q}+1}\r\}\r|^{\frac{1}{q}}
\right\|_{(E_\Phi^r)_t(\rn)}\\
&\quad\sim\|\,|\nabla f|\,\|_{(E_\Phi^r)_t(\rn)},\noz
\end{align}
where the positive equivalence constants are independent of $f$. Moreover, for any $f\in \dot{W}^{1,(E_\Phi^r)_t}(\rn),$
\begin{align}\label{612}
&\sup_{\ld\in(0,\infty)}\lambda\left
\|\lf|\lf\{y\in\rn:\ |f(\cdot)-f(y)|>\lambda|
\cdot-y|^{\frac{n}{q}+1}\r\}\r|^{\frac{1}{q}}
\right\|_{(E_\Phi^r)_t(\rn)}\\\noz
&\quad=\left[\frac{K(q,n)}{n}\right]^{\f 1 q}
\|\,|\nabla f|\,\|_{(E_\Phi^r)_t(\rn)}.
\end{align}
\end{theorem}
\begin{proof}
We  prove the present theorem by considering the following two cases on both $r_{\Phi}^-$ and $r_{\Phi}^+$.

Case 1)  $1< r_{\Phi}^-\leq
r_{\Phi}^+<\infty$.
In this case, by both \cite[Lemma 2.28]{ZYYW}
and \cite[Remark 7.41(i)]{ref5},
we find that the Orlicz-slice space
$(E_\Phi^r)_t(\rn)$ is a ball Banach function space.
Since $n(\frac{1}{\min\{r_{\Phi}^-,r\}}-\f1q)<1$, it follows that there exists a
$p\in[1,\min\{r_{\Phi}^-,r\})$ such that  $n(\frac{1}{p}-\f1q)<1.$
Using  \cite[Lemma 2.31]{ZYYW}, we find that
\begin{align*}
\left[(E_\Phi^r)_t(\rn)\right]^{1/p}=(E_{\Phi_p}^{r/p})_t(\rn).
\end{align*}
Furthermore, by \cite[Theorem 2.26 and Lemmas 4.4 and 4.5]{ZYYW},
we conclude that $(E_\Phi^r)_t(\rn)$
has an absolutely continuous norms and
that the Hardy--Littlewood maximal operator $\cm$ is
bounded on $([(E_\Phi^r)_t(\rn)]^{1/p})'$.
Thus, all the assumptions of Theorem \ref{theorem3.9}(ii) are satisfied
for $X:=(E_\Phi^r)_t(\rn)$
and hence, by Theorem \ref{theorem3.9}(ii) with $X:=(E_\Phi^r)_t(\rn)$, we find that both \eqref{612} and \eqref{613}
hold true  for any $f\in W^{1,(E_\Phi^r)_t}(\rn).$

Case 2)  $1= r_{\Phi}^-\leq
r_{\Phi}^+<\infty$. In this case,  by
the proof of \cite[Theorem 2.20]{ZYYW},
we conclude
that, when
$1= r_{\Phi}^-\leq r_{\Phi}^+<\infty,$ $r\in[1,\infty),$ and $\theta\in(0,1),$
for any $f\in[(E_{\Phi/\theta}^{r/\theta})_t(\rn)]',$
$$\|\cm f\|_{[(E_{\Phi/\theta}^{r/\theta})_t(\rn)]'}
\lesssim \left[(3C_{r_{\Phi}^+})^{\frac{3r_{\Phi}^+}{\theta}}+
\frac{r}{\theta}\right]
\|f\|_{[(E_{\Phi/\theta}^{r/\theta})_t(\rn)]'},$$
where the implicit positive  constant depends only on $n$. Thus, \eqref{sjx}
holds true for $X:=(E_\Phi^r)_t(\rn).$
From this and Lemma \ref{brslice}, we
deduce that  all the assumptions of Theorem
\ref{theorem3}(ii) are satisfied for
$X:=(E_\Phi^r)_t(\rn).$ By
Theorem
\ref{theorem3}(ii) with
$X:=(E_\Phi^r)_t(\rn),$ we conclude that, for any $f\in \dot{W}^{1,(E_\Phi^r)_t}(\rn),$ both
\eqref{612} and \eqref{613} hold true.
\end{proof}
Using both Corollaries \ref{corollary3.111}
and \ref{corollary1001}, we obtain the following conclusions whose proofs are similar to
that of Theorem \ref{kan}; we omit the details here.
\begin{corollary}\label{oss1}
Assume that  $t\in(0,\infty),\,r\in(1,\infty),$ and $\Phi$
be an Orlicz function with positive lower
type $r_{\Phi}^-$ and positive upper type
$r_{\Phi}^+$. Let $1\leq r_{\Phi}^-\leq
r_{\Phi}^+<\infty,$ $q_1\in[1,\infty],$ and
$\ta\in (0,1)$. Let
$q\in [1, q_1]$ satisfy
$ \frac{1}{q}=\frac{1-\theta}{q_1}+\theta$.
\begin{itemize}
\item[{\rm(i)}] If $q_1\in[1,\infty),$ then there exists a positive constant $C$
such that, for any $f\in \dot{W}^{1,(E_\Phi^r)_t}(\rn),$
\begin{align*}
	&\sup_{\ld\in(0,\infty)}\lambda\left\|\left|
	\lf\{y\in\rn:\ |f(x)-f(y)|>\lambda|x-y|^{\f n
		{q}+\ta}\r\}\right|\right\|
	_{(E_\Phi^r)_t(\rn)}^{\f1q}\\
	&\quad\leq C\|f\|
	_{(E_{q_1\Phi}^{q_1r})_t(\rn)}^{1-\ta}\|\,|\nabla f|\,\|
	_{(E_\Phi^r)_t(\rn)}^{\ta}.
\end{align*}
\item[{\rm(ii)}] If $q_1=\infty,$ then there exists a positive constant $C$
such that, for any $f\in \dot{W}^{1,(E_\Phi^r)_t}(\rn),$
\begin{align*}
	&\sup_{\ld\in(0,\infty)}\lambda\left\|\left|
	\lf\{y\in\rn:\ |f(x)-f(y)|>\lambda|x-y|^{\f n
		{q}+\ta}\r\}\right|\right\|
	_{(E_\Phi^r)_t(\rn)}^{\f1q}\\
	&\quad\leq C\|f\|
	_{L^\infty(\rn)}^{1-\ta}\|\,|\nabla f|\,\|
	_{(E_\Phi^r)_t(\rn)}^{\ta}.
\end{align*}
\end{itemize}
\end{corollary}
\begin{corollary}\label{oss2}
Let $t\in(0,\infty),\,r\in(1,\infty),$ and
$\Phi$ be an Orlicz function with
positive lower type $r_{\Phi}^-$
and positive upper type $r_{\Phi}^+$.
Let $1\leq r_{\Phi}^-\leq r_{\Phi}^+<\infty$, $s_1
\in (0,1)$, $q_1\in(1,\infty)$, and
$\theta\in(0,1).$ Let $s\in(s_1,1)$ and $q
\in(1,\min\{r_{\Phi}^-,r\})$ satisfy
$s=(1-\theta) s_1+\theta$ and $\frac{1}{q}
=\frac{1-\theta}{q_1}+\theta.$
Then, for any $f\in \dot{W}^{1,(E_\Phi^r)_t}(\rn),$
\eqref{zhongyang} holds
true with $X$ replaced by $(E_\Phi^r)_t(\rn).$
\end{corollary}

\begin{remark}
To the best of our knowledge, the
Gagliardo--Nirenberg type inequalities of both Corollaries \ref{oss1}
and \ref{oss2} on the
Sobolev--Orlicz-slice space are totally new.
\end{remark}

\noindent\textbf{Acknowledgements}\quad
Dachun Yang would like to thank Professor Ha\"im Brezis and Professor Po-Lam Yung
for  kindly  providing him the reference \cite{ref8}.

\medskip

\noindent\textbf{Data Availability}\quad Data sharing not applicable to this 
article as no datasets were generated or analysed during
the current study.

\bigskip

\smallskip

\noindent Feng Dai

\smallskip

\noindent Department of Mathematical and
Statistical Sciences, University of Alberta
Edmonton, Alberta T6G 2G1, Canada

\smallskip

\noindent {\it E-mail}: \texttt{fdai@ualberta.ca}

\bigskip

\noindent Xiaosheng Lin, Dachun Yang
(Corresponding author),
Wen Yuan and Yangyang Zhang

\smallskip

\noindent Laboratory of Mathematics and Complex Systems
(Ministry of Education of China),
School of Mathematical Sciences, Beijing Normal University,
Beijing 100875, The People's Republic of China

\smallskip

\noindent {\it E-mails}: \texttt{xiaoslin@mail.bnu.edu.cn} (X. Lin)

\noindent\phantom{{\it E-mails:}} \texttt{dcyang@bnu.edu.cn} (D. Yang)

\noindent\phantom{{\it E-mails:}} \texttt{wenyuan@bnu.edu.cn} (W. Yuan)

\noindent\phantom{{\it E-mails:}} \texttt{yangyzhang@mail.bnu.edu.cn} (Y. Zhang)

\end{document}